\documentclass[12pt]{amsart}

\usepackage{etex}
\usepackage[usenames,dvipsnames]{pstricks} 
\usepackage{epsfig}
\usepackage{graphicx,color}
\usepackage{geometry}
\usepackage[all]{xy}
\usepackage{amssymb,amscd}
\usepackage{cite}
\usepackage{fullpage}
\usepackage{marvosym}
\xyoption{poly}
\usepackage{url}
\usepackage{comment}
\usepackage{float}
\usepackage{bm}
\usepackage{multicol}
\usepackage{enumitem}

\usepackage{tikz}
\usepackage{tikz-cd}
\usetikzlibrary{decorations.markings}
\usetikzlibrary{arrows.meta}
\usetikzlibrary{shapes.geometric}
\tikzset{every loop/.style={min distance=15mm,in=0,out=90,looseness=20}}

\newtheorem{introtheorem}{Theorem}

\newtheorem{theorem}{Theorem}[section]
\newtheorem{lemma}[theorem]{Lemma}
\newtheorem{proposition}[theorem]{Proposition}
\newtheorem{corollary}[theorem]{Corollary}
\theoremstyle{definition}
\newtheorem{definition}[theorem]{Definition}

\newtheorem{notation}[theorem]{Notation}
\newtheorem{remark}[theorem]{Remark}

\newtheorem{example}[theorem]{Example}

\newtheorem*{question*}{Question}

\newtheorem*{questions*}{Questions}

\newtheorem*{steps*}{Answer/steps}

\newtheorem*{progress*}{Progress}

\newtheorem*{classification*}{Classification}

\newtheorem*{construction*}{Classification}
\newtheorem*{example*}{Example}

\newtheorem*{remark*}{Remark}
\newtheorem*{remarks*}{Remarks}
\newtheorem*{definition*}{Definition}
\newtheorem{convention}[theorem]{Convention}

\usepackage{calrsfs}
\usepackage{url}
\usepackage{longtable}
\usepackage[OT2, T1]{fontenc}
\usepackage{textcomp}
\usepackage{times}
\usepackage[scaled=0.92]{helvet}

\usepackage[colorlinks=true,
pdfstartview=FitV,
hyperindex=true,
hyperfigures=false]
{hyperref}

\renewcommand{\tilde}{\widetilde}

\newcommand{\X}{\mathcal{X}}

\DeclareMathOperator{\Res}{Res}

\DeclareMathOperator{\GL}{GL}

\DeclareSymbolFont{cyrletters}{OT2}{wncyr}{m}{n}
\DeclareMathSymbol{\Sha}{\mathalpha}{cyrletters}{"58}

\makeatletter

\def\greekbolds#1{%
 \@for\next:=#1\do{%
    \def\X##1;{%
     \expandafter\def\csname V##1\endcsname{\boldsymbol{\csname##1\endcsname}}
     }
   \expandafter\X\next;
  }
}

\greekbolds{alpha,beta,iota,gamma,lambda,nu,eta,Gamma,varsigma,Lambda}

\def\make@bb#1{\expandafter\def
  \csname bb#1\endcsname{{\mathbb{#1}}}\ignorespaces}

\def\make@bbm#1{\expandafter\def
  \csname bb#1\endcsname{{\mathbbm{#1}}}\ignorespaces}

\def\make@bf#1{\expandafter\def\csname bf#1\endcsname{{\bf
      #1}}\ignorespaces} 

\def\make@gr#1{\expandafter\def
  \csname gr#1\endcsname{{\mathfrak{#1}}}\ignorespaces}

\def\make@scr#1{\expandafter\def
  \csname scr#1\endcsname{{\mathscr{#1}}}\ignorespaces}

\def\make@cal#1{\expandafter\def\csname cal#1\endcsname{{\mathcal
      #1}}\ignorespaces} 
\def\do@Letters#1{#1A #1B #1C #1D #1E #1F #1G #1H #1I #1J #1K #1L #1M
                 #1N #1O #1P #1Q #1R #1S #1T #1U #1V #1W #1X #1Y #1Z}
\def\do@letters#1{#1a #1b #1c #1d #1e #1f #1g #1h #1i #1j #1k #1l #1m
                 #1n #1o #1p #1q #1r #1s #1t #1u #1v #1w #1x #1y #1z}
\do@Letters\make@bb   \do@letters\make@bbm
\do@Letters\make@cal  
\do@Letters\make@scr 
\do@Letters\make@bf \do@letters\make@bf   
\do@Letters\make@gr   \do@letters\make@gr
\makeatother

\def\makeop#1{\expandafter\def\csname#1\endcsname
  {\mathop{\rm #1}\nolimits}\ignorespaces}
\makeop{Hom}   \makeop{End}   \makeop{Aut}   \makeop{Isom}  \makeop{Pic} 
\makeop{Gal}   \makeop{ord}   \makeop{Char}  \makeop{Div}   \makeop{Lie} 
\makeop{PGL}   \makeop{Corr}  \makeop{PSL}   \makeop{sgn}   \makeop{Spf}
\makeop{Spec}  \makeop{Tr}    \makeop{Nr}    \makeop{Fr}    \makeop{disc}
\makeop{Proj}  \makeop{supp}  \makeop{ker}   \makeop{im}    \makeop{dom}
\makeop{coker} \makeop{Stab}  \makeop{SO}    \makeop{SL}    \makeop{SL}
\makeop{Cl}    \makeop{cond}  \makeop{Br}    \makeop{inv}   \makeop{rank}
\makeop{id}    \makeop{Fil}   \makeop{Frac}  \makeop{GL}    \makeop{SU}
\makeop{Nrd}   \makeop{Sp}    \makeop{Tr}    \makeop{Trd}   \makeop{diag}
\makeop{Res}   \makeop{ind}   \makeop{depth} \makeop{Tr}    \makeop{st}
\makeop{Ad}    \makeop{Int}   \makeop{tr}    \makeop{Sym}   \makeop{can}
\makeop{length}\makeop{SO}    \makeop{torsion} \makeop{GSp} \makeop{Ker}
\makeop{Adm}   \makeop{Mat}

\DeclareMathSymbol{\twoheadrightarrow} {\mathrel}{AMSa}{"10}

\begin{document}


\title{Twisted Gelfand-Ponomarev modules}

\author{Joseph Muller}
\address{(Muller)National Center for Theoretic Sciences, National Taiwan University, Taipei, Taiwan}
\email{muller@ncts.ntu.edu.tw}

\author{Chia-Fu Yu}
\address{(Yu)Institute of Mathematics, Academia  Sinica, Taipei, Taiwan}
\email{chiafu@math.sinica.edu.tw}

\date{\today}
% \keywords{Gauss problem, Hermitian lattices, abelian varieties, central leaves, mass formula}
% \subjclass{14K10 (14K15, 11G10, 11E41, 16H20)}

\begin{abstract} 
In this expository paper, given a field $K$ and two automorphisms $\sigma, \tau \in \mathrm{Aut}(K)$, we give a self-contained proof of the classification of finite dimensional $K$-vector spaces equipped with two operators $F$ and $V$, respectively $\sigma$-linear and $\tau$-linear, such that $FV = VF = 0$. This classification was originally due to the combined results of Gelfand and Ponomarev (1968), and of Kraft (1975). Following a recent suggestion of Chai (2025), we reworked their classification in light of the notion of Kraft quivers. As an application, we generalize and give an algebraic proof of a theorem by Kottwitz and Rapoport concerning the existence of $F$-crystals. 
\end{abstract}

\maketitle
\setcounter{tocdepth}{2}
\tableofcontents

\section*{Introduction}

Given a triple $(K,\sigma,\tau)$ consisting of a field $K$ together with two automorphisms $\sigma$ and $\tau$, we consider the $K$-algebra $K[F,V]_{\sigma,\tau}$ generated by two indeterminates $F$ and $V$ subject to the following relations 
\begin{align*}
    FV = VF = 0, & & F\lambda = \sigma(\lambda)F, & & V\lambda = \tau(\lambda)V,
\end{align*}
for all $\lambda \in K$. We are concerned with the classification of left $K[F,V]_{\sigma,\tau}$-modules of finite dimension over $K$, which we call \textit{twisted Gelfand-Ponomarev modules}, as this problem has first been considered and entirely solved in the case $(K,\sigma,\tau) = (\mathbb C,\mathrm{id},\mathrm{id})$ by Gelfand and Ponomarev in the second chapter of \cite{gelfand-ponomarev:1968}. As it turns out, their proof did not intrinsically rely on that peculiar choice of $(\mathbb C,\mathrm{id},\mathrm{id})$, so that the generalization to any $(K,\sigma,\tau)$ is not much of a challenge. This might have been first picked up by Kraft who, in the Appendix (``\textit{Anhang}'') of \cite{kraft:1975}, spelled out this ``mild generalization'' \footnote{in loc. cit., ``\textit{eine leichte Verallgemeinerung der Resultate von Gelfand-Ponomarev}'', p.~69.} of Gelfand and Ponomarev's proof to the case of $(K,\sigma,\sigma^{-1})$, with $K$ any field and $\sigma$ any automorphism. Specifically, Kraft's motivation was the case where $K$ is perfect of positive characteristic $p > 0$, and $\sigma: x \mapsto x^p$ is the arithmetic Frobenius. In this case, $K[F,V]_{\sigma,\sigma^{-1}}$ is nothing 
but the Dieudonné ring modulo $p$, and by Dieudonné theory, twisted Gelfand-Ponomarev modules correspond to commutative group schemes over $K$ on which multiplication by $p$ is the trivial map. Examples of such groups include $\mathrm{BT}_1$ groups, among which figure the $p$-torsion subgroups of abelian varieties over $K$. The main body of Kraft's paper \cite{kraft:1975} is precisely concerned with applications of such nature.\\

This expository paper aims at giving a self-contained account of the Appendix of \cite{kraft:1975}, with the slight generalization consisting in replacing $\sigma^{-1}$ by any automorphism $\tau$. In recent unpublished notes \cite{chai:kraft2025} (see Paragraph 6.1), Chai argued in favor of the necessity of such an exposition for different reasons. First, Kraft's original paper \cite{kraft:1975} has been hardly accessible with virtually no online presence for a long time. A scanned version is now hosted on Chai's professional website\footnote{url: \url{https://www2.math.upenn.edu/~chai/papers_pdf/kraft_v1.pdf}} (in German). Then, Kraft's proof actually relies on and assumes familiarity with the ``functorial interpretation'' of Gelfand and Ponomarev's results. This functorial framework is due to Gabriel, and has been allegedly worked out during a seminar in Bonn. While no public written trace of this seminar seems to remain, most papers dealing with classification of modules over string algebras and generalizations thereof make use of Gabriel's approach and give a close account of it; see for instance the Lemma of Section 3 of \cite{ringel:1975}\footnote{The paper \cite{ringel:1975} is claimed to be ``probably the best reference for the method'' by the author of \cite{crawley-boevey:2017}, p.~3292, referring to the ``functorial method''.}, and also \cite{crawley-boevey:2017} and \cite{clannishalgebras:2024} among others. There are also other strategies to the problem of classifying modules of $K[F,V]_{\sigma,\tau}$ (with $\sigma = \tau = \mathrm{id}$); see for instance \cite{nazarova:1975, schroer:2004, oblak:2008, bondarenko:2009, liuzhang:2011}, and Remark 1 of \cite{bondarenko:2009} for a survey of these works.\\

In this paper, for sake of clarity and in order to keep the exposition elementary, we chose not to rely on Gabriel's work and instead we adapted Gelfand and Ponomarev's arguments directly from \cite{gelfand-ponomarev:1968}. For this reason, this paper is not a translation of the Appendix of \cite{kraft:1975}, but a rework of it based on Chai's exposition in \cite{chai:kraft2025}. While our main Theorems \ref{IntroTheoremA} and \ref{IntroTheoremB} are special cases of the wider theory developped in \cite{crawley-boevey:2017} (for the case $\sigma=\tau=\mathrm{id}$) and in \cite{clannishalgebras:2024} (for general $\sigma$ and $\tau$), our exposition aims at being more introductory.\\

Let us sum up the main results. Following \cite{chai:kraft2025}, we introduce the notion of \textit{Kraft quivers}, which are directed graphs with edges labeled by $F$ or $V$ with some additional conditions; see Definition \ref{DefinitionKraftQuiver}. The idea of Kraft quivers was already informally exploited in \cite{gelfand-ponomarev:1968} and \cite{kraft:1975}, and it corresponds to the notions of ``bands'' and ``strings'' in \cite{crawley-boevey:2017} and \cite{clannishalgebras:2024}. To any Kraft quiver $\Gamma$ and to any strict $(\sigma,\tau)$-linear representation $(U,\rho)$ of $\Gamma$ (Definition \ref{DefinitionSigmaLinearRepresentation}), one can associate a finite dimensional $K$-vector space $M(\Gamma,U,\rho)$ (Definition \ref{ModuleAttachedToLinearKraftQuiver}) equipped with two operators $F$ and $V$ which are respectively $\sigma$-linear and $\tau$-linear. In fact, Kraft quivers are designed precisely so that $M(\Gamma,U,\rho)$ turns out to be a twisted Gelfand-Ponomarev module, that is, we have $FV = VF = 0$ on $M(\Gamma,U,\rho)$. The main result states that all the twisted Gelfand-Ponomarev modules arise in this way, in an essentially unique manner. 

\begin{introtheorem}[Theorem \ref{AllTwistedGFModulesComeFromKraftQuiver}]\label{IntroTheoremA}
    Let $M$ be a twisted Gelfand-Ponomarev module. There exist
    \begin{itemize}
        \item a Kraft quiver $\Gamma$ whose connected components are pairwise non-isomorphic, and all of whose circular connected components have no repetitions, and 
        \item a strict $(\sigma,\tau)$-linear representation $(U,\rho)$ on $\Gamma$,
    \end{itemize}
    such that $M \simeq M(\Gamma,U,\rho)$ as $K[F,V]_{\sigma,\tau}$-modules. If $\Gamma'$ and $(U',\rho')$ are another Kraft quiver and strict $(\sigma,\tau)$-linear representation meeting the same conditions, then there exists an isomorphism $\Gamma \simeq \Gamma'$ making $(U,\rho)$ isomorphic to $(U',\rho')$.
\end{introtheorem}

The proof of Theorem \ref{IntroTheoremA} is almost identic to the proof of Gelfand and Ponomarev, which covers the integrality of Chapter 2 of \cite{gelfand-ponomarev:1968} in the case $(K,\sigma,\tau) = (\mathbb C,\mathrm{id},\mathrm{id})$. To this extent, we claim no originality regarding our proof of Theorem \ref{IntroTheoremA}, but we hope that our treatment based on the use of Kraft quivers makes their arguments somewhat clearer and more systematic. Besides, one major upshot of introducing Kraft quivers is that it allows us to classify explicitely the twisted Gelfand-Ponomarev modules which are \textit{indecomposable} (i.e. which are not the direct sum of two proper $K[F,V]_{\sigma,\tau}$-submodules). 

\begin{introtheorem}[Theorem \ref{TheoremClassificationIndecomposableGFModules}]\label{IntroTheoremB}
    Let $M$ be an indecomposable twisted Gelfand-Ponomarev module. Then $M$ is either of the first kind or of the second kind. 
    \begin{itemize}
        \item If $M$ is of the first kind, there exists a connected linear Kraft quiver $\Gamma$, unique up to isomorphism, such that $M \simeq M(\Gamma,\mathbf 1_{\Gamma})$.
        \item If $M$ is of the second kind, there exists a connected circular Kraft quiver with no repetitions $\Gamma$ and a strict indecomposable $(\sigma,\tau)$-linear representation $(U,\rho)$ such that $M \simeq M(\Gamma,U,\rho)$. The datum $(\Gamma,U,\rho)$ is unique up to isomorphism.
    \end{itemize}
\end{introtheorem}

We refer to the body of the text for the notions of \textit{linear} and \textit{circular} Kraft quivers (Definition \ref{DefLinCirc}), as well as for the notions of twisted Gelfand-Ponomarev modules \textit{of the first kind} and \textit{of the second kind} (Definition \ref{DefinitionTwistedGFModuleOfFirstSecondKind}). We point out that the second item of Theorem \ref{IntroTheoremB} (i.e. the case where $M$ is of the second kind) might be made more explicit depending on the choice of $(K,\sigma,\tau)$. Namely, while the proofs of Theorems \ref{IntroTheoremA} and \ref{IntroTheoremB} do not depend on this choice, the classification of strict $(\sigma,\tau)$-linear representations $(U,\rho)$ of a connected circular Kraft quiver depends significantly on $(K,\sigma,\tau)$. In fact, such a representation is entirely determined up to isomorphism by the conjugacy class of its associated \textit{monodromy operator} (Definition \ref{DefinitionMonodromy}), which is a $\Xi$-linear automorphism on a finite dimensional $K$-vector space, where $\Xi \in \mathrm{Aut}(K)$ is a non-empty product of $\sigma$'s and $\tau^{-1}$'s. If $(K,\sigma,\tau) = (\mathbb C,\mathrm{id},\mathrm{id})$, the conjugacy class of such an automorphism is classified by its Jordan normal form, cf. Proposition \ref{PropositionJordanForm}. If $K$ is algebraically closed of positive characteristic $p>0$, $\sigma: x \mapsto x^p$ is the Frobenius morphism and $\tau := \sigma^{-1}$, the Lang-Steinberg theorem states that for any $k>0$, any two $\sigma^k$-linear automorphisms are conjugate, cf. Corollary \ref{ModulesAttachedToCircularKraftQuiverAlgCl}. There might be other specific choices of $(K,\sigma,\tau)$ for which semilinear automorphisms of finite dimensional vector spaces admit a normal form, in which case Theorem \ref{IntroTheoremB} might be made slightly more explicit.\\

In the last part of this paper, we show the existence of a system of stabilized lines in a representation of a doubled circular quiver. This generalizes Theorem 6.1 of Kottwitz--Rapoport~\cite{kottwitzrapoport:2003} (also see Corollary~\ref{TheoremKottwitzRapoport}). The theorem serves as a key step in the proof of the converse of Mazur's inequality (roughly speaking, the necessary condition "the Newton polygon lying above the Hodge polygon" for the existence of $F$-crystals) for $(G,\mu, {\bfK)}$ where $G$ is a reductive group over a local field $F$ of the form $\Res_{F'/F}\GL_n$ or $\Res_{F'/F}\GSp_{2n}$ with finite unramified field extension $F'/F$, $\mu$ is a geometric minuscule cocharacter and $\bfK$ is a hyperspecial  subgroup. The proof presented here is algebraic and direct; it uses the classification result of Theorem \ref{IntroTheoremA}, while the original proof of Kottwitz--Rapoport adapts a clever algebraic geometry argument.  
We remark that such a direct proof was already spelled out by Ringel \cite{ringellecture:2005} in a lecture given in 2005, but for which no public written trace seems to remain. See Remark \ref{RemarkTheorem6.1} for more context. \\

The paper contains five Sections. Section \ref{Section1} introduces the notion of $\sigma$-linear relations and discusses certain structure theorems under more or less restrictive conditions. The contents of Section \ref{Section1} might be of independent interest, and we allowed ourselves to work with infinite dimensional relations for the sake of generality, even though we will only make use of the finite dimensional case later on. The notions developed in Section \ref{Section1} are used only in Section \ref{Section3}, so that the reader might as well start their reading directly with Section \ref{Section2}. In Section \ref{Section2}, we develop the notions of Kraft quivers and $(\sigma,\tau)$-linear representations on them, and explain how to produce twisted Gelfand-Ponomarev modules out of them. Most definitions are taken from \cite{chai:kraft2025}. In Section \ref{Section3}, we study the problem conversely. Namely starting with a twisted Gelfand-Ponomarev module, we explain how to produce a Kraft quiver together with a $(\sigma,\tau)$-linear representation. Most of the constructions and proofs in Section \ref{Section3} are adapted from Chapter 2 of \cite{gelfand-ponomarev:1968}. In Section \ref{Section4}, we check that the constructions detailed in the two previous sections are inverse of each other. All together, this concludes the proofs of Theorems \ref{IntroTheoremA} and \ref{IntroTheoremB}. Lastly, in Section \ref{Section5}, we give the proof of Theorem~\ref{MaintheoremC} and hence a new and algebraic proof of Theorem 6.1 of \cite{kottwitzrapoport:2003}.

%by using the classification result of Theorem \ref{IntroTheoremA}. Such a direct proof was already spelled out by Ringel \cite{ringellecture:2005} in a lecture given in 2005, but for which no public written trace seems to remain. See Remark \ref{RemarkTheorem6.1} for more context. 

\section*{Acknowledgements}

The authors are thankful to Ching-Li Chai for introducing this problem to them during a lecture series in the National Center for Theoretical Sciences, National Taiwan University. The authors are also grateful to William Crawley-Boevey and to Raphael Bennett-Tennenhaus for explaining their results in \cite{clannishalgebras:2024} and providing further references and suggestions. Yu is partially supported by the grants AS-IA-112-M01 and NSTC 114-2115-M-001-001.

\section{$\sigma$-linear relations} \label{Section1}
\subsection{Basic definitions}

In this section, unless stated otherwise, $K$ is any field and $\sigma$ is any automorphism of $K$. If $\sigma_1,\sigma_2$ are two automorphisms of $K$, we write $\sigma_2\sigma_1 = \sigma_2\circ\sigma_1$ for their composition. We fix a $K$-vector space $M$. 

\begin{definition}
A \textit{$\sigma$-linear relation} on $M$ is an additive subgroup $B$ of $M\oplus M$ such that, for every $(x,y) \in B$ and $\lambda \in K$, we have $(\lambda x,\sigma(\lambda)y)\in B$.
\end{definition}

In other words, a $\sigma$-linear relation is a $K$-vector subspace of $M\oplus M^{\sigma}$, where $M^{\sigma} := M\otimes_{K,\sigma} K$. This definition was seemingly first considered in the Appendix of \cite{kraft:1975}, and the special case where $K$ is a perfect field of positive characteristic and $\sigma$ is the Frobenius automorphism was used in the main body of loc. cit. When $\sigma = \mathrm{id}$, we recover the notion of \textit{linear relations} which was, as far as we can tell, originally\footnote{In \cite{maclane:1961}, the author mentions unpublished lecture notes of Puppe where the notion of linear relations was already used in the context of spectral sequences.} introduced independently in \cite{maclane:1961} and in \cite{arens:1961}.

\begin{example} There are two standard examples of $\sigma$-linear relations. 
\begin{itemize}
    \item Let $f$ be a $\sigma$-linear endomorphism of $M$. Its graph 
\begin{equation*}
    \Gamma_f := \{(x,f(x))\,|\, x \in M\}
\end{equation*}
is a $\sigma$-linear relation on $M$. In particular, the identity $\mathrm{id}$ on $M$ defines a linear relation which we denote $\mathbf 1 = \Gamma_{\mathrm{id}}$.
\item Let $N$ be a subspace of $M$. Then
\begin{equation*}
    \theta_N := N \oplus \{0\}
\end{equation*}
is a $\sigma$-linear relation on $M$ for every automorphism $\sigma$ of $K$. When $N = M$, we omit the subscript and write $\theta := \theta_M$. When $N = \{0\}$, we write $\mathbf 0 := \theta_{\{0\}}$ and refer to it as the \textit{trivial relation} on $M$.
\end{itemize}
\end{example}

In analogy with morphisms, if $B$ is a $\sigma$-linear relation on $M$ and $(x,y) \in B$, we will say that $x$ is a \textit{preimage} of $y$ by $B$, or that $y$ is an \textit{image} of $x$ by $B$. We will often write $x \xrightarrow{B} y$ instead of $(x,y) \in B$. When $B$ is clear from the context, we will also write $x \to y$ instead.

\begin{definition}
    For $i=1,2$, let $\sigma_i$ be an automorphism of $K$ and let $B_i$ be a $\sigma_i$-linear relation on $M$. The \textit{product} or \textit{composition} of $B_2$ and $B_1$ is the $\sigma_2\sigma_1$-linear relation given by
    \begin{equation*}
        B_2B_1 = \{(x,z) \in M\oplus M \,|\, \exists y \in M, x\xrightarrow{B_1} y \text{ and } y \xrightarrow{B_2} z\}.
    \end{equation*}
\end{definition}

For example, we have $\mathbf 1 B = B \mathbf 1 = B$ for every $\sigma$-linear relation $B$. It is easy to check that composition is associative, and that we have $\Gamma_g\Gamma_f = \Gamma_{g\circ f}$ whenever $f$ and $g$ are $\sigma_1,\sigma_2$-linear endomorphisms of $M$. If $(B_i)_{1\leq i \leq n}$ are $\sigma_i$-linear relations on $M$, the composition $B_n\cdots B_1$ is the $\sigma_n\cdots\sigma_1$-linear relation consisting of all pairs $(x_0,x_{n}) \in M\oplus M$ such that there exists $x_1,\ldots , x_{n-1} \in M$ satisfying $x_i \xrightarrow{B_i} x_{i+1}$ for all $0\leq i \leq n-1$. We will write 
\begin{equation*}
    x_0 \xrightarrow{B_1} x_1 \xrightarrow{B_2} \cdots \xrightarrow{B_{n-1}} x_{n-1} \xrightarrow{B_n} x_{n}
\end{equation*}
in order to describe such a situation. When $B_1 = \cdots = B_n = B$, we write $B^n$ for the product $BB\cdots B$ and we will drop the superscripts to write $x_0 \to \cdots \to x_{n}$ instead.

\begin{proposition}\label{UsefulProp}
    For $1 \leq i \leq 3$, let $\sigma_i$ be an automorphism of $K$ and let $B_i$ be a $\sigma_i$-linear relation on $M$. If $B_1 \subseteq  B_2$, then $B_3B_1 \subseteq  B_3B_2$ and $B_1B_3 \subseteq  B_2B_3$.
\end{proposition}

Here, $\subseteq $ denotes the inclusion as subgroups of $M\oplus M$. Equivalently, we have $B_1 \subseteq  B_2$ whenever $x \xrightarrow{B_1} y$ implies $x \xrightarrow{B_2} y$. 

\begin{proof}
    Let $(x,z) \in B_3B_1$. There exists some $y \in M$ such that 
    \begin{equation*}
    x \xrightarrow{B_1} y \xrightarrow{B_3} z.
    \end{equation*}
    It follows that $x \xrightarrow{B_2} y \xrightarrow{B_3} z$ so that $(x,z)  \in B_3B_2$ as required. Likewise, one may easily check that $B_1B_3 \subseteq  B_2B_3$.
\end{proof}

\begin{definition}
    Given a $\sigma$-linear relation $B$ on $M$, its \textit{converse} is the $\sigma^{-1}$-linear relation given by
    \begin{equation*}
        B^{\#} := \{(x,y) \in M\oplus M \,|\, (y,x) \in B\}.
    \end{equation*}
\end{definition}

In other words, $x\xrightarrow{B^{\#}} y \iff y \xrightarrow{B} x$. Clearly, we have $(B^{\#})^{\#} = B$ and if $B_1$ and $B_2$ are two $\sigma$-linear relations on $M$, then we have $(B_2B_1)^{\#} = B_1^{\#}B_2^{\#}$. If $f$ is a $\sigma$-linear automorphism of $M$, then $\Gamma_f^{\#} = \Gamma_{f^{-1}}$. 

\begin{definition}
    Let $B$ be a $\sigma$-linear relation. We define four subspaces of $M$ as follows:
    \begin{align*}
        \mathrm{Dom}(B) & := \{x \in M \,|\, \exists y \in M, x \xrightarrow{B} y\}, & \mathrm{Ker}(B) & := \{x \in M \,|\, x \xrightarrow{B} 0\},\\
        \mathrm{Im}(B) & := \{y \in M \,|\, \exists x \in M, x \xrightarrow{B} y\}, & \mathrm{Indet}(B) & := \{y \in M \,|\, 0 \xrightarrow{B} y\}.
    \end{align*}
\end{definition}

We refer to these subspaces respectively as the domain, kernel, image and indeterminacy of $B$. Clearly, we have 
\begin{align*}
    \mathrm{Ker}(B) \subseteq  \mathrm{Dom}(B), & & \mathrm{Indet}(B) \subseteq  \mathrm{Im}(B).
\end{align*}
Moreover, $B$ induces a $\sigma$-linear isomorphism 
\begin{equation*}
    \mathrm{Dom}(B)/\mathrm{Ker}(B) \xrightarrow{\sim} \mathrm{Im}(B)/\mathrm{Indet}(B).
\end{equation*}
Alternatively, one can also define these four subspaces by means of composition of relations. Indeed, we have 
\begin{align*}
    \theta B & = \theta_{\mathrm{Dom}(B)}, & \mathbf{0}B & = \theta_{\mathrm{Ker}(B)},\\
    \theta B^{\#} & = \theta_{\mathrm{Im}(B)}, & \mathbf{0}B^{\#} & = \theta_{\mathrm{Indet}(B)}.
\end{align*}
In particular, it follows that $\mathrm{Dom}(B^{\#}) = \mathrm{Im}(B)$ and $\mathrm{Ker}(B^{\#}) = \mathrm{Indet}(B)$. By replacing $B$ with $B^{\#}$, one obtains similar statements for the image and the indeterminacy of $B^{\#}$. 

\begin{example}
    If $f$ is a $\sigma$-linear endomorphism of $M$, we have 
    \begin{align*}
        \mathrm{Dom}(\Gamma_f) &= M, & \mathrm{Ker}(\Gamma_f) &= \mathrm{Ker}(f),\\
        \mathrm{Im}(\Gamma_f) &= \mathrm{Im}(f), & \mathrm{Indet}(\Gamma_f) &= \{0\}.
    \end{align*}
    If $N$ is a subspace of $M$, we have 
    \begin{align*}
        \mathrm{Dom}(\theta_N) = \mathrm{Ker}(\theta_N) = N, & & \mathrm{Im}(\theta_N) = \mathrm{Indet}(\theta_N) = \{0\}.
    \end{align*}
\end{example}

\begin{definition}\label{DefNonNull}
    A $\sigma$-linear relation $B$ is said to be \textit{null} if $\mathrm{Ker}(B) = \mathrm{Dom}(B)$. Otherwise, we say that $B$ is \textit{non-null}. 
\end{definition}

For instance, the relation $\theta_N$ is null for every subspace $N \subseteq  M$.

\begin{lemma}\label{UsefulLemma}
    For $1 \leq i \leq 2$, let $\sigma_i$ be an automorphism of $K$ and let $B_i$ be a $\sigma_i$-linear relation on $M$. We have 
    \begin{align*}
    \mathrm{Dom}(B_2B_1) & \subseteq  \mathrm{Dom}(B_1), & \mathrm{Ker}(B_1) & \subseteq  \mathrm{Ker}(B_2B_1),\\
    \mathrm{Im}(B_1B_2) & \subseteq  \mathrm{Im}(B_1), & \mathrm{Indet}(B_1) & \subseteq  \mathrm{Indet}(B_1B_2).\\
\end{align*}
\end{lemma}

\begin{proof}
    This follows from Proposition \ref{UsefulProp}. For instance, the inclusion $\theta B_2 \subseteq  \theta$ implies that $\theta B_2B_1 \subseteq  \theta B_1$, which is equivalent to $\mathrm{Dom}(B_2B_1) \subseteq  \mathrm{Dom}(B_1)$.
\end{proof}

\begin{definition}
    Let $B$ be a $\sigma$-linear relation on $M$ and let $N$ be a subspace of $M$. We define the \textit{image} and the \textit{preimage} of $N$ by $B$ as subspaces of $M$ as follows:
    \begin{align*}
        B(N) & := \{y \in M \,|\, \exists x \in N, x \xrightarrow{B} y \},\\
        B^{-1}(N) & := \{x \in M \,|\, \exists y \in N, x \xrightarrow{B} y \}.
    \end{align*}
\end{definition}

Applying the definition with $N = \{0\}$ or $N = M$, we recover the domain, kernel, image and indeterminacy of $B$. Alternatively, one can also recover $B(N)$ and $B^{-1}(N)$ by considering $\theta_NB^{\#}$ and $\theta_NB$ respectively.

\begin{definition}
    Let $M = M_1 \oplus M_2$ be a decomposition of $M$ as a direct sum of two vector subspaces. Let $B_1,B_2$ be two $\sigma$-linear relations respectively on $M_1$ and on $M_2$. The \textit{direct sum} of $B_1$ and $B_2$ is the $\sigma$-linear relation $B = B_1\oplus B_2$ on $M$ defined by 
    \begin{equation*}
        B = \{(x_1+x_2,y_1+y_2) \in (M_1\oplus M_2)\oplus (M_1 \oplus M_2) \,|\, x_1 \xrightarrow{B_1} y_1 \text{ and } x_2 \xrightarrow{B_2} y_2\}.
    \end{equation*}
\end{definition}

The notation $B_1\oplus B_2$ is coherent since the sum $B_1 + B_2$ is direct as a subspace of $M \oplus M^{\sigma}$.

\begin{definition}
    Let $N \subseteq  M$ be a subspace, and let $B$ be a $\sigma$-linear relation on $M$. The \textit{restriction of $B$ to $N$} is the $\sigma$-linear relation 
    \begin{equation*}
        B_{|N} := B\cap (N\oplus N)
    \end{equation*}
    defined on $N$.
\end{definition}

\begin{proposition}
    Let $N \subseteq  M$ be a subspace, and let $B$ be a $\sigma$-linear relation on $M$. We have 
    \begin{align*}
        \mathrm{Im}(B_{|N}) = \mathrm{Im}(B) \cap N, & & \mathrm{Indet}(B_{|N}) = \mathrm{Indet}(B) \cap N.
    \end{align*}
\end{proposition}

The proof is obvious. We point out that similar identities for the domain and the image of $B_{|N}$ do not hold in general.

\subsection{Stable domain, kernel, image and indeterminacy}

We start with a definition.

\begin{definition}\label{DefStableSubspaces}
    Let $B$ be a $\sigma$-linear relation on $M$. We define four subspaces of $M$ as follows:
    \begin{align*}
        \mathrm{Dom}(B^{\infty}) & := \bigcap_{n\geq 1} \mathrm{Dom}(B^n), & \mathrm{Ker}(B^{\infty}) & := \bigcup_{n\geq 1} \mathrm{Ker}(B^n),\\
        \mathrm{Im}(B^{\infty}) & := \bigcap_{n\geq 1} \mathrm{Im}(B^n), & \mathrm{Indet}(B^{\infty}) & := \bigcup_{n\geq 1} \mathrm{Indet}(B^n).
    \end{align*}
\end{definition}

We refer to these subspaces respectively as the stable domain, stable kernel, stable image and stable indeterminacy of $B$. For every $n\geq 1$, we have 
\begin{equation*}
    \mathrm{Ker}(B^n) \subseteq  \mathrm{Ker}(B^{n+1}) \subseteq  \mathrm{Dom}(B^{n+1}) \subseteq  \mathrm{Dom}(B^n),
\end{equation*}
from which it follows that $\mathrm{Ker}(B^{\infty}) \subseteq  \mathrm{Dom}(B^{\infty})$. By replacing $B$ with $B^{\#}$, we also have $\mathrm{Indet}(B^{\infty}) \subseteq  \mathrm{Im}(B^{\infty})$.

\begin{proposition}\label{PropStableKernelIndet}
Let $B$ be a $\sigma$-linear relation on $M$. 
\begin{enumerate}
    \item If $\dim\mathrm{Indet}(B^{\infty}) < \infty$, then 
    \begin{equation*}
        \mathrm{Dom}(B^{\infty}) \cap \mathrm{Indet}(B) \subseteq  \mathrm{Dom}(B^{\infty}) \cap \mathrm{Indet}(B^{\infty}) \subseteq  \mathrm{Ker}(B^{\infty}).
    \end{equation*}
    \item If $\dim \mathrm{Ker}(B^{\infty}) < \infty$, then 
    \begin{equation*}
        \mathrm{Im}(B^{\infty}) \cap \mathrm{Ker}(B) \subseteq  \mathrm{Im}(B^{\infty}) \cap \mathrm{Ker}(B^{\infty}) \subseteq  \mathrm{Indet}(B^{\infty}).
    \end{equation*}
\end{enumerate}
\end{proposition}

\begin{proof}
Point (2) follows from point (1) by replacing $B$ with $B^{\#}$. We prove (1). Since $\mathrm{Indet}(B) \subseteq \mathrm{Indet}(B^{\infty})$, the first inclusion is obvious. In order to prove the second inclusion, we claim that it is enough to prove 
\begin{equation*}
    \mathrm{Dom}(B^{\infty}) \cap \mathrm{Indet}(B) \subseteq  \mathrm{Ker}(B^{\infty}).
\end{equation*}
Indeed, let us assume that this inclusion is already proved for all $\sigma$-linear relations $B$ such that $\dim\mathrm{Indet}(B^{\infty}) < \infty$, and with respect to any automorphism $\sigma$ of $K$. Let $x \in \mathrm{Dom}(B^{\infty})\cap \mathrm{Indet}(B^{\infty})$. Thus we have $x \in \mathrm{Indet}(B^n)$ for some $n\geq 1$. Using our assumption with respect to the $\sigma^n$-linear relation $B^n$, we have 
\begin{equation*}
    \mathrm{Dom}((B^n)^{\infty}) \cap \mathrm{Indet}(B^n) \subseteq  \mathrm{Ker}((B^n)^{\infty}).
\end{equation*}
It is clear that $\mathrm{Dom}((B^n)^{\infty}) = \mathrm{Dom}(B^{\infty})$ and that $\mathrm{Ker}((B^n)^{\infty}) = \mathrm{Ker}(B^{\infty})$, from which it follows that $x \in \mathrm{Ker}(B^{\infty})$ as desired.\\
Now let $x \in \mathrm{Dom}(B^{\infty}) \cap \mathrm{Indet}(B)$. Let $n := \dim \mathrm{Indet}(B^{\infty})$. Since $x \in \mathrm{Dom}(B^n)$, there exists $x_0 = x, x_1, \ldots, x_{n} \in M$ such that 
\begin{equation*}
    x = x_0 \to x_1 \to \cdots \to x_{n}.
\end{equation*} 
Moreover, we know that $0 \to x$. It follows that $x_0,\ldots ,x_{n} \in \mathrm{Indet}(B^{\infty})$. Thus, these vectors can not be linearly independent. Let $m \leq n$ be the smallest integer such that the vectors $x_0,\ldots ,x_m$ are linearly dependent. We can write 
\begin{equation*}
    \mu_mx_m + \cdots + \mu_0x_0 = 0,
\end{equation*}
for some scalars $\mu_0,\ldots ,\mu_m \in K$ that are not all identically equal to $0$. By minimality of $m$, we must have $\mu_m \not = 0$. Consider 
\begin{equation*}
\lambda_i := \sigma^{-i-1}\left(\frac{\mu_i}{\mu_m}\right), \qquad \forall\, 0 \leq i \leq m-1.
\end{equation*}
By additivity and $\sigma$-linearity of $B$, we have 
\begin{align*}
    0 \to x_0 & \to x_1 + \sigma(\lambda_{m-1}) x_0\\
    & \to x_2 + \sigma^2(\lambda_{m-1})x_1 + \sigma(\lambda_{m-2})x_0 \\
    & \to \cdots \\
    & \to x_{m} + \sigma^{m}(\lambda_{m-1})x_{m-1} + \cdots + \sigma^2(\lambda_1)x_1 + \sigma(\lambda_0)x_0.
\end{align*}
By construction, $x_{m} + \sigma^{m}(\lambda_{m-1})x_{m-1} + \cdots + \sigma(\lambda_0)x_0 = 0$. Thus, we have proved that $x_0 \in \mathrm{Ker}(B^{m}) \subseteq  \mathrm{Ker}(B^{\infty})$, which concludes the proof.
\end{proof}

We provide a counterexample to Proposition \ref{PropStableKernelIndet} (1) when $\mathrm{Inder}(B^{\infty})$ has infinite dimension. Let $M = K^{\mathbb N}$ and for $i\geq 1$, denote by $e_i = (\delta_{i,j})_{j\geq 1}$ the canonical basis ($\delta_{i,j}$ is the Kronecker delta). Let $B$ be the $\sigma$-linear relation generated by the pairs $(0,e_1)$ and $(e_i,e_{i+1})$ for all $i\geq 1$. It is clear that $\mathrm{Dom}(B^{\infty}) = \mathrm{Indet}(B^{\infty}) = M$, $\mathrm{Indet}(B) = \mathrm{Span}_K\{e_1\}$ and $\mathrm{Ker}(B^{\infty}) = \{0\}$.

\subsection{The decomposition theorem}

The contents of this section will only be used in the proof of Lemma \ref{LemmaTechnicalSecondKind} and can be skipped at first.\\

In this section, we recall the decomposition theorem for $\sigma$-linear relations $B$. We discuss two statements, a ``weak'' decomposition theorem for the restriction of $B$ to the subspace $\mathrm{Dom}(B^{\infty})$ (version 1) or to $\mathrm{Dom}(B^{\infty}) + \mathrm{Im}(B^{\infty})$ (version 2), which we assume to be finite dimensional, and a ``strong'' decomposition theorem for the whole of $B$, provided that the ambient space $M$ is finite dimensional.\\
The weak version is, as far as we can tell, due to \cite{kraft:1975} (see Zerlegungssatz in Section 7 of the Appendix). However in loc. cit. the author assumes that the whole of $M$ is finite dimensional and restricts the relation to $\mathrm{Dom}(B^{\infty})$ only, so that our versions 1 and 2 offer a mild generalization.\\
The strong version is the main theorem of \cite{towber:1971}, which deals with linear relations but which can easily be generalized to its $\sigma$-linear version, and for which we provide only a statement without proof. \\
Finally, we note that the hypotheses of both versions of the weak decomposition theorems are automatically satisfied when $M$ is finite dimensional.

\begin{theorem}[Weak decomposition - Version 1]\label{WeakDecomposition1}
    Let $B$ be a $\sigma$-linear relation such that $\mathrm{Dom}(B^{\infty})$ and $\mathrm{Indet}(B^{\infty})$ are finite dimensional, and assume that 
    $$B(\mathrm{Dom}(B^{\infty})) \subseteq  \mathrm{Dom}(B^{\infty}) + \mathrm{Inder}(B).$$ 
    There exist a subspace $S \subseteq  \mathrm{Dom}(B^{\infty}) \cap \mathrm{Im}(B^{\infty})$ and a $\sigma$-linear automorphism $T:S\xrightarrow{\sim} S$ such that 
    \begin{enumerate}
        \item $\mathrm{Dom}(B^{\infty}) = S \oplus \mathrm{Ker}(B^{\infty})$,
        \item $B_{|\mathrm{Dom}(B^{\infty})} = B_{|S} \oplus B_{|\mathrm{Ker}(B^{\infty})}$,
        \item $B_{|S} = \Gamma_T$.
    \end{enumerate}
\end{theorem}

\begin{proof}
We decompose the proof in several steps, following \cite{kraft:1975} and \cite{chai:kraft2025}.\\

\textbf{Step 1:} We prove that $B$ induces a $\sigma$-linear automorphism 
\begin{equation*}
    \overline{T}:\mathrm{Dom}(B^{\infty})/\mathrm{Ker}(B^{\infty}) \xrightarrow{\sim} \mathrm{Dom}(B^{\infty})/\mathrm{Ker}(B^{\infty}).
\end{equation*}
Let $x_0 \in \mathrm{Dom}(B^{\infty})$, and let $y$ be an image of $x_0$. By hypothesis, we can write $y=x_1+z$ where $x_1 \in \mathrm{Dom}(B^{\infty})$ and $z\in \mathrm{Indet}(B)$. It follows that $x_0\to x_1$, and we consider the associated coset $[x_1] \in \mathrm{Dom}(B^{\infty})/\mathrm{Ker}(B^{\infty})$. We claim that $[x_1]$ does not depend on the choice of $y$. That is, let $y' \in M$ be another image of $x_0$, and write $y' = x_1' + z'$ as above. It follows that $y-y' = (x_1-x_1') + (z-z') \in \mathrm{Indet}(B)$. By difference, we have 
\begin{equation*}
    x_1-x_1' \in \mathrm{Dom}(B^{\infty})\cap \mathrm{Indet}(B) \subseteq  \mathrm{Ker}(B^{\infty}),
\end{equation*}
where we used Proposition \ref{PropStableKernelIndet}. Thus $[x_1] = [x_1']$, so we have a well-defined map $x_0 \mapsto [x_1]$. It is easy to check that it factors through an injective $\sigma$-linear morphism
\begin{equation*}
    \overline{T}:\mathrm{Dom}(B^{\infty})/\mathrm{Ker}(B^{\infty}) \hookrightarrow \mathrm{Dom}(B^{\infty})/\mathrm{Ker}(B^{\infty}),
\end{equation*}
and surjectivity follows as $\mathrm{Dom}(B^{\infty})$ is finite dimensional by assumption.\\

\textbf{Step 2:} Fix a linear section $\iota:\mathrm{Dom}(B^{\infty})/\mathrm{Ker}(B^{\infty}) \to \mathrm{Dom}(B^{\infty})$ of the natural quotient map. We prove that there exists a linear map $f:\mathrm{Dom}(B^{\infty})/\mathrm{Ker}(B^{\infty}) \to \mathrm{Ker}(B^{\infty})$ such that 
\begin{equation*}
     \forall\, [x] \in \mathrm{Dom}(B^{\infty})/\mathrm{Ker}(B^{\infty}), \qquad \iota([x])+f([x]) \to \iota(\overline{T}([x])) + f(\overline{T}([x])).
\end{equation*}
Fix a $K$-basis $(e_1,\ldots,e_d)$ of $\mathrm{Dom}(B^{\infty})/\mathrm{Ker}(B^{\infty})$, and let $1 \leq i \leq d$. By construction, there exists $z_i \in \mathrm{Ker}(B^{\infty})$ such that $\iota(e_i) \to \iota(\overline{T}(e_i))+z_i$. Define a $\sigma$-linear map $\delta_0:\mathrm{Dom}(B^{\infty})/\mathrm{Ker}(B^{\infty})\to \mathrm{Ker}(B^{\infty})$ by setting $\delta_0(e_i) = z_i$ and extend it by $\sigma$-linearity.\\
Since $\mathrm{Ker}(B^{\infty}) = \bigcup_{n\geq 1}\mathrm{Ker}(B^n)$ is finite dimensional, there must be some $m \geq 1$ such that $\mathrm{Ker}(B^{\infty}) = \mathrm{Ker}(B^m)$. In particular, for every $1 \leq i \leq d$ we can find vectors $z^i_1,\ldots , z^i_{m-1}$ such that 
\begin{equation*}
    \delta_0(e_i)=z_i \to z^i_1 \to \cdots \to z^i_{m-1} \to 0.
\end{equation*}
For $1 \leq j \leq m-1$, define a $\sigma^{j+1}$-linear map 
\[ \delta_{j}:\mathrm{Dom}(B^{\infty})/\mathrm{Ker}(B^{\infty})\to \mathrm{Ker}(B^{\infty}), \quad e_i \mapsto z^i_j \] 
% by setting $\delta_{j}(e_i) = z^i_j$ 
and extend it by $\sigma^{j+1}$-linearity. In doing so, we have constructed a family of maps $(\delta_j)_{j\geq 1}$ from $\mathrm{Dom}(B^{\infty})/\mathrm{Ker}(B^{\infty})$ to $\mathrm{Ker}(B^{\infty})$ such that $\delta_j$ is $\sigma^{j+1}$-linear, and
\begin{align*}
\forall\, [x]\in \mathrm{Dom}(B^{\infty})/\mathrm{Ker}(B^{\infty}), & & \iota([x]) & \to \iota(\overline{T}([x])) + \delta_0([x]),\\
\forall\, 0 \leq j \leq m-1, & & \delta_j([x]) & \to \delta_{j+1}([x]),
\end{align*}
where we put $\delta_{m} := 0$. The desired linear map $f:\mathrm{Dom}(B^{\infty})/\mathrm{Ker}(B^{\infty}) \to \mathrm{Ker}(B^{\infty})$ is given by
\begin{equation*}
    f([x]) := \sum_{j=0}^{m-1} \delta_j(\overline{T}^{-j-1}([x])).
\end{equation*}
Indeed, we have 
\begin{equation*}
    f(\overline{T}([x])) - \delta_0([x]) = \sum_{j=0}^{m-2} \delta_{j+1}(\overline{T}^{-j-1}([x])).
\end{equation*}
By construction, for all $0 \leq j \leq m-2$ we have $\delta_j(\overline{T}^{-j-1}([x])) \to \delta_{j+1}(\overline{T}^{-j-1}([x]))$. Moreover we have $\delta_{m-1}(\overline{T}^{-m}([x])) \to 0$. By addition, it follows that 
\begin{equation*}
    \forall\, [x] \in \mathrm{Dom}(B^{\infty})/\mathrm{Ker}(B^{\infty}), \qquad f([x]) \to f(\overline{T}([x])) - \delta_0([x]).
\end{equation*}
Besides, we also have $\iota([x]) \to \iota(\overline{T}([x])) + \delta_0([x])$. By addition, it follows that 
\begin{equation*}
    \iota([x]) + f([x]) \to \iota(\overline{T}([x])) + f(\overline{T}([x])),
\end{equation*}
as desired.\\

\textbf{Step 3:} We define 
\[ S := (\iota+f)(\mathrm{Dom}(B^{\infty})/\mathrm{Ker}(B^{\infty})) \] 
and prove that it meets all the requirements.\\
It is clear that $S \subseteq  \mathrm{Dom}(B^{\infty})$. Besides, if $x \in S$, it is easy to check that $x = \iota([x]) + f([x])$. Then, by construction of $f$, for every $n\geq 1$ we have 
\begin{equation*}
   \hspace*{-2em} 
   \begin{split}
       \iota(\overline T^{-n}([x])) + f(\overline T^{-n}([x])) & \to \iota(\overline T^{-n+1}([x])) + f(\overline T^{-n+1}([x])) 
       \\ 
       & \to \cdots \to \iota(\overline T^{-1}([x])) + f(\overline T^{-1}([x])) \to x,
   \end{split}
\end{equation*}
from which it follows that $x \in \mathrm{Im}(B^n)$. Thus $S \subseteq  \mathrm{Dom}(B^{\infty}) \cap \mathrm{Im}(B^{\infty})$.\\
We prove (1). If $x \in S$, since $x = \iota([x]) + f([x])$ it is clear that $[x] = 0 \implies x = 0$, thus $S \cap \mathrm{Ker}(B^{\infty}) = \{0\}$. Now let $x \in \mathrm{Dom}(B^{\infty})$, and define $s := \iota([x]) + f([x]) \in S$. Since $[s] = [\iota([x])] = [x]$, it follows that $x = s + (x-s) \in S \oplus \mathrm{Ker}(B^{\infty})$ as desired.\\
We prove (2) and (3). According to (1), the quotient map $x\mapsto [x]$ induces a linear isomorphism $\varphi: S \xrightarrow{\sim} \mathrm{Dom}(B^{\infty})/\mathrm{Ker}(B^{\infty})$. We consider $T := \varphi^{-1}\circ\overline T \circ \varphi$, which is a $\sigma$-linear automorphism of $S$. By construction, we have $x \to T(x)$ for all $x \in S$, thus $\Gamma_T \subseteq  B_{|S}$.\\
Now consider $(x,x') \in B_{|\mathrm{Dom}(B^{\infty})}$. By (1), we can decompose $x = s + z$ and $x' = s'+z'$ where $s,s' \in S$ and $z,z' \in \mathrm{Ker}(B^{\infty})$. Since $T^{-1}(s') \to s'$, we have 
\begin{equation*}
    (s-T^{-1}(s'))+z \to z'.
\end{equation*}
Since $z' \in \mathrm{Ker}(B^{\infty})$, it follows that $(s-T^{-1}(s'))+z \in \mathrm{Ker}(B^{\infty})$, so that $s-T^{-1}(s') = 0$ by (1). Hence $s' = T(s)$, we have $s \to s'$ and thus $z \to z'$. This concludes the proof of (2) and (3).
\end{proof}

\begin{remark}
The hypothesis $B(\mathrm{Dom}(B^{\infty})) \subseteq  \mathrm{Dom}(B^{\infty}) + \mathrm{Inder}(B)$ is important, as the following counterexample illustrates. Let $M$ be freely generated by a vector $f$ and a family of vectors $(e_{i,j})_{1\leq j \leq i}$. Let $B$ be the $\sigma$-linear relation generated by $(f,e_{i,1})$ and by $(e_{i,j},e_{i,j+1})$ for all $i\geq 1$ and $1 \leq j \leq i-1$. 

\begin{tikzcd}[cramped,row sep = small]
&&&& \ldots & & & &\\
&&&& e_{n,1} \arrow[r] & e_{n,2} \arrow[r] & e_{n,3} \arrow[r] & {\cdots} \arrow[r] & e_{n,n} \\
f \arrow[dddrrrr] \arrow[ddrrrr] \arrow[drrrr] \arrow[rrrr] \arrow[urrrr] \arrow[uurrrr] &&&& \cdots & & & & \\
&&&& e_{3,1} \arrow[r] & e_{3,2} \arrow[r] & e_{3,3} & & \\
&&&& e_{2,1} \arrow[r] & e_{2,2} & & & \\
&&&& e_{1,1} & & & &
\end{tikzcd}

In this case, the sable domain $\mathrm{Dom}(B^{\infty}) = K\cdot f$ is one dimensional and the stable kernel and indeterminacy are trivial. However the restriction $B_{|K\cdot f} = \{0\}$ is clearly not a graph. Indeed, the weak decomposition does not apply since $B(\mathrm{Dom}(B^{\infty})) \not \subseteq  \mathrm{Dom}(B^{\infty})$.
\end{remark}

\begin{theorem}[Weak decomposition - Version 2]\label{WeakDecomposition2}
    Let $B$ be a $\sigma$-linear relation such that $M' := \mathrm{Dom}(B^{\infty}) + \mathrm{Im}(B^{\infty})$ is finite dimensional, and assume that there exists some $m\geq 1$ such that $\mathrm{Dom}(B^{\infty}) = \mathrm{Dom}(B^m)$. Let $N = \mathrm{Ker}(B^{\infty}) + \mathrm{Indet}(B^{\infty})$. There exist a subspace $S \subseteq  \mathrm{Dom}(B^{\infty}) \cap \mathrm{Im}(B^{\infty})$ and a $\sigma$-linear automorphism $T:S\xrightarrow{\sim} S$ such that 
    \begin{enumerate}
        \item $M' = S \oplus N$,
        \item $B_{|M'}= B_{|S} \oplus B_{|N}$,
        \item $B_{|S} = \Gamma_T$.
    \end{enumerate}
\end{theorem}

\begin{proof}
Let us check that the hypotheses of Theorem \ref{WeakDecomposition1} are satisfied. Clearly $\mathrm{Dom}(B^{\infty})$ and $\mathrm{Indet}(B^{\infty})$ are finite dimensional. Then, let $y \in B(\mathrm{Dom}(B^{\infty}))$ so that there exists some $x_0 \in \mathrm{Dom}(B^{\infty})$ with $x_0\to y$. We can find vectors $x_1,\ldots, x_{m+1}$ such that 
\begin{equation*}
    x_0 \to x_1 \to x_2 \to \cdots \to x_{m+1}.
\end{equation*}
It follows that $x_1 \in \mathrm{Dom}(B^m) = \mathrm{Dom}(B^{\infty})$ and $y-x_1 \in \mathrm{Indet}(B)$. Thus $B(\mathrm{Dom}(B^{\infty})) \subseteq  \mathrm{Dom}(B^{\infty}) + \mathrm{Indet}(B)$ as desired.\\
Thus, we can apply the Version 1 of the weak decomposition, giving rise to a subspace $S$ together with a $\sigma$-linear automorphism $T$ satisfying (1), (2) and (3) of Theorem \ref{WeakDecomposition1}. We will prove the Version 2 by using the same subspace $S$. Note that (3) already follows from the Version 1.\\
We prove (1). Let $x \in S \cap N$, where $N$ is as in the statement. Write $x = z_1+z_2$ where $z_1 \in \mathrm{Ker}(B^{\infty})$ and $z_2 \in \mathrm{Indet}(B^{\infty})$. It follows that 
\begin{equation*}
    z_2 = x-z_1 \in \mathrm{Dom}(B^{\infty})\cap \mathrm{Indet}(B^{\infty}) \subseteq  \mathrm{Ker}(B^{\infty}),
\end{equation*}
where we used Proposition \ref{PropStableKernelIndet}. Thus $x \in S \cap \mathrm{Ker}(B^{\infty}) = \{0\}$. Therefore we have $S \cap N = \{0\}$.\\
Now let $y \in \mathrm{Im}(B^{\infty})$. We can find vectors $x_0,\ldots ,x_{m-1}$ such that 
\begin{equation*}
    x_0 \to x_1 \to \cdots \to x_{m-1} \to y.
\end{equation*}
It follows that $x_0 \in \mathrm{Dom}(B^{m}) = \mathrm{Dom}(B^{\infty})$. Let $s_0 := \iota([x_0]) + f([x_0]) \in S$, so that $x_0 = s_0 + z_0$ for some $z_0 \in \mathrm{Ker}(B^{\infty})$. We can find vectors $z_1,\ldots, z_r$ for some $r\geq 0$ such that 
\begin{equation*}
    z_0 \to z_1 \to \cdots \to z_r \to 0.
\end{equation*}
Up to adding zeroes at the end of this sequence, we can assume $r\geq m$. Thus, we have $x_0 \xrightarrow{B^m} y$ and $x_0 \xrightarrow{B^m} T^m(s_0)+z_m$. It follows that $t:= y - T^m(s_0) - z_m \in \mathrm{Indet}(B^m) \subseteq  \mathrm{Indet}(B^{\infty}).$ Moreover, we have 
\begin{equation*}
    z_m = y - T^m(s_0) - t \in \mathrm{Im}(B^{\infty}) \cap \mathrm{Ker}(B^{\infty}) \subseteq  \mathrm{Indet}(B^{\infty}),
\end{equation*}
where we used Proposition \ref{PropStableKernelIndet}. All in all, we have $y = T^m(s_0) + t + z_m$, which proves that $\mathrm{Im}(B^{\infty}) = S\oplus \mathrm{Indet}(B^{\infty})$. It follows that 
\begin{equation*}
    M' = \mathrm{Dom}(B^{\infty}) + \mathrm{Im}(B^{\infty}) = S \oplus N.
\end{equation*}
We prove (2). First observe that $B(N) \subseteq  N$. Indeed, let $y \in B(N)$ so that we have $x \to y$ for some $x \in N$. Write $x=x_1 + x_2$ where $x_1 \in \mathrm{Ker}(B^{\infty})$ and $x_2 \in \mathrm{Indet}(B^{\infty})$. There exists some $y_1 \in \mathrm{Ker}(B^{\infty})$ such that $x_1 \to y_1$. It follows that $x_2 \to y - y_1$, which implies that $y-y_1 \in \mathrm{Indet}(B^{\infty})$. Thus $y \in N$.\\
Now let $(x,y) \in B_{|M'}$. We can decompose $x = s + n$ and $y = s' + n'$ where $s,s' \in S$ and $n,n' \in N$. Since $s \to T(s)$, we have $n \to (s'-T(s)) + n'$. Thus, $(s'-T(s)) + n' \in B(N) \subseteq  N$, which implies that $s' = T(s)$. Eventually it follows that $n\to n'$, which concludes the proof. 
\end{proof}

\begin{remark}
    The hypothesis of Theorem \ref{WeakDecomposition2} is satisfied in particular when $\mathrm{Dom}(B^n)$ is finite dimensional for some $n\geq 1$. Indeed, the domains of successive powers of $B$ form a descending chain, which must stabilize if it becomes finite dimensional at some point. A fortiori, Theorem \ref{WeakDecomposition2} is true unconditionally if $M$ has finite dimension, but in which case Theorem \ref{StrongDecomposition} below is stronger.
\end{remark}

For sake of completeness, we recall the main result of \cite{towber:1971} generalized to $\sigma$-linear relations, without proof. Given $n\geq 1$ we define four $\sigma$-linear relations on $K^n$ equipped with its canonical basis $(e_i)_{1\leq i \leq n}$:
\begin{enumerate}
    \item $T(n)$ is generated by $(e_i,e_{i+1})$ for $1 \leq i \leq n-1$,
    \item $T^+(n)$ is generated by $(e_i,e_{i+1})$ for $1 \leq i \leq n-1$ and by $(e_n,0)$,
    \item ${}^+T(n)$ is generated by $(e_i,e_{i+1})$ for $1 \leq i \leq n-1$ and by $(0,e_1)$,
    \item ${}^+T^+(n)$ is generated by $(e_i,e_{i+1})$ for $1 \leq i \leq n-1$, by $(0,e_1)$ and by $(e_n,0)$.
\end{enumerate}

\begin{tikzcd}
e_1 \arrow[r,"T(n)"] & {\cdots} \arrow[r,"T(n)"] & e_n & & & 0 \arrow[r,"{}^+T(n)"] & e_1 \arrow[r,"{}^+T(n)"] & \cdots \arrow[r,"{}^+T(n)"] & e_n & \\
e_1 \arrow[r,"T^+(n)"] & \cdots \arrow[r,"T^+(n)"] & e_n \arrow[r,"T^+(n)"] & 0 & & 0 \arrow[r,"{}^+T^+(n)"] & e_1 \arrow[r,"{}^+T^+(n)"] & \cdots \arrow[r,"{}^+T^+(n)"] & e_n \arrow[r,"{}^+T^+(n)"] & 0
\end{tikzcd}

\begin{theorem}[Strong decomposition]\label{StrongDecomposition}
    Assume that $M$ is finite dimensional, and let $B$ be a $\sigma$-linear relation. There exists two subspaces $S,N \subseteq  M$, a $\sigma$-linear automorphism $T:S \xrightarrow{\sim} S$, a basis $(e_1,\ldots ,e_n)$ of $N$ where $n := \dim(N)$, and a partition $n = n_1 + \cdots + n_r$ for some $r\geq 0$ and some positive integers $n_i$'s, such that
    \begin{enumerate}
        \item $M = S \oplus N$,
        \item we have 
        \begin{equation*}
            B = B_{|S} \oplus \bigoplus_{1\leq i \leq r} B_{|N_i},
        \end{equation*}
        where $N_i := \mathrm{Span}_K\{{e_{n_1+\cdots+n_{i-1}+1}, \ldots , e_{n_1+\cdots+n_{i}}}\}$ for $1\leq i \leq r$, 
        \item $B_{|S} = \Gamma_T$,
        \item for $1 \leq i \leq r$ and after identifying $N_i$ with $K^{n_{i}}$ via the specified basis, the relation $B\cap (N_i\oplus N_i)$ is identified with one of $T(n_i), T^+(n_i), {}^+T(n_i)$ or ${}^+T^+(n_i)$.
    \end{enumerate}
\end{theorem}

The proof of \cite{towber:1971} is by induction on $\dim(M)$, and as such is quite different from the proofs of the weak decomposition Theorems given above. 

\begin{remark}
    We have not discussed uniqueness of the decomposition theorems stated above. In fact, contrary to the more familiar case where $B$ is a graph, the subspace $S$ need not be unique in general. However the ``isomorphism class'' of $(S,\Gamma_T)$ is determined by $B$. This is proved in \cite{towber:1971}, and can be applied in the context of Theorems \ref{WeakDecomposition1} and \ref{WeakDecomposition2} after restricting $B$ to the finite dimensional subspace $\mathrm{Dom}(B^{\infty}) + \mathrm{Indet}(B^{\infty})$ or $\mathrm{Dom}(B^{\infty})+\mathrm{Im}(B^{\infty})$ respectively. Moreover, in Theorem \ref{StrongDecomposition}, the numbers of copies of each of the four kinds of $\sigma$-linear relations $T(n_i), T^+(n_i), {}^+T(n_i)$ and ${}^+T^+(n_i)$ occurring in $B$, are entirely determined by $B$.
\end{remark}

\section{Kraft quivers and $(\sigma,\tau)$-linear representations} \label{Section2}

Kraft quivers have been used in a systematic way in \cite{kraft:1975}, but the idea can already be found in \cite{gelfand-ponomarev:1968} through the notion of ``open/closed sequences'' (see paragraph 4 and 5 of Chapter II). They are also an incarnation of of the notions of ``strings'' and ``bands'', which are used to classify indecomposable modules over (semilinear) string algebras. All the definitions and results in this section are taken from \cite{chai:kraft2025} Sections 3 and 4, although they might sometimes be spelled out differently here. Throughout this section, we consider three independent symbols $F, V$ and $V^{\#}$. Later, these will be interpreted as the relations associated to a $\sigma$-linear operator $F$ and a $\tau$-linear operator $V$. This labelling is inspired by Dieudonné theory, where $F$ is the Frobenius and $V$ is the Verschiebung.

\subsection{Kraft quivers}

Let $\mathcal A$ be a set, called the \textit{alphabet}. 

\begin{definition}
A \textit{finite directed graph labeled by $\mathcal A$} is a tuple $\Gamma = (\mathcal V, \mathcal E, \mathrm{label})$ where:
\begin{itemize}
    \item $\mathcal V$ is a finite set,
    \item $\mathcal E$ is a subset of $\mathcal V \times \mathcal V$,
    \item $\mathrm{label}:\mathcal E \to \mathcal A$ is a function.
\end{itemize}
The set $\mathcal V$ consists of the \textit{vertices} of $\Gamma$, and the set $\mathcal E$ consists of the \textit{directed edges} or \textit{arrows} of $\Gamma$. If $E = (v_1,v_2) \in \mathcal E$, we think of $v_1$ as the \textit{tail} of $E$ and $v_2$ as its \textit{head}. We also define two functions $\mathrm{tail}, \mathrm{head}: \mathcal E \to \mathcal V$ such that $\mathrm{tail}(E) = v_1$ and $\mathrm{head}(E) = v_2$ for all arrows $E = (v_1,v_2) \in \mathcal E$. We say that two vertices $v_1,v_2 \in \mathcal V$ are \textit{connected} if $(v_1,v_2)$ or $(v_2,v_1)$ is in $\mathcal E$.
\end{definition} 

Note that we allow graphs to have loops, that is, arrows of the form $(v,v)$ are allowed. 

\begin{definition}
If $\Gamma_i = (\mathcal V_i, \mathcal E_i, \mathrm{label}_i)$ for $i=1,\ldots , n$ are finite directed graphs labeled by $\mathcal A$, their \textit{disjoint union graph} 
\begin{equation*}
    \Gamma := \Gamma_1 \sqcup \cdots \sqcup \Gamma_n,
\end{equation*}
is the graph $\Gamma = (\mathcal V, \mathcal E, \mathrm{label})$ where 
\begin{itemize}
    \item $\mathcal V = \bigsqcup_{i=1}^n \mathcal V_i$,
    \item $\mathcal E = \bigsqcup_{i=1}^n \mathcal E_i$,
    \item $\mathrm{label}:\mathcal E \to \mathcal A$ is determined by $\mathrm{label}(E) := \mathrm{label}_i(E)$ whenever $E \in \mathcal E_i \subseteq  \mathcal E$.
\end{itemize}
\end{definition}

\begin{definition}
We say that $\Gamma$ is \textit{connected} if the underlying undirected graph is connected. In other words, it is connected if for any vertices $v,v' \in \mathcal V$, there exists a sequence of edges $E_1,\ldots , E_k \in \mathcal E$ and a sequence of vertices $v = v_1, v_2,\ldots , v_{k+1} = v'$ such that, for all $1 \leq i \leq k$, $v_i$ and $v_{i+1}$ are connected via $E_i$.
\end{definition}

Clearly, any finite directed graph $\Gamma$ labeled by $\mathcal A$ decomposes uniquely (up to ordering) as a disjoint union of connected directed graphs labeled by $\mathcal A$. These connected graphs are called the \textit{connected components} of $\Gamma$.

\begin{definition}
    For $i=1,2$, let $\Gamma_i = (\mathcal V_i, \mathcal E_i, \mathrm{label}_i)$ be two finite directed graphs labeled by $\mathcal A$. A \textit{morphism of graphs} $f:\Gamma_1 \to \Gamma_2$ is a map $\mathcal V_1 \to \mathcal V_2$, which is also denoted by $f$, such that
    \begin{align*}
        \forall\, (v_1,v_2) \in \mathcal E_1, & & (f(v_1),f(v_2)) \in \mathcal E_2 \quad \text{ and } \quad \mathrm{label}_1((v_1,v_2)) = \mathrm{label}_2(f(v_1),f(v_2)).
    \end{align*}
\end{definition}
By the first point, $f$ induces a map $\mathcal E_1 \to \mathcal E_2$ by sending an arrow $E = (v_1,v_2)$ to $f(E) := (f(v_1),f(v_2))$. The second condition can then be rewritten as $\mathrm{label}_1(E) = \mathrm{label}_2(f(E))$.\\
If a graph morphism $f$ defines a bijection between the sets of vertices \textit{and} between the sets of arrows, then the inverse map $f^{-1}:\mathcal V_2 \to \mathcal V_1$ induces a graph morphism $f^{-1}:\Gamma_2 \to \Gamma_1$. We say that $f$ is a \textit{graph isomorphism}.

\begin{notation}
    From now on, we will only consider the alphabets $\mathcal A := \{F,V\}$ and $\mathcal A^{\#} := \{F,V^{\#}\}$. If $\mathrm{label}(E) = F$, we say that $E$ is an $F$-arrow (and similarly for $V$ and $V^{\#}$).
\end{notation}

\begin{definition}\label{DefinitionConverseGraph}
    Given a directed diagram $\Gamma= (\mathcal V, \mathcal E, \mathrm{label})$ labeled by $\mathcal A$, its \textit{converse graph} $\Gamma^{\#} = (\mathcal V, \mathcal E^{\#}, \mathrm{label}^{\#})$ is the directed graph labeled by $\mathcal A^{\#} = \{F,V^{\#}\}$ which shares the same set of vertices as $\Gamma$, and such that:
    \begin{itemize}
        \item $\mathcal E^{\#} = \mathcal E^{\#}_F \sqcup \mathcal E^{\#}_{V}$ where 
        \begin{align*}
            \mathcal E^{\#}_F & = \{(v_1,v_2) \,|\, \forall\, E = (v_1,v_2) \in \mathcal E \text{ such that } \mathrm{label}(E) = F\},\\
            \mathcal E^{\#}_{V} &= \{(v_2,v_1) \,|\, \forall\, E = (v_1,v_2) \in \mathcal E \text{ such that } \mathrm{label}(E) = V\},
        \end{align*}
        \item for all $E \in \mathcal E^{\#}$, 
        \begin{equation*}
            \mathrm{label}^{\#}(E) = \begin{cases}
                F & \text{if } E \in \mathcal E^{\#}_F;\\
                V^{\#} & \text{if } E \in \mathcal E^{\#}_V.
            \end{cases}
        \end{equation*}
    \end{itemize}
\end{definition}

In other words, $\Gamma^{\#}$ is obtained from $\Gamma$ by reversing all the $V$-arrows, and relabel them with the symbol $V^{\#}$ instead. Note that $\Gamma$ is entirely determined by its converse $\Gamma^{\#}$.

\begin{definition}\label{DefinitionOppositeGraph}
    Given a directed graph $\Gamma= (\mathcal V, \mathcal E, \mathrm{label})$ labeled by $\mathcal A$, its \textit{opposite graph} $\Gamma^{\mathrm{op}} = (\mathcal V, \mathcal E^{\mathrm{op}}, \mathrm{label}^{\mathrm{op}})$ is given by 
    \begin{itemize}
        \item $\mathcal E^{\mathrm{op}} = \{(v_2,v_1) \,|\, (v_1,v_2) \in \mathcal E\}$,
        \item $\mathrm{label}^{\mathrm{op}}((v_2,v_1)) = \mathrm{label}((v_1,v_2))$ for all $(v_2,v_1) \in \mathcal E^{\mathrm{op}}$.
    \end{itemize}
\end{definition}

In other words, $\Gamma^{\mathrm{op}}$ is obtained from $\Gamma$ by reversing all the arrows without changing the labels.

\begin{definition}\label{DefinitionKraftQuiver}
    A \textit{Kraft quiver} is a finite directed graph $\Gamma= (\mathcal V, \mathcal E, \mathrm{label})$ labeled by $\{F,V\}$ such that: 
    \begin{enumerate}
        \item each vertex $v \in \mathcal V$ is the tail or the head of at most $2$ arrows,
        \item two distinct arrows with the same label have distinct heads and distinct tails, 
        \item the head of any $F$-arrow is not the tail of any $V$-arrow, and the tail of any $F$-arrow is not the head of any $V$-arrow.
    \end{enumerate}
\end{definition}

Clearly, the connected components of a Kraft quiver are again Kraft quivers. Moreover $\Gamma$ is a Kraft quiver if and only if $\Gamma^{\mathrm{op}}$ is.

\begin{figure} 
\centering
\begin{tikzpicture}[
       decoration = {markings,
                     mark=at position .5 with {\arrow{Stealth[length=2mm]}}},
       dot/.style = {circle, fill, inner sep=2.4pt, node contents={},
                     label=#1},
every edge/.style = {draw, postaction=decorate}
                        ]

\node (e1) at (0,4) [dot];
\node (e2) at (2,4) [dot];
\node (e3) at (4,4) [dot];
\node (e4) at (4,2) [dot];
\node (e5) at (6,2) [dot];
\node (e6) at (6,0) [dot];

\node (f1) at (11,4) [dot];
\node (f2) at (10,2) [dot];
\node (f3) at (10,0) [dot];
\node (f4) at (12,0) [dot];
\node (f5) at (12,2) [dot];

\path (e1) edge node[above] {$F$} (e2);
\path (e2) edge node[above] {$F$} (e3);
\path (e4) edge node[right] {$V$} (e3);
\path (e4) edge node[above] {$F$} (e5);
\path (e6) edge node[right] {$V$} (e5);

\path (f1) edge node[above,right] {$F$} (f5);
\path (f4) edge node[right] {$V$} (f5);
\path (f3) edge node[below] {$V$} (f4);
\path (f3) edge node[left] {$F$} (f2);
\path (f1) edge node[above,left] {$V$} (f2);
\end{tikzpicture}
\caption{Two examples of Kraft quivers.}
\label{Figure1}
\end{figure}

\begin{figure} 
\centering
\begin{tikzpicture}[
       decoration = {markings,
                     mark=at position .5 with {\arrow{Stealth[length=2mm]}}},
       dot/.style = {circle, fill, inner sep=2.4pt, node contents={},
                     label=#1},
every edge/.style = {draw, postaction=decorate}
                        ]

\node (e1) at (0,0) [dot];
\node (e2) at (2,0) [dot];
\node (e3) at (4,0) [dot];
\node (e4) at (2,1.5) [dot];

\path (e1) edge node[below] {$F$} (e2);
\path (e3) edge node[below] {$V$} (e2);
\path (e4) edge node[above] {$V$} (e3);
\path (e4) edge node[above] {$F$} (e1);
\path (e4) edge [loop] node[right] {$F$} (e4);
\end{tikzpicture}
\caption{A directed graph labeled by $\{F,V\}$ which is not a Kraft quiver.}
\label{Figure2}
\end{figure}

\begin{figure}
\centering
\begin{tikzpicture}[
       decoration = {markings,
                     mark=at position .5 with {\arrow{Stealth[length=2mm]}}},
       dot/.style = {circle, fill, inner sep=2.4pt, node contents={},
                     label=#1},
every edge/.style = {draw, postaction=decorate}
                        ]

\node (e1) at (0,4) [dot];
\node (e2) at (2,4) [dot];
\node (e3) at (4,4) [dot];
\node (e4) at (4,2) [dot];
\node (e5) at (6,2) [dot];
\node (e6) at (6,0) [dot];

\node (f1) at (11,4) [dot];
\node (f2) at (10,2) [dot];
\node (f3) at (10,0) [dot];
\node (f4) at (12,0) [dot];
\node (f5) at (12,2) [dot];

\path (e1) edge node[above] {$F$} (e2);
\path (e2) edge node[above] {$F$} (e3);
\path (e3) edge node[right] {$V^{\#}$} (e4);
\path (e4) edge node[above] {$F$} (e5);
\path (e5) edge node[right] {$V^{\#}$} (e6);

\path (f1) edge node[above,right] {$F$} (f5);
\path (f5) edge node[right] {$V^{\#}$} (f4);
\path (f4) edge node[below] {$V^{\#}$} (f3);
\path (f3) edge node[left] {$F$} (f2);
\path (f2) edge node[above,left] {$V^{\#}$} (f1);
\end{tikzpicture}
\caption{The converse graphs of the Kraft quivers of Figure \ref{Figure1}.}
\label{Figure3}
\end{figure}

Connected Kraft quivers fall into two disjoint categories, the linear ones and the circular ones. 

\begin{definition}\label{DefLinCirc}
    Let $\Gamma= (\mathcal V, \mathcal E, \mathrm{label})$ be a Kraft quiver, and write $n = \#\mathcal V$. 
    \begin{enumerate}
        \item We say that $\Gamma$ is a \textit{connected linear Kraft quiver} if there exists an ordering $v_1,\ldots , v_n$ of its vertices such that 
        \begin{equation*}
            \mathcal E^{\#} = \{(v_i,v_{i+1}), 1 \leq i \leq n-1\}.
        \end{equation*}
        \item We say that $\Gamma$ is a \textit{connected circular Kraft quiver} if there exists an ordering $v_1,\ldots , v_n$ of its vertices such that 
        \begin{equation*}
            \mathcal E^{\#} = \{(v_i,v_{i+1}), 1 \leq i \leq n\},
        \end{equation*}
        where $v_{n+1} := v_1$.
    \end{enumerate}
    We say that a Kraft quiver $\Gamma$ is \textit{linear} (resp.~\textit{circular}) if all its connected components are connected linear (resp.~connected circular) Kraft quivers.
\end{definition}

\begin{remark}
    A Kraft quiver is both linear and circular if and only if it is empty, that is, its set of vertices is empty.\\
    A Kraft quiver with a single vertex is trivially connected. It is linear if it contains no edge, and it is circular if it contains a loop. 
\end{remark}

\begin{lemma}
Any non-empty connected Kraft quiver is either linear or circular.    
\end{lemma}

\begin{proof}
    Let $\Gamma = (\mathcal V, \mathcal E, \mathrm{label})$ be a non-empty connected Kraft quiver. We may assume that $n := \#\mathcal V \geq 2$. The connectedness of $\Gamma$ forces every vertex to be the tail or the head of at least one arrow, and $\Gamma$ can not contain any loop of the form $(v,v)$ for some $v\in \mathcal V$.\\
    
    \textbf{Case 1:} there exists $v_1 \in \mathcal V$ which is connected to exactly one single arrow in $\Gamma$. Let us denote the corresponding arrow in $\mathcal E^{\#}$ by $E_1$, so that we have $E_1 = (v_1,v_2)$ or $(v_2,v_1)$ for some $v_2 \in \mathcal V$. Up to replacing $\Gamma$ with $\Gamma^{\mathrm{op}}$, we may assume that $E_1 = (v_1,v_2)$. Assume that there are pairwise distinct vertices $v_1,\ldots ,v_{j+1} \in \mathcal V$ for some $j\geq 1$ such that $E_i := (v_i,v_{i+1}) \in \mathcal E^{\#}$ for all $1 \leq i \leq j$. By point (1) of Definition \ref{DefinitionKraftQuiver}, the vertices $v_1,\ldots , v_j$ can not be the head nor the tail of any other arrow besides $E_1,\ldots , E_j$. Assume that $j<n-1$. If $v_{j+1}$ was the head or the tail exactly one single arrow, then $v_1,\ldots ,v_{j+1}$ would form a proper connected component in $\Gamma$, which is a contradiction. Thus, there must be some arrow $E_{j+1} \in \mathcal E^{\#}$ connected to $v_{j+1}$ and such that $E_{j+1} \not = E_j$. Since $v_{j+1}$ is already the head of $E_j$, it must be the tail of $E_{j+1}$ by construction of $\Gamma^{\#}$ and point (3) of Definition \ref{DefinitionKraftQuiver}. Let us write $E_{j+1} = (v_{j+1},v_{j+2})$ for some $v_{j+2} \in \mathcal V$. Then $v_{j+2}$ must be different from $v_1,\ldots , v_{j+1}$. By induction, we have found pairwise distinct vertices $v_1,\ldots , v_{n}$ such that $(v_i,v_{i+1}) \in \mathcal E^{\#}$ for all $1 \leq i \leq n-1$. Then point (1) of Definition \ref{DefinitionKraftQuiver} implies that there can not be any other arrow in $\mathcal E^{\#}$ besides $E_1,\ldots , E_{n-1}$, so that $\Gamma$ is linear.\\
    
    \textbf{Case 2:} all vertices in $\mathcal V$ are connected to exactly two distinct arrows in $\Gamma$. Fix $v_1 \in \mathcal V$. Then $v_1$ is either the head and the tail of two arrows in $\mathcal E^{\#}$ with the same label, or the common head or tail of two arrows in $\mathcal E^{\#}$ with different labels. In all cases, one of the two arrows connected to $v$ corresponds to an arrow $E_1 \in \mathcal E^{\#}$ of the form $E_1 = (v_1,v_2)$. Assume that there are pairwise distinct vertices $v_1,\ldots ,v_{j+1} \in \mathcal V$ for some $j\geq 1$ such that $E_i := (v_i,v_{i+1}) \in \mathcal E^{\#}$ for all $1 \leq i \leq j$. By point (1) of Definition \ref{DefinitionKraftQuiver}, the vertices $v_2,\ldots , v_j$ can not be the head or the tail of any other arrow besides $E_1,\ldots , E_j$. Let $E_{j+1} \not = E_j$ be the second arrow in $\mathcal E^{\#}$ connected to $v_{j+1}$. Necessarily, $v_{j+1}$ is the tail of $E_{j+1}$. Let us write $E_{j+1} = (v_{j+1},v_{j+2})$. We have $v_{j+2} \not = v_1$, otherwise $v_1,\ldots , v_{j+1}$ would form a proper connected component in $\Gamma$, which is a contradiction. Thus $v_{j+2}$ is distinct from $v_1,\ldots , v_{j+1}$. By induction, we have found pairwise distinct vertices $v_1,\ldots , v_{n}$ such that $(v_i,v_{i+1}) \in \mathcal E^{\#}$ for all $1 \leq i \leq n-1$. Let $E_n \not = E_{n-1}$ denote the second arrow in $\mathcal E^{\#}$ which is connected to $v_{n}$. Then the head of $E_n$ can not be any vertex other than $v_1$, so that $\Gamma$ is circular.
\end{proof}

Consequently, any Kraft quiver $\Gamma$ can be uniquely decomposed as a disjoint union 
\begin{equation*}
    \Gamma = \Gamma_{\mathrm{lin}} \sqcup \Gamma_{\mathrm{circ}},
\end{equation*}
where $\Gamma_{\mathrm{lin}}$ (resp.~$\Gamma_{\mathrm{circ}}$) is the disjoint union of all the linear (resp.~circular) connected components of $\Gamma$. We call $\Gamma_{\mathrm{lin}}$ the \textit{linear part} of $\Gamma$, and $\Gamma_{\mathrm{circ}}$ the \textit{circular part} of $\Gamma$.

\begin{convention}
    When $\Gamma$ is a non-empty connected circular Kraft quiver as in Definition \ref{DefLinCirc} (2), we will always consider the indices $i$ of the vertices $v_i$ as elements of $\mathbb Z/n\mathbb Z$. 
\end{convention}

\begin{definition}\label{DefRepetition}
    Let $\Gamma$ be a non-empty connected circular Kraft quiver together with an ordering $v_1,\ldots ,v_n$ of its vertices as in Definition \ref{DefLinCirc} (2). For $i \in \mathbb Z/n\mathbb Z$, let $E_i := (v_i,v_{i+1}) \in \mathcal E^{\#}$. We say that $\Gamma$ \textit{has repetitions} if there exists some integer $1 \leq m < n$ dividing $n$ such that 
    \begin{equation*}
        \mathrm{label}^{\#}(E_i) = \mathrm{label}^{\#}(E_{m+i}),
    \end{equation*}
    for all $i \in \mathbb Z/n\mathbb Z$. If no such $m$ exists, we say that $\Gamma$ \textit{has no repetition}.
\end{definition}

Let $\Gamma$ be a non-empty connected circular Kraft quiver with repetitions as above. Assume that $m$ is the smallest proper divisor of $n$ satisfying the condition of Definition \ref{DefRepetition}. We construct a directed graph $(\Gamma')^{\#} = (\mathcal V',(\mathcal E')^{\#},(\mathrm{label}')^{\#})$ with no repetition as follows:
\begin{itemize}
    \item $\mathcal V' := \{v_1',\ldots , v_m'\}$ where the indices are taken in $\mathbb Z/m\mathbb Z$,
    \item $(\mathcal E')^{\#} := \{(v_i',v_{i+1}')\,|\, i \in \mathbb Z/m\mathbb Z\}$,
    \item $(\mathrm{label}')^{\#}((v_i',v_{i+1}')) := \mathrm{label}^{\#}((v_j,v_{j+1}))$ where $j$ is any lift of $i$ via the natural quotient map $\mathbb Z/n\mathbb Z \twoheadrightarrow \mathbb Z/m\mathbb Z$.
\end{itemize}
The graph $(\Gamma')^{\#}$ is the converse of a uniquely determined, connected circular Kraft quiver $\Gamma'$ with no repetitions. We call $\Gamma'$ the \textit{reduction} of $\Gamma$. Note that this construction does not depend on the choice of the ordering of the vertices of $\Gamma$.

\begin{example}
    Figure \ref{Figure1} exhibits a connected linear Kraft quiver on the left, and a connected circular Kraft quiver with no repetition on the right.\\
    Figure \ref{Figure4} shows some connected circular Kraft quivers with repetitions on the left, and their reductions on the right.
\end{example}

\begin{figure} 
\centering
\begin{tikzpicture}[
       decoration = {markings,
                     mark=at position .5 with {\arrow{Stealth[length=2mm]}}},
       dot/.style = {circle, fill, inner sep=2.4pt, node contents={},
                     label=#1},
every edge/.style = {draw, postaction=decorate}
                        ]
\node[regular polygon, regular polygon sides=9, minimum size=6cm] at (0,0) (P) {};
\node (e1) at (P.corner 1) [dot];
\node (e2) at (P.corner 2) [dot];
\node (e3) at (P.corner 3) [dot];
\node (e4) at (P.corner 4) [dot];
\node (e5) at (P.corner 5) [dot];
\node (e6) at (P.corner 6) [dot];
\node (e7) at (P.corner 7) [dot];
\node (e8) at (P.corner 8) [dot];
\node (e9) at (P.corner 9) [dot];

\node[regular polygon, regular polygon sides=3, minimum size = 4cm] at (8,0) (T) {};
\node (f1) at (T.corner 1) [dot];
\node (f2) at (T.corner 2) [dot];
\node (f3) at (T.corner 3) [dot];

\node[regular polygon, regular polygon sides=4, minimum size = 4cm] at (0,-6) (S) {};
\node (g1) at (S.corner 1) [dot];
\node (g2) at (S.corner 2) [dot];
\node (g3) at (S.corner 3) [dot];
\node (g4) at (S.corner 4) [dot];

\node (h) at (8,-6) [dot];

\path (e1) edge node[above] {$F$} (e2);
\path (e2) edge node[above,left] {$F$} (e3);
\path (e4) edge node[left] {$V$} (e3);
\path (e4) edge node[below,left] {$F$} (e5);
\path (e5) edge node[below] {$F$} (e6);
\path (e7) edge node[below,right] {$V$} (e6);
\path (e7) edge node[right] {$F$} (e8);
\path (e8) edge node[above,right] {$F$} (e9);
\path (e1) edge node[above] {$V$} (e9);

\path (f1) edge node[left] {$F$} (f2);
\path (f2) edge node[below] {$F$} (f3);
\path (f1) edge node[right] {$V$} (f3);

\path (g1) edge node[above] {$F$} (g2);
\path (g2) edge node[left] {$F$} (g3);
\path (g3) edge node[below] {$F$} (g4);
\path (g4) edge node[right] {$F$} (g1);

\path (h) edge [loop] node[right] {$F$} (h);

\end{tikzpicture}
\caption{Connected circular Kraft quivers with repetitions (left), and their reductions (right).}
\label{Figure4}
\end{figure}

%\begin{definition}
%    Let $\Gamma= (\mathcal V, \mathcal E, \mathrm{label})$ be a Kraft quiver. The \textit{dual Kraft quiver} of $\Gamma$ is the Kraft quiver $\Gamma^{\vee} := (\mathcal V,\mathcal E^{\vee},\mathrm{label}^{\vee})$ which shares the same set of vertices as $\Gamma$, and such that:
%    \begin{itemize}
%        \item $\mathcal E^{\vee} = \{(v_2,v_1) \,|\, \forall (v_1,v_2) \in \mathcal E\}$,
%        \item for all $E^{\vee} = (v_2,v_1) \in \mathcal E^{\vee}$, writing $E = (v_1,v_2) \in \mathcal E$, we have 
%        \begin{equation*}
%            \mathrm{label}^{\vee}(E^{\vee}) = \begin{cases}
%                F & \text{if } \mathrm{label}(E) = V,\\
%                V & \text{if } \mathrm{label}(E) = F.
%            \end{cases}
%        \end{equation*}
%    \end{itemize}
%\end{definition}
%
%In other words, the dual Kraft quiver is obtained by flipping all the arrows and their labels. Clearly, we have $\Gamma = (\Gamma^{\vee})^{\vee}$. 

\subsection{Connected Kraft quivers and words}\label{SubsectionWords}

Non empty connected Kraft quivers can be entirely described in terms of words in the alphabet $\mathcal A^{\#} = \{F,V^{\#}\}$. More precisely, linear Kraft quivers correspond to finite words, and circular Kraft quivers correspond to (infinite) periodic words.\\
A finite word (in the alphabet $\mathcal A^{\#}$) is a sequence 
\begin{equation*}
    w = w_mw_{m-1}\cdots w_2w_1,
\end{equation*}
where $m \geq 0$ and $w_i \in \mathcal A^{\#}$ for all $1 \leq i \leq m$. The integer $m$ is called the \textit{length} of the word, and is also written $l(w) := m$. The unique word of length $0$ is called the empty word, and sometimes denoted by $\emptyset$. We denote by $W$ the set of finite words. It is a monoid under the operation of concatenation, namely if $w = w_m\cdots w_1$ and $w' = w'_{m'}\cdots w_1'$ then 
\begin{equation*}
    ww' := w_m\cdots w_1w'_{m'}\cdots w_1' \in W.
\end{equation*}
Of course, we have $l(ww') = l(w) + l(w')$.

\begin{definition}\label{DefLinearKraftQuiverWithWord}
    Let $w = w_m\cdots w_1$ be a finite word. We define a directed graph $\Gamma(w)^{\#} = (\mathcal V(w),\mathcal E(w)^{\#},\mathrm{label}^{\#})$ labeled by $\mathcal A^{\#}$ as follows:
    \begin{itemize}
        \item $\mathcal V(w) = \{v_1,\ldots , v_{m+1}\}$ is a set with $m+1$ elements,
        \item $\mathcal E(w)^{\#} = \{(v_i,v_{i+1}) \,|\, 1 \leq i \leq m\}$,
        \item for $E_i := (v_i,v_{i+1}) \in \mathcal E(w)^{\#}$, we have
        \begin{equation*}
            \mathrm{label}^{\#}(E_i) = w_i. 
        \end{equation*}
    \end{itemize}
    The graph $\Gamma(w)^{\#}$ is the converse of a connected linear Kraft quiver uniquely determined up to isomorphism, which we denote by $\Gamma(w)$. 
\end{definition}

\begin{example}
    The connected linear Kraft quiver in Figure \ref{Figure1}, and its converse in Figure \ref{Figure3}, correspond to the word $w = V^{\#}FV^{\#}FF$. We note that the empty word corresponds to the Kraft quiver $\Gamma(\emptyset)$ consisting of a single vertex and no edge.
\end{example}

\begin{definition}\label{DefinitionInfiniteWords}
    An \textit{infinite word} in the alphabet $\mathcal A^{\#}$ is an infinite sequence $\tilde w = \cdots w_t w_{t-1} \cdots w_1$ of letters with $w_i \in \mathcal A^{\#}$ for all $i\geq 1$. If $w' = w_m'\cdots w_1' \in W$ is a finite word, then 
    \begin{equation*}
        \tilde w w := \cdots w_tw_{t-1} \cdots w_1w_m'\cdots w'_1
    \end{equation*}
    is again an infinite word. For $j \geq 0$, the \text{$j$-shift of $\tilde w$} is the infinite word  $\tilde w[j] := \cdots w_{t}w_{t-1}\cdots w_{j+1}$. It is characterized by the identity 
    \begin{equation*}
        \tilde w = \tilde w[j] w_j\cdots w_1.
    \end{equation*}
    An infinite word $\tilde w$ is said to be \textit{periodic} if there exists $t \geq 1$ such that $w_{i+t} = w_i$ for all $i \geq 1$. In this case, the smallest such integer $t$ is called the \textit{period} of $\tilde w$, and we write $[w]$ instead of $\tilde w$, where $w := w_t\cdots w_1 \in W$.
\end{definition}

\begin{remark}\label{RemarkInfinitePeriodicWord}
    We put the stress on the fact that with our conventions, whenever we write $[w]$ it is understood that $w$ has no repetition, so that the period of $[w]$ is $l(w)$. For instance, we do not allow ourselves to write $[FV^{\#}FV^{\#}]$ instead of $[FV^{\#}]$.
\end{remark}

\begin{definition}\label{DefCircularKraftQuiverWithWord}
    Let $[w]$ be a periodic word and let $m$ be a multiple of the period $t = l(w)$ of $[w]$. We define a directed graph $\Gamma([w],m) = (\mathcal V([w],m), \mathcal E([w],m)^{\#}, \mathrm{label}^{\#})$ as follows:
    \begin{itemize}
        \item $\mathcal V([w],m) = \{v_1,\ldots , v_m\}$ is a set whose elements $v_i$ are indexed by $i \in \mathbb Z/m\mathbb Z$,
        \item $\mathcal E([w],m)^{\#} = \{(v_i,v_{i+1}) \,|\, \forall\, 1\leq i \leq m\}$,
        \item for $E_i = (v_i,v_{i+1}) \in \mathcal E([w],m)^{\#}$, we have 
        \begin{equation*}
            \mathrm{label}^{\#}(E_i) = \tilde w_j,
        \end{equation*}
        where $1 \leq j \leq t$ is the residue of $i$ modulo $t$.
    \end{itemize}
    The graph $\Gamma([w],m)^{\#}$ is the converse of a connected circular Kraft quiver, uniquely determined up to isomorphism, which we denote by $\Gamma([w],m)$.
\end{definition}

\begin{example}
    The connected circular Kraft quiver of Figure \ref{Figure1} is $\Gamma([FV^{\#}FV^{\#}V^{\#}],5)$. Those of Figure \ref{Figure4} are respectively, from left to right and top to bottom, $\Gamma([FFV^{\#}],9)$, $\Gamma([FFV^{\#}],3)$, $\Gamma([F],4)$ and $\Gamma([F],1)$.
\end{example}

It is clear that a connected circular Kraft quiver of the form $\Gamma([w],m)$ has no repetition if and only if $m = t$ is the period of $[w]$. Given a periodic word $[w]$ and $j \geq 0$, its $j$-shift $[w][j]$ is again a periodic word with the same period as $\tilde w$. More precisely, if $t = l(w)$ and $0 \leq j \leq t-1$ we have $[w][j] = [w(j)]$ where 
\begin{equation*}
    w(j) := w_j \cdots w_2w_1w_t \cdots w_{j+2}w_{j+1}, 
\end{equation*}
and if $j>t-1$, we have $[w][j] = [w][j']$ where $0 \leq j' \leq t-1$ is the residue of $j$ modulo $t$. The connected circular Kraft quiver $\Gamma([w],m)$ and $\Gamma([w(j)],m)$ are isomorphic for all $m$ multiple of the period $t$.

\begin{proposition}\label{WordOfKraftQuiver}
    Let $\Gamma$ be a non-empty connected Kraft quiver. If $\Gamma$ is linear, there exists a unique finite word $w$ such that $\Gamma \simeq \Gamma(w)$. If $\Gamma$ is circular, there exists a periodic word $[w]$, unique up to shift, and a unique multiple $m$ of $l(w)$ such that $\Gamma \simeq \Gamma([w],m)$. 
\end{proposition}

We omit the proof.

\subsection{$(\sigma,\tau)$-linear representations of Kraft quivers}

Let $K$ be any field and let $\sigma,\tau \in \mathrm{Aut}(K)$ be any field automorphisms of $K$. 

\begin{definition}\label{DefinitionSigmaLinearRepresentation}
    Let $\Gamma = (\mathcal V, \mathcal E, \mathrm{label})$ be a finite directed graph labeled by $\mathcal A = \{F,V\}$. A \textit{$(\sigma,\tau)$-linear representation} of $\Gamma$ over $K$ is a collection of the form $(U,\rho) = (U_v,\rho_E)_{v,E}$ for all $v \in \mathcal V$ and $E \in \mathcal E$, where 
    \begin{itemize}
        \item $U_v$ is a finite-dimensional $K$-vector space,
        \item $\rho_E: U_{v_1} \rightarrow U_{v_2}$ is a $\sigma$-linear (resp.~$\tau$-linear) homomorphism when $E = (v_1,v_2)$ is an $F$-arrow (resp.~a $V$-arrow).
    \end{itemize}
    We say that the representation $(U,\rho)$ is \textit{strict} if $\rho_E$ is an isomorphism for all $E \in \mathcal E$.
\end{definition}

The \textit{trivial $(\sigma,\tau)$-linear representation} of a finite directed graph $\Gamma$ labeled by $\mathcal A$ is denoted by $\mathbf 1_{\Gamma}$. It corresponds to the choice of $U_v = K$ for all $v\in \mathcal V$, and $\rho_E = \sigma$ (resp.~$\rho_E = \tau$) for all $F$-arrow $E$ (resp.~$V$-arrow $E$). Clearly $\mathbf 1_{\Gamma}$ is strict. 

\begin{definition}
    If $(U,\rho)$ and $(U',\rho')$ are two $(\sigma,\tau)$-linear representations of a finite directed graph $\Gamma$ labeled by $\mathcal A$, a \textit{morphism of representations} $f:(U,\rho) \to (U',\rho')$ is a collection of linear homomorphisms $f_v:U_v\to U_v'$ for all $v\in \mathcal V$, such that the following diagram commutes
    \begin{center}
        \begin{tikzcd}
            U_{v_1} \arrow[r,"f_{v_1}"] \arrow[d,"\rho_{E}"] & U'_{v_1} \arrow[d,"\rho'_E"] \\
            U_{v_2} \arrow[r,"f_{v_2}"] & U'_{v_2}
        \end{tikzcd}
    \end{center}
    for all arrows $E = (v_1,v_2) \in \mathcal E$. The morphism $f$ is an isomorphism if all the $f_v$ are linear isomorphisms.
\end{definition}

The direct sum of two $(\sigma,\tau)$-linear representations is given by the natural formula 
\begin{equation*}
    (U,\rho) \oplus (U',\rho') := (U_v\oplus U_{v'}, \rho_E \oplus \rho'_E)_{v,E}.
\end{equation*}
If $\Gamma_1,\Gamma_2$ are two directed graphs and $(U_1,\rho_1)$, $(U_2,\rho_2)$ are two $(\sigma,\tau)$-linear representations on $\Gamma_1$ and $\Gamma_2$ respectively, they induce naturally a $(\sigma,\tau)$-linear representation on the disjoint union $\Gamma_1\sqcup \Gamma_2$.

\begin{notation}
    We introduce the twisted free $K$-algebra $K\langle F,V\rangle_{\sigma,\tau}$ generated by two non-commutative indeterminates $F$ and $V$ subject to the following condition
    \begin{equation*}
        \forall\, a \in K, \qquad Fa = \sigma(a)F, \qquad Va = \tau(a)V.
    \end{equation*}
    We will also denote by $K[F,V]_{\sigma,\tau}$ the quotient of $K\langle F,V\rangle_{\sigma,\tau}$ by the two-sided ideal generated by $FV$ and by $VF$. When $K$ is a perfect field of characteristic $p$, $\sigma$ is the Frobenius morphism and $\tau = \sigma^{-1}$, $K[F,V]_{\sigma,\sigma^{-1}}$ is the Dieudonné ring modulo $p$. Left $K[F,V]_{\sigma,\tau}$-modules which are finite dimensional over $K$ are called \textit{twisted Gelfand-Ponomarev modules}.
\end{notation} 

\begin{definition}\label{DefinitionModuleAttachedToKraftQuiver}
    Let $(U,\rho)$ be a $(\sigma,\tau)$-linear representation of a finite directed graph $\Gamma$ labeled by $\mathcal A$. The \textit{module attached to $(U,\rho)$} is the left $K\langle F,V\rangle_{\sigma,\tau}$-module defined by 
    \begin{equation*}
        M(\Gamma,U,\rho) := \bigoplus_{v \in \mathcal V} U_v,
    \end{equation*}
    together with the operators $F$ and $V$ whose restrictions to a given summand $U_v$ are
    \begin{equation*}
        F_{|U_v} = \bigoplus_{\substack{E = (v,v') \in \mathcal E\\ \mathrm{label}(E)  = F}} \rho_E,
    \end{equation*}
    and 
    \begin{equation*}
        V_{|U_v} = \bigoplus_{\substack{E = (v,v') \in \mathcal E\\ \mathrm{label}(E)  = V}} \rho_E.
    \end{equation*}
\end{definition}

In particular, if there is no $F$-arrow (resp.~$V$-arrow) whose tail is $v$, then $F_{|U_v} = 0$ (resp.~$V_{|U_v} = 0$). Clearly, the module attached to a direct sum of $(\sigma,\tau)$-linear representations, is the direct sum of the individual modules. We also have another direct sum decomposition along the connected components of $\Gamma$.

\begin{proposition}
    Let $(U,\rho)$ be a $(\sigma,\tau)$-linear representation of a finite directed graph $\Gamma$ labeled by $\mathcal A$. Let $\Gamma_1,\ldots , \Gamma_n$ be the connected components of $\Gamma$, and let $(U_i,\rho_i)_{1\leq i \leq n}$ denote the $(\sigma,\tau)$-linear representations of the $\Gamma_i$ induced by $(U,\rho)$. There is a natural decomposition
    \begin{equation*}
        M(\Gamma,U,\rho) = \bigoplus_{i=1}^n M(\Gamma_i,U_i,\rho_i),
    \end{equation*}
    as the direct sum of left $K\langle F,V \rangle_{\sigma,\tau}$-modules.
\end{proposition}

We omit the proof.

\begin{proposition}
    Let $\Gamma$ be a Kraft quiver together with a $(\sigma,\tau)$-linear representation $(U,\rho)$. The module $M(\Gamma,U,\rho)$ is naturally a twisted Gelfand-Ponomarev module.
\end{proposition}

\begin{proof}
    By the definition of Kraft quivers, if $v$ is the head of an $F$-arrow then $V_{|U_v} = 0$ since there is no $V$-arrow whose tail is $v$. It follows that $VF = 0$ on $M(\Gamma,U,\rho)$. Likewise we also have $FV = 0$.
\end{proof}

In order to classify $(\sigma,\tau)$-linear representations of Kraft quivers which are strict on every connected component, one can restrict to the connected case and distinguish between linear and circular Kraft quivers. 

\begin{proposition}\label{ModuleAttachedToLinearKraftQuiver}
    Let $\Gamma = (\mathcal V,\mathcal E,\mathrm{label})$ be a connected linear Kraft quiver, and let $(\rho,U)$ be a strict $(\sigma,\tau)$-linear representation of $\Gamma$. Let $d$ denote the dimension of the $K$-vector space $U_v$ for any vertex $v$. Then $(\rho,U)$ is isomorphic to $\mathbf 1_{\Gamma}^{\oplus d}$.
\end{proposition}

\begin{proof}
    The proof is the same as \cite{chai:kraft2025} Lemma 4.4. Since $\Gamma$ is connected and $(U,\rho)$ is strict, it is clear that all the $K$-vector spaces $U_v$ have the same dimension $d$. Fix an ordering $v_1, \ldots , v_n$ of the vertices of $\Gamma$ such that $\mathcal E^{\#}$ consists of the arrows $(v_i,v_{i+1})$ for all $1 \leq i \leq n-1$. Let $E_i \in \mathcal E$ denote the arrow corresponding to $(v_i,v_{i+1}) \in \mathcal E^{\#}$. Fix a basis $(e_{1,j})_{1 \leq j \leq d}$ of $U_{v_1}$. By induction, define vectors $e_{i,j}$ for $1\leq i \leq n-1$ and $1 \leq j \leq d$ via the formula 
    \begin{equation*}
        e_{i+1,j} = \begin{cases}
            \rho_{E_{i}}(e_{i,j}) & \text{if } \mathrm{label}(E_{i}) = F;\\
            \rho_{E_i}^{-1}(e_{i,j}) & \text{if } \mathrm{label}(E_i) = V.
        \end{cases}
    \end{equation*}
    Clearly, for each $1\leq i \leq n$, the family $(e_{i,j})_{1\leq j\leq d}$ is a basis of $U_{v_i}$. For each $j$, define a $(\sigma,\tau)$-linear subrepresentation $(U_j,\rho_j)$ of $(U,\rho)$ by setting $U_{j,i} := K\cdot e_{i,j}$ and $\rho_{j,E}$ the restriction of $\rho_{E}$ to the corresponding subspace $U_{j,i}$. Then $(U_j,\rho_j) \simeq \mathbf 1_{\Gamma}$ is isomorphic to the trivial representation, and we have a decomposition
    \begin{equation*}
        (U,\rho) \simeq \bigoplus_{j=1}^{d} (U_j,\rho_j).
    \end{equation*}
\end{proof}

\begin{remark}
    It follows that if $\Gamma$ is a connected linear Kraft quiver and $(U,\rho)$ is a strict $(\sigma,\tau)$-linear representation of $\Gamma$ with $\dim(U_v) = d$ for all vertices $v$, then 
    \begin{equation*}
        M(\Gamma,U,\rho) \simeq M(\Gamma,\mathbf 1_{\Gamma})^{d},
    \end{equation*}
    as left $K[F,V]_{\sigma,\tau}$-modules.
\end{remark}

Let us now consider a connected circular Kraft quiver $\Gamma = (\mathcal V,\mathcal E,\mathrm{label})$, together with a strict $(\sigma,\tau)$-linear representation $(U,\rho)$. Fix an ordering $v_1,\ldots , v_n$ such that $\mathcal E^{\#} = \{(v_i,v_{i+1}), i \in \mathbb Z/n\mathbb Z\}$. For each $i \in \mathbb Z/n\mathbb Z$, let $E_i \in \mathcal E$ denote the arrow corresponding to $(v_i,v_{i+1}) \in \mathcal E^{\#}$, and consider the semilinear isomorphism $\phi_{v_i}: U_{v_i} \xrightarrow{\sim} U_{v_{i+1}}$ given by 
\begin{equation*}
    \phi_{v_i} = \begin{cases}
        \rho_{E_i} & \text{if } \mathrm{label}(E_i) = F;\\
        \rho_{E_i}^{-1} & \text{if } \mathrm{label}(E_i) = V.
    \end{cases}
\end{equation*}
We write $\xi_i := \sigma$ if $\mathrm{label}(E_i) = F$, and $\xi_i := \tau^{-1}$ if $\mathrm{label}(E_i) = V$. Thus $\phi_{v_i}$ is a $\xi_i$-linear isomorphism.

\begin{definition}\label{DefinitionMonodromy}
    With the notations above, the \textit{monodromy operator} of $(U,\rho)$ at a vertex $v_i$ is the $\Xi_i$-linear automorphism $\Phi_{v_i}:U_{v_i}\xrightarrow{\sim} U_{v_i}$ given by
    \begin{equation*}
        \Phi_{v_i} := \phi_{v_{i-1}} \circ \phi_{v_{i-2}} \circ \cdots \circ \phi_{v_{i+1}} \circ \phi_{v_i},
    \end{equation*}
    where $\Xi_i := \xi_{i-1}\xi_{i-2}\cdots \xi_{i+1}\xi_{i} \in \mathrm{Aut}(K)$.
\end{definition}

\begin{proposition}\label{Monodromy}
    Let $\Gamma = (\mathcal V,\mathcal E,\mathrm{label})$ be a non-empty connected circular Kraft quiver and let $n = \# \mathcal V = \# \mathcal E > 0$. Let $(U,\rho)$ and $(U',\rho')$ be two strict $(\sigma,\tau)$-linear representations of $\Gamma$. The following statements are equivalent.
    \begin{enumerate}
        \item The representations $(U,\rho)$ and $(U',\rho')$ are isomorphic.
        \item For every $v \in \mathcal V$, the monodromy operators $\Phi_v$ and $\Phi'_v$ at $v$ for $(U,\rho)$ and $(U',\rho')$ respectively are conjugate.
        \item For some $v \in \mathcal V$, the monodromy operators $\Phi_v$ and $\Phi'_v$ at $v$ for $(U,\rho)$ and $(U',\rho')$ respectively are conjugate.
    \end{enumerate}
\end{proposition}

Here, we say that $\Phi_v$ and $\Phi'_v$ are conjugate if there exists a linear isomorphism $h:U_v \xrightarrow{\sim} U'_v$ such that 
\begin{equation*}
    \Phi'_v = h \circ \Phi_v \circ h^{-1}.
\end{equation*}

\begin{proof}
    The implication (2) $\implies$ (3) is trivial.\\
    We prove (1) $\implies$ (2). Let $f:(U,\rho) \xrightarrow{\sim} (U',\rho')$ be an isomorphism, and let $v_1,\ldots , v_n$ be an ordering of the vertices of $\Gamma$ such that $\mathcal E^{\#} = \{(v_i,v_{i+1}),i\in\mathbb Z/n\mathbb Z\}$. For all $i$, let $\phi_{v_i}$ and $\phi'_{v_i}$ be the $\xi_i$-linear isomorphisms defined above for to $(U,\rho)$ and $(U',\rho')$ respectively. By definition, we have $f_{v_{i+1}}\circ \phi_{v_i} = \phi'_{v_{i}} \circ f_{v_i}$ for all $i$. From this, it follows that
    \begin{equation*}
        \Phi_{v_i} = f_{v_i}^{-1} \circ \Phi_{v_i}' \circ f_{v_i},
    \end{equation*}
    so that $\Phi_{v_i}$ and $\Phi'_{v_i}$ are conjugate.\\
    We prove (3) $\implies$ (1). Let $i \in \mathbb Z/n\mathbb Z$ such that $\Phi_{v_i} = h^{-1} \circ \Phi_{v_i}' \circ h$ for some $h: U_{v_i} \xrightarrow{\sim} U'_{v_i}$. Put $f_{v_i} := h$, and for $1 \leq k \leq n-1$ define inductively a linear isomorphism $f_{v_{i+k}}:U_{v_{i+k}} \xrightarrow{\sim} U'_{v_{i+k}}$ by
    \begin{equation*}
        f_{v_{i+k}} := \phi_{v_{i+k-1}}' \circ f_{v_{i+k-1}} \circ \phi_{v_{i+k-1}}^{-1}.
    \end{equation*}
    It is easy to check that $f := (f_{v_i})_{i \in \mathbb Z/n\mathbb Z}$ defines an isomorphism $(U,\rho) \xrightarrow{\sim} (U',\rho')$.
\end{proof}

In other words, given a connected circular Kraft quiver $\Gamma$ with $n$ vertices, there is a natural bijection
\begin{equation*}
    \hspace*{-1cm}  
    \left\{\begin{array}{c}
    \text{isomorphism classes of strict } (\sigma,\tau)\text{-linear}\\
    \text{representations of } \Gamma \text{ over } K
    \end{array}\right\} \xrightarrow{\sim}
    \left\{\begin{array}{c}
    \text{conjugacy classes of finite dimensional } K\text{-vector}\\
    \text{spaces equipped with a } \Xi \text{-linear automorphism}
    \end{array}\right\},
\end{equation*}
where $\Xi \in \mathrm{Aut}(K)$ is any one of the automorphisms $\Xi_i$ associated to the vertices of $\Gamma$. The structure of the set on the right-hand side may depend heavily on the choice of $(K,\sigma,\tau)$. We give two classical examples. 

\begin{proposition}\label{PropositionJordanForm}
    Assume that $K = \mathbb C$ and that $\sigma = \tau = \mathrm{id}$. Conjugacy classes of finite dimensional $\mathbb C$-vector spaces equipped with an automorphism correspond bijectively to finite sets of the form $X = \{(m_i,d_i,\mu_i)\}$ where, for $r := \#X$ and for every $1 \leq i \leq r$, $m_i,d_i \geq 1$ are positive integers, $\mu_i \in \mathbb C^{\times}$ and $(d_i,\mu_i) \not = (d_j,\mu_j)$ whenever $i\not = j$.
\end{proposition}

\begin{proof}
    This amounts to writing the automorphism in its Jordan normal form. The number of distinct Jordan blocks is $r$, the Jordan blocks of eigenvalue $\mu_i$ have size $d_i$, and $m_i$ is the multiplicity of the Jordan block $(d_i,\mu_i)$. 
\end{proof}

\begin{proposition}\label{LangSteinbergTheorem}
    Assume that $K$ is an algebraically closed field of positive characteristic $p$, and that $\sigma$ is given by $\sigma(x) = x^p$. Let $k \geq 1$. For every $d\geq 0$, there is only a single conjugacy class of $d$-dimensional vector spaces equipped with a $\sigma^k$-linear automorphism.
\end{proposition}

\begin{proof}
    This is the Lang-Steinberg theorem, which was first proved by Hasse and Witt. See \cite{chai:kraft2025} Proposition 2.4 for a detailed account. 
\end{proof}

\begin{corollary}\label{ModulesAttachedToCircularKraftQuiverAlgCl}
    Assume that $K$ is an algebraically closed field of positive characteristic $p$, that $\sigma$ is given by $\sigma(x) = x^p$ and that $\tau := \sigma^{-1}$. Let $\Gamma$ be a connected circular Kraft quiver and let $(U,\rho)$ be a strict $(\sigma,\tau)$-linear representation of $\Gamma$. Then $M(\Gamma,U,\rho) \simeq M(\Gamma,\mathbf{1}_{\Gamma})^d$ where $d = \mathrm{dim}(U_v)$ for any vertex $v$.
\end{corollary}

Eventually, we note that if we are only interested in the classification of modules of the form $M(\Gamma,U,\rho)$ where $\Gamma$ is a connnected circular Kraft quiver, then we can always reduce to the case where $\Gamma$ has no repetition in virtue of the following Lemma.

\begin{lemma}\label{ReductionToCaseOfNonRepetition}
    Let $\Gamma = \Gamma([w],m)$ be a connected circular Kraft quiver as defined in Definition \ref{DefCircularKraftQuiverWithWord}, where $m = kt$ is a multiple of $t = l(w)$, the period of $[w]$. For every $(\sigma,\tau)$-linear representation $(U,\rho)$ on $\Gamma$, there exists a $(\sigma,\tau)$-linear representation $(U',\rho')$ on $\Gamma' := \Gamma([w],t)$ such that we have an isomorphism of twisted Gelfand-Ponomarev modules 
    \begin{equation*}
        M(\Gamma,U,\rho) \xrightarrow{\sim} M(\Gamma',U',\rho').
    \end{equation*}
\end{lemma}

\begin{proof}
    Let $(U,\rho)$ be as in the statement. With the same notations as Definition \ref{DefCircularKraftQuiverWithWord}, let us write $v_1,\ldots , v_{m}$ for the vertices of $\Gamma$, and $v_1',\ldots ,v_t'$ for the vertices of $\Gamma'$. The indices of the $v_i$'s (resp.~of the $v_i'$'s) are seen in $\mathbb Z/m\mathbb Z$ (resp.~in $\mathbb Z/t\mathbb Z$). Let $\varphi:\mathbb Z/m\mathbb Z \twoheadrightarrow \mathbb Z/t\mathbb Z$ denote the projection. For $i \in \mathbb Z/t\mathbb Z$, we define 
    \begin{equation*}
        U_{v_i'}' := \bigoplus_{j \in \varphi^{-1}\{i\}} U_{v_j},
    \end{equation*}
    and if $E' = (v_{i}',v'_{i+1})$ is an $F$-arrow in $\Gamma'$, we denote by $\varphi^{-1}\{E'\}$ the set of all the $F$-arrows $(v_{i_0+at},v_{i_0+at+1})$ in $\Gamma$, where $i_0 \in \mathbb Z/m \mathbb Z$ is a fixed lift of $i$ and $0 \leq a \leq k-1$. We define $\varphi^{-1}\{E'\}$ similarly if $E'$ is a $V$-arrow. Then, given any arrow $E'$ in $\Gamma'$, we define
    \begin{equation*}
        \rho_E' := \bigoplus_{E \in \varphi^{-1}\{E'\}} \rho_E.
    \end{equation*}
    It is clear that $M(\Gamma',U',\rho')$ produces a module which is isomorphic to $M(\Gamma,U,\rho)$.
\end{proof}

\section{Twisted Gelfand-Ponomarev modules of the first and second kinds} \label{Section3}

In Definition \ref{DefinitionModuleAttachedToKraftQuiver}, we have seen how to build a twisted Gelfand-Ponomarev module out of any $(\sigma,\tau)$-linear representation of a Kraft quiver. In this section, we will study the converse. Namely, given such a module $M$, how can we produce a Kraft quiver $\Gamma$ together with a $(\sigma,\tau)$-linear representation $(U,\rho)$ such that $M \simeq M(\Gamma,U,\rho)$? As it turns out, the process does not depend much on the choice of $(K,\sigma,\tau)$, so that most of the proofs here will be directly adapted from \cite{gelfand-ponomarev:1968}. We will use notations and definitions introduced in Section \ref{Section1} without recalling most of them.

\begin{remark}
    We warn the reader who might be interested in reading the original proofs in \cite{gelfand-ponomarev:1968} that there, the authors use the symbol $\subset$ for \textit{proper} inclusions. This differs from our notations, as we use $\subsetneq$ instead. 
\end{remark}

\subsection{Stabilized sequence and elementary intervals}

Let $M$ be a twisted Gelfand-Ponomarev module. Recall that $W$ denotes the monoid of finite words in the alphabet $\mathcal A^{\#} = \{F,V^{\#}\}$ as in Section \ref{SubsectionWords}.

\begin{definition}
    Let $w_1,w_2 \in W$. We say that $w_1$ is a \textit{prefix} (resp.~a \textit{suffix}) of $w_2$ if there exists some $w \in W$ such that $w_2 = w_1w$ (resp.~$w_2 = ww_1$).
\end{definition}

We consider the $\sigma$-linear relation $\Gamma_F$ and the $\tau^{-1}$-linear relation $(\Gamma_V)^{\#}$ on $M$ in the sense of Section \ref{Section1}. By abuse of notations, we still write $F$ and $V^{\#}$ instead of $\Gamma_F$ and $(\Gamma_V)^{\#}$. If $w = w_{m}w_{m-1} \cdots w_1 \in W$ is a word of length $m$, we write 
\begin{equation*}
    \xi(w)_i := \begin{cases}
        \sigma & \text{if } w_i = F;\\
        \tau^{-1} & \text{if } w_i = V^{\#},
    \end{cases}
\end{equation*} 
and we define $\xi(w) := \xi(w)_m \cdots \xi(w)_1 \in \mathrm{Aut}(W)$. Then $w$ gives rise to a $\xi(w)$-linear relation $D(w)$ on $M$. We call \textit{monomial} (in $F$ and $V^{\#}$) any semilinear relation $D$ on $M$ of the form 
\begin{equation*}
    D = F^{m_1}(V^{\#})^{m_2} \cdots F^{m_{r-1}}(V^{\#})^{m_r},
\end{equation*}
where $r \geq 1$, and $m_1,m_r \geq 0$ while $m_i > 0$ for $1 < i < r$, and $m = \sum m_i$. With $r=1$ and $m_1 = 0$, we also consider $\mathbf 1$ as a monomial. The set of monomials is denoted $\Sigma$. The application $w \mapsto D(w)$ defines a surjective morphism of monoids $W \to \Sigma$, where $\Sigma$ is equipped with the composition of relations. In general, this application is not injective. For instance, if $\Gamma = \Gamma(V^{\#}F)$, $M = M(\Gamma,\mathbf 1_{\Gamma})$ and $\tau = \sigma$, one may check that $D(FV^{\#}FV^{\#}) = D(FV^{\#})$.

\begin{remark}
    With our notations, formulas of the form ``$D(V^{\#}FF) = V^{\#}FF = V^{\#}F^2$'' make sense. On the left-hand side, $V^{\#}FF$ is seen as a word in $W$, whereas on the right-hand side it is seen as a monomial in $\Sigma$. We hope that the context is sufficiently clear throughout the exposition so that no confusion arises.
\end{remark}

\begin{remark}
    Since $FV = 0$ on $M$, we have $\mathrm{Dom}(V^{\#}) \subseteq  \mathrm{Ker}(F)$.
\end{remark}

\begin{definition}
    The \textit{stabilized sequence} of $M$ is the set 
    \begin{equation*}
        \mathcal F_M = \{\mathrm{Ker}(D), D \in \Sigma\} \cup \{\mathrm{Dom}(D), D\in \Sigma\}.
    \end{equation*}
\end{definition}

Thus $\mathcal F_M$ is a set consisting of certain subspaces of $M$. As it turns out, it is actually a flag. 

\begin{theorem}[cf.~\cite{gelfand-ponomarev:1968} Theorem 1.1]\label{TheoremStabilizedSequence}
    The stabilized sequence $\mathcal F_M$ is finite, non-empty, and there exists a unique numbering $\mathcal F = \{\beta_1, \ldots , \beta_s\}$ of its elements such that 
    \begin{equation*}
        \{0\} = \beta_0 \subsetneq \beta_1 \subsetneq \beta_2 \subsetneq \cdots \subsetneq \beta_{s-1} \subsetneq \beta_s = M.
    \end{equation*}
\end{theorem}

We call \textit{interval} any pair $(\alpha_1,\alpha_2)$ of $K$-subspaces of $M$ such that $\alpha_1 \subseteq  \alpha_2$. The intervals of the form $(\beta_i,\beta_{i+1})$ for $0 \leq i \leq s-1$ are called the \textit{elementary intervals} of $M$. Before proving the Theorem, we start with a Lemma.

\begin{lemma}[cf.~\cite{gelfand-ponomarev:1968} Lemma 1.1 and \cite{chai:kraft2025} Lemma 6.3.2]\label{LemmaRelativePositionIntervals}
    Let $w_1,w_2 \in W$ be two words. Then, up to exchanging $w_1$ and $w_2$ if necessary, exactly one of the following assertions hold:
    \begin{enumerate}
        \item $w_1 = wV^{\#}w_0$ and $w_2 = w'Fw_0$ for some words $w,w',w_0 \in W$. In this case the interval $(\mathrm{Ker}D(w_2),\mathrm{Dom}D(w_2))$ follows $(\mathrm{Ker}D(w_1),\mathrm{Dom}D(w_1))$, that is,
        \begin{equation*}
            \mathrm{Ker}D(w_1) \subseteq  \mathrm{Dom}D(w_1) \subseteq  \mathrm{Ker}D(w_2) \subseteq  \mathrm{Dom}D(w_2),
        \end{equation*}
        \item $w_2 = ww_1$ for some word $w \in W$. In this case the interval $(\mathrm{Ker}D(w_2),\mathrm{Dom}D(w_2))$ is sandwiched inside $(\mathrm{Ker}D(w_1),\mathrm{Dom}D(w_1))$, that is,
        \begin{equation*}
            \mathrm{Ker}D(w_1) \subseteq  \mathrm{Ker}D(w_2) \subseteq  \mathrm{Dom}D(w_2) \subseteq  \mathrm{Dom}D(w_1).
        \end{equation*}
    \end{enumerate}
\end{lemma}

\begin{proof}
    Let $w_0$ denote the maximal common suffix of $w_1$ and $w_2$, so that $w_1 = w_1'w_0$ and $w_2 = w_2'w_0$ for some possibly empty words $w_1',w_2'$. Then either $w_1'$ or $w_2'$ is the empty word, either both of them are non-empty. The former corresponds to (2) and the latter corresponds to (1).\\
    Assume first that both $w_1'$ and $w_2'$ are not empty. The maximality of $w_0$ implies that, up to exchanging $w_1$ and $w_2$, we have $w_1' = wV^{\#}$ and $w_2' = w'F$ for some words $w,w' \in W$. By Lemma \ref{UsefulLemma}, we have
    \begin{equation*}
        \mathrm{Dom}D(wV^{\#}) \subseteq  \mathrm{Dom}(V^{\#}) \subseteq  \mathrm{Ker}(F) \subseteq  \mathrm{Ker}D(w'F).
    \end{equation*}
    In particular, we have $\mathrm{Dom}D(wV^{\#}) \subseteq  \mathrm{Ker}D(w'F)$. By Proposition \ref{UsefulProp}, multiplying on the right with $D(w_0)$ preserves the inclusion, so that $\mathrm{Dom}D(w_1) \subseteq  \mathrm{Ker}D(w_2)$. We deduce that
    \begin{equation*}
        \mathrm{Ker}D(w_1) \subseteq  \mathrm{Dom}D(w_1) \subseteq  \mathrm{Ker}D(w_2) \subseteq  \mathrm{Dom}D(w_2),
    \end{equation*}
    Assume now that $w_1'$ or $w_2'$ is empty. Up to exchanging $w_1$ and $w_2$, we may assume that $w_1'$ is empty. It follows that $w_2 = w_2'w_1$. By Lemma \ref{UsefulLemma}, we readily have
    \begin{equation*}
            \mathrm{Ker}D(w_1) \subseteq  \mathrm{Ker}D(w_2) \subseteq  \mathrm{Dom}D(w_2) \subseteq  \mathrm{Dom}D(w_1).
    \end{equation*}
\end{proof}

\begin{proof}[Proof of Theorem \ref{TheoremStabilizedSequence}]
Lemma \ref{LemmaRelativePositionIntervals} implies that the set $\mathcal F$ is a chain with respect to inclusion. Since $M$ is finite dimensional, $\mathcal F$ must be finite. It is also trivially non-empty.
\end{proof}

From now on, we will always write $\mathcal F = \{\beta_0,\ldots , \beta_s\}$ as in Theorem \ref{TheoremStabilizedSequence}. One may separate the elementary intervals $(\beta_i,\beta_{i+1})$ into two families. Observe that for every $D \in \Sigma$, we have
\begin{equation}\label{CompletionInterval}
\mathrm{Ker}(D) \subseteq  \mathrm{Dom}(V^{\#}D) \subseteq  \mathrm{Ker}(FD) \subseteq  \mathrm{Dom}(D).
\end{equation}

\begin{definition}
    We define a subset $\Sigma_1 \subseteq  \Sigma$ as follows
    \begin{equation*}
        \Sigma_1 := \{D \in \Sigma, \mathrm{Dom}(V^{\#}D) \subsetneq \mathrm{Ker}(FD)\}.
    \end{equation*}
    The elements of $\Sigma_1$ are called \textit{monomials of the first kind}. We also define $W_1 := D^{-1}(\Sigma_1)$, and elements of $W_1$ are called \textit{words of the first kind}.
\end{definition}

In particular, every monomial of the first kind $D \in \Sigma_1$ is non-null in the sense of Definition \ref{DefNonNull}.

\begin{proposition}[cf.~\cite{gelfand-ponomarev:1968} Propositions 1.6 and 1.7]\label{MonomialOfFirstKind}
    For $w \in W_1$, the interval $(\mathrm{Dom}D(V^{\#}w),\mathrm{Ker}D(Fw))$ is elementary. Moreover if $w'\in W_1$, then $w \not = w'$ implies that $(\mathrm{Dom}D(V^{\#}w),\mathrm{Ker}D(Fw)) \not = (\mathrm{Dom}D(V^{\#}w'),\mathrm{Ker}D(Fw'))$.
\end{proposition}

\begin{proof}
    Let us consider some $\beta_i \in \mathcal F_M$ lying in between $\mathrm{Dom}D(V^{\#}w)$ and $\mathrm{Ker}D(Fw)$, that is, 
    \begin{equation*}
        \mathrm{Dom}D(V^{\#}w) \subseteq  \beta_i \subseteq  \mathrm{Ker}D(Fw).
    \end{equation*}
    We must prove that one of these two inclusions is in fact an equality. There is a word $v \in W$ such that $\beta_i$ is equal to $\mathrm{Ker}D(v)$ or $\mathrm{Dom}D(v)$. Assume first that $w$ is not a suffix of $v$. Since we have $\mathrm{Ker}D(w) \subseteq  \beta_i \subseteq  \mathrm{Dom}D(w)$, by distinguishing cases in Lemma \ref{LemmaRelativePositionIntervals} with respect to $w$ and $v$, we find out that $\beta_i = \mathrm{Ker}D(w)$ or $\beta_i = \mathrm{Dom}D(w)$. The former implies that $\beta_i = \mathrm{Dom}D(V^{\#}w)$ and the latter implies that $\beta_i = \mathrm{Ker}D(Fw)$. Thus, let now assume that we can write $v = v'w$ for some word $v' \in W$. \\

    \textbf{Case 1:} $v'$ is the empty word. If $\beta_i = \mathrm{Ker}D(v) = \mathrm{Ker}D(w)$, then we get $\beta_i = \mathrm{Dom}D(V^{\#}w)$. If $\beta_i = \mathrm{Dom}D(v) = \mathrm{Dom}D(w)$, then we get  $\beta_i = \mathrm{Ker}D(Fw)$. \\

    \textbf{Case 2:} $v' = v''V^{\#}$ for some word $v'' \in W$. In this case $V^{\#}w$ is a suffix of $v$ so that we are in Case (2) of Lemma \ref{LemmaRelativePositionIntervals}, that is,
    \begin{equation*}
        \mathrm{Ker}D(V^{\#}w) \subseteq  \mathrm{Ker}D(v) \subseteq  \mathrm{Dom}D(v) \subseteq  \mathrm{Dom}D(V^{\#}w).    
    \end{equation*}
    It follows that $\beta_i \subseteq  \mathrm{Dom}D(V^{\#}w) \subseteq  \beta_i$, so that this is in fact an equality.\\

    \textbf{Case 3:} $v' = v''F$ for some word $v'' \in \Sigma$. In this case $Fw$ is a suffix of $v$ so that we are in Case (2) of Lemma \ref{LemmaRelativePositionIntervals}, that is,
    \begin{equation*}
        \mathrm{Ker}D(Fw) \subseteq  \mathrm{Ker}D(v) \subseteq  \mathrm{Dom}D(v) \subseteq  \mathrm{Dom}D(Fw).    
    \end{equation*}
    It follows that $\beta_i \subseteq  \mathrm{Ker}D(Fw) \subseteq  \beta_i$, so that there is an equality.\\

    This proves that the interval $(\mathrm{Dom}D(V^{\#}w),\mathrm{Ker}D(Fw))$ is elementary. Let us now consider some $w' \in W_1$ with $w \not = w'$. If $w$ and $w'$ are in Case (1) of Lemma \ref{LemmaRelativePositionIntervals}, then clearly the intervals $(\mathrm{Dom}D(V^{\#}w),\mathrm{Ker}D(Fw))$ and $(\mathrm{Dom}D(V^{\#}w'),\mathrm{Ker}D(Fw'))$ are distinct. Thus, up to exchanging $w$ and $w'$, we may assume that $w$ is a suffix of $w'$. Since $w \not = w'$, we actually have that $V^{\#}w$ or $Fw$ is a suffix of $w'$.\\

    \textbf{Case 1:} $w' = w_1V^{\#}w$ for some word $w_1 \in W$. By Lemma \ref{LemmaRelativePositionIntervals} with respect to $w'$ and $V^{\#}w$, we have 
    \begin{equation*}
        \mathrm{Ker}D(V^{\#}w) \subseteq  \mathrm{Ker}D(w') \subsetneq \mathrm{Dom}D(w') \subseteq  \mathrm{Dom}D(V^{\#}w).
    \end{equation*}
    \sloppy It follows that $\mathrm{Dom}D(w')$ separates the intervals $(\mathrm{Dom}D(V^{\#}w),\mathrm{Ker}D(Fw))$ and $(\mathrm{Dom}D(V^{\#}w'),\mathrm{Ker}D(Fw'))$. Namely, we have 
    \begin{equation*}
        \mathrm{Dom}D(V^{\#}w') \subsetneq \mathrm{Ker}D(Fw')) \subseteq  \mathrm{Dom}D(w') \subseteq  \mathrm{Dom}D(V^{\#}w) \subsetneq \mathrm{Ker}D(Fw)).
    \end{equation*}

    \textbf{Case 2:} $w' = w_1Fw$ for some word $w_1 \in W$. By Lemma \ref{LemmaRelativePositionIntervals} with respect to $w'$ and $Fw$, we have 
    \begin{equation*}
        \mathrm{Ker}D(Fw) \subseteq  \mathrm{Ker}D(w') \subsetneq \mathrm{Dom}D(w') \subseteq  \mathrm{Dom}D(Fw).
    \end{equation*}
    \sloppy It follows that $\mathrm{Ker}D(w')$ separates the intervals $(\mathrm{Dom}D(V^{\#}w),\mathrm{Ker}D(Fw))$ and $(\mathrm{Dom}D(V^{\#}w'),\mathrm{Ker}D(Fw'))$. Namely, we have 
    \begin{equation*}
        \mathrm{Dom}D(V^{\#}w) \subsetneq \mathrm{Ker}D(Fw) \subseteq  \mathrm{Ker}D(w') \subseteq  \mathrm{Dom}D(V^{\#}w') \subsetneq \mathrm{Ker}D(Fw').
    \end{equation*}
    Thus, the Proposition is proved.
\end{proof}

\begin{definition}
    An elementary interval $(\beta_i,\beta_{i+1})$ is said to be \textit{of the first kind} if there exists some $D \in \Sigma_1$ such that $(\beta_i,\beta_{i+1}) = (\mathrm{Dom}(V^{\#}D),\mathrm{Ker}(FD))$. The elementary intervals which are not of the first kind are said to be \textit{of the second kind}.
\end{definition}

In particular, Proposition \ref{MonomialOfFirstKind} implies that $D:W \to \Sigma$ is injective on $W_1$, so that it induces a bijection $W_1 \xrightarrow{\sim} \Sigma_1$. Besides, the sets $W_1$ and $\Sigma_1$ are also in bijection with the set of elementary intervals of the first kind. It follows that they are in fact finite sets. 

\begin{definition}\label{DefinitionTwistedGFModuleOfFirstSecondKind}
    We say that $M$ is a twisted Gelfand-Ponomarev module \textit{of the first kind} (resp.~\textit{of the second kind}) if all its elementary intervals are of the first kind (resp.~of the second kind).
\end{definition}

Over the next two subsections, we will show that $M$ can be decomposed as a direct sum 
\begin{equation*}
    M = M_1 \oplus M_2,
\end{equation*}
where $M_1$ (resp.~$M_2$) is a submodule of the first (resp.~second) kind. 

\subsection{Submodules of the first kind and linear Kraft quivers}

In this section, we assume that $W_1 \not = \emptyset$. We denote by $m \geq 0$ the maximal length of the words $w \in W_1$ of the first kind. We start by defining the \textit{prefix order} on $W$.

\begin{definition}
    We define an ordering relation $\preceq$ on $W$ as follows: for $w_1,w_2 \in W$,
    \begin{equation*}
        w_1 \preceq w_2 \iff w_2 = w_1 w \text{ for some word } w \in W.
    \end{equation*}
\end{definition}

In other words, $w_1 \preceq w_2$ if and only if $w_1$ is a prefix of $w_2$. If $w_2$ is not empty, we say that $w_1$ is an \textit{immediate prefix} of $w_2$ if $w_1 \preceq w_2$ and $l(w_1) = l(w_2) - 1$. Clearly, each non-empty word has a unique immediate prefix.

\begin{remark}
    We note that $\preceq$ does not descend to an ordering relation on $\Sigma$ in general, since anti-symmetry might fail. For example, if $M = M(\Gamma,\mathbf 1_{\Gamma})$ with $\Gamma = \Gamma(V^{\#}F)$ and $\tau = \sigma$, we have $(FV^{\#})^2 = FV^{\#}$ as monomials in $\Sigma$. If we were to define $\preceq$ directly on $\Sigma$, we would have $FV^{\#} \preceq FV^{\#}F$ and $FV^{\#}F \preceq FV^{\#}$, but $FV^{\#} \not = FV^{\#}F$ as monomials. However, since $\Sigma_1$ is in bijection with $W_1$, the restriction of $\preceq$ to $W_1$ does define an order on $\Sigma_1$.
\end{remark}

\begin{definition}
    We define a directed graph $\Gamma(M,1\mathrm{st})^{\#} = (\mathcal V,\mathcal E^{\#},\mathrm{label})$ labeled by $\{F,V^{\#}\}$ as follows:
    \begin{itemize}
        \item $\mathcal V := W_1$, so that vertices are the words of the first kind,
        \item $\mathcal E^{\#}$ consists of all pairs $(w,w')$ where $w, w' \in W_1$ and $w$ is an immediate prefix of $w'$,
        \item for $E = (w,w') \in \mathcal E^{\#}$, we define 
        \begin{equation*}
            \mathrm{label}(E) = \begin{cases}
                F & \text{if } w' = wF;\\
                V^{\#} & \text{if } w' = wV^{\#}.
            \end{cases}
        \end{equation*}
    \end{itemize}
    We define $\Gamma(M,1\mathrm{st})$ as the opposite of the converse graph of $\Gamma(M,1\mathrm{st})^{\#}$ in the sense of Definitions \ref{DefinitionConverseGraph} and \ref{DefinitionOppositeGraph}.
\end{definition}

\sloppy In other words, $\Gamma(M,1\mathrm{st})$ is obtained from $\Gamma(M,1\mathrm{st})^{\#}$ by reversing all the $F$-arrows and by replacing the labels $V^{\#}$ with $V$. See Figure \ref{Figure5} for an example of a graph $\Gamma(M,1\mathrm{st})$ where $W_1$ consists of the words $\emptyset, F, V^{\#}, F^2, V^{\#}F, (V^{\#})^2, V^{\#}F^2, V^{\#}FV^{\#}, (V^{\#})^3$ (such examples of $M$ do exist). In order to distinguish between the labels of arrows and the names of the vertices, we wrote the latter in bold letters.

\begin{figure} 
\centering
\begin{tikzpicture}[
       decoration = {markings,
                     mark=at position .5 with {\arrow{Stealth[length=2mm]}}},
       dot/.style = {circle, fill, inner sep=2.4pt, node contents={},
                     label=#1},
every edge/.style = {draw, postaction=decorate}
                        ]

\node (e1) at (0,0) [dot,label=below:$\bm{\emptyset}$];
\node (e2) at (1,2) [dot,label=right:$\bm{F}$];
\node (e3) at (-1,2) [dot,label=left:$\bm{V^{\#}}$];
\node (e4) at (2,4) [dot,label=right:$\bm{F^2}$];
\node (e5) at (0,4) [dot,label=right:$\bm{V^{\#}F}$];
\node (e6) at (-2,4) [dot,label=left:$\bm{(V^{\#})^2}$];
\node (e7) at (1,6) [dot,label=right:$\bm{V^{\#}F^2}$];
\node (e8) at (-1,6) [dot,label=above:$\bm{V^{\#}FV^{\#}}$];
\node (e9) at (-3,6) [dot,label=left:$\bm{(V^{\#})^3}$];

\path (e2) edge node[right] {$F$} (e1);
\path (e4) edge node[right] {$F$} (e2);
\path (e5) edge node[right] {$F$} (e3);
\path (e7) edge node[right] {$F$} (e5);
\path (e1) edge node[left] {$V$} (e3);
\path (e3) edge node[left] {$V$} (e6);
\path (e5) edge node[left] {$V$} (e8);
\path (e6) edge node[left] {$V$} (e9);

\end{tikzpicture}
\caption{An example of graph $\Gamma(M,1\mathrm{st})$.}
\label{Figure5}
\end{figure}

\begin{theorem}[cf.~\cite{gelfand-ponomarev:1968} Theorem 4.1 and \cite{chai:kraft2025} Lemma 6.5.2]
    Assume that $W_1 \not = \emptyset$. The graph $\Gamma(M,1\mathrm{st})^{\#}$ is a connected tree with a unique root, and such that each vertex is the tail of at most two arrows and the head of at most one arrow. 
\end{theorem}

\begin{proof}
    First let us show that any prefix of a word of the first kind is again of the first kind. Namely, assume that $w_1 \preceq w_2$ and $w_2 \in W_1$. Let us write $w_2 = w_1w$ for some word $w \in W$. If we had $\theta D(V^{\#}w_1) = \mathbf 0D(Fw_1)$ then multiplying by $D(w)$ on the right would yield $\theta D(V^{\#}w_2) = \mathbf 0D(Fw_2)$, that is, $\mathrm{Dom}D(V^{\#}w_2) = \mathrm{Ker}D(Fw_2)$, which is a contradiction. Thus $w_1 \in W_1$ must be of the first kind. \\
    Since $W_1$ is not empty, it contains some word $w$ and also all of its prefixes. In particular, the empty word $\emptyset$ belongs to $W_1$, and constitutes the unique root of $\Gamma(M,1\mathrm{st})^{\#}$. Since the empty word is a prefix of every word, it also follows that $\Gamma(M,1\mathrm{st})^{\#}$ is connected. If $w \not = \emptyset$ is a word of the first kind, it has a unique immediate prefix, thus it is the head of a single arrow in $\Gamma(M,1\mathrm{st})^{\#}$. Besides, if an arrow has $w$ as its tail, its head is either equal to $wV^{\#}$ or to $wF$. Thus, $w$ is the tail of at most two arrows. Since each successor of $w$ has length strictly higher than the length of $w$, the graph $\Gamma(M,1\mathrm{st})^{\#}$ has no cycle hence is a tree.
\end{proof}

We point out that in general, $\Gamma(M,1\mathrm{st})$ is not a Kraft quiver. However we will explain later in this section how it can be decomposed into linear Kraft quivers. We now proceed to define a $(\sigma,\tau)$-linear representation $(U^{1\mathrm{st}},\rho^{1\mathrm{st}})$ on $\Gamma(M,1\mathrm{st})$. 

\begin{lemma}\label{Morphisms1stKind}
    Let $w_1, w_2 \in W_1$. 
    \begin{enumerate}
        \item Assume that $w_2 = w_1F$. Then $F$ induces a $\sigma$-linear injection
        \begin{equation*}
            \rho_{(w_2,w_1)}^{1\mathrm{st}}: \mathrm{Ker}D(Fw_2)/\mathrm{Dom} D(V^{\#}w_2) \hookrightarrow \mathrm{Ker}D(Fw_1)/\mathrm{Dom}D(V^{\#}w_1).
        \end{equation*}
        \item Assume that $w_2 = w_1V^{\#}$. Then $V$ induces a $\tau$-linear surjection
        \begin{equation*}
            \rho_{(w_1,w_2)}^{1\mathrm{st}}: \mathrm{Ker}D(Fw_1)/\mathrm{Dom} D(V^{\#}w_1) \twoheadrightarrow \mathrm{Ker}D(Fw_2)/\mathrm{Dom}D(V^{\#}w_2).
        \end{equation*}
    \end{enumerate}
\end{lemma}

\begin{proof}
    Assume first that $w_2 = w_1F$. It is easy to check that 
    \begin{align*}
        F^{-1}(\mathrm{Ker}D(Fw_1)) = \mathrm{Ker}D(Fw_2), & & F^{-1}(\mathrm{Dom}D(V^{\#}w_1)) = \mathrm{Dom}D(V^{\#}w_2).
    \end{align*}
    Indeed, for the first equality, we have $x \xrightarrow{D(Fw_2)} 0$ if and only if there exists some $y \in M$ such that $x \xrightarrow{F} y \xrightarrow{D(Fw_1)} 0$. This, in turn, is equivalent to $x \in F^{-1}(\mathrm{Ker}D(Fw_1))$. For the second equality, given $x \in M$, there exists some $z \in M$ with $x \xrightarrow{D(V^{\#}w_2)} z$ if and only if there exists some $y, z \in M$ such that $x \xrightarrow{F} y \xrightarrow{D(V^{\#}w_1)} z$. This, in turn, is equivalent to $x \in F^{-1}(\mathrm{Dom}D(V^{\#}w_1))$. It follows that the restriction of $F$ defines a $\sigma$-linear map $\mathrm{Ker}D(Fw_2) \rightarrow \mathrm{Ker}D(Fw_1)$, which induces an injective map upon taking the quotients on both sides.\\
    The case $w_2 = w_1V^{\#}$ is similar once we prove the equalities
    \begin{align*}
        V(\mathrm{Ker}D(Fw_1)) = \mathrm{Ker}D(Fw_2), & & V(\mathrm{Dom}D(V^{\#}w_1)) = \mathrm{Dom}D(V^{\#}w_2).
    \end{align*}
\end{proof}

\begin{definition}\label{DefinitionRepnFirstKind}
    We define a $(\sigma,\tau)$-linear representation $(U^{1\mathrm{st}},\rho^{1\mathrm{st}})$ on $\Gamma(M,1\mathrm{st})$ as follows:
    \begin{itemize}
        \item for every $w \in W_1$, $U^{1\mathrm{st}}_w := \mathrm{Ker}D(Fw)/\mathrm{Dom}D(V^{\#}w)$,
        \item for every arrow $E$, $\rho^{1\mathrm{st}}_E$ is the $\sigma$-linear (resp.~$\tau$-linear) morphism defined in Lemma \ref{MonomialOfFirstKind} when $\mathrm{label}(E) = F$ (resp.~$= V^{\#}$).
    \end{itemize}
\end{definition}

\begin{proposition}\label{Gr1stIsOfTheFirstKind}
    The module $\mathrm{gr}^{1\mathrm{st}}(M) := M(\Gamma(M,1\mathrm{st}),U^{1\mathrm{st}},\rho^{1\mathrm{st}})$ associated to $(U^{1\mathrm{st}},\rho^{1\mathrm{st}})$ in the sense of Definition \ref{DefinitionModuleAttachedToKraftQuiver} is a twisted Gelfand-Ponomarev module of the first kind.
\end{proposition}

\begin{proof}
    First, we must prove that $FV = VF = 0$ on $\mathrm{gr}^{1\mathrm{st}}(M)$. But this is obvious by construction, since the morphisms $\rho_{E}^{1\mathrm{st}}$ are induced by the operators $F$ or $V$ on $M$ depending on whether $E$ is an $F$-arrow or a $V$-arrow respectively. Then, we must prove that all the elementary intervals of $\mathrm{gr}^{1\mathrm{st}}(M)$ are of the first kind. Given a word $w \in W$, we write $D^{1\mathrm{st}}(w)$ for the monomial induced by $w$ on $\mathrm{gr}^{1\mathrm{st}}(M)$. For $w \in W_1$, let us show by induction on $l(w)$ that 
    \begin{align}\label{ElementaryIntervalsFirstKind}
          \mathrm{Ker}D^{1\mathrm{st}}(Fw) =  U^{1\mathrm{st}}_{w} \oplus \mathrm{Dom} D^{1\mathrm{st}}(V^{\#}w).
    \end{align}
    We note that the proof is formal, that is, it does not depend on the definition of $U^{1\mathrm{st}}_w$, but only on the shape of the graph $\Gamma(M,1\mathrm{st})$ and the injectivity/surjectivity of $\rho^{1\mathrm{st}}_E$. \\

    \textbf{Step 1:} We claim that for every $w\in W_1$, $\mathrm{Dom}D^{1\mathrm{st}}(w)$ is a direct sum of $U^{1\mathrm{st}}_v$'s for $v$ running through a certain subset of $W_1$. Indeed, first notice that for every subset $I \subseteq  W_1$, we have 
    \begin{align*}
        F^{-1}\left(\bigoplus_{v \in I} U^{1\mathrm{st}}_v\right) = \sum_{v \in I} F^{-1}(U^{1\mathrm{st}}_v), & & V\left(\bigoplus_{v \in I} U^{1\mathrm{st}}_v\right) = \bigoplus_{v \in I} V(U^{1\mathrm{st}}_v).
    \end{align*}
    The right equality is obvious since any vertex of $\Gamma(M,1\mathrm{st})$ is the head of at most one $V$-arrow, and the left equality holds because 
    \begin{align*}
    \mathrm{Im}(F) = \bigoplus_{\substack{v \in W_1 \text{ s.t.}\\ vF \in W_1}} \mathrm{Im}(\rho_{vF,v}^{1\mathrm{st}}), & & \mathrm{Im}(\rho_{vF,v}^{1\mathrm{st}}) \subseteq  U^{1\mathrm{st}}_v.
    \end{align*}
    Moreover, given any $v \in W_1$, we have 
    \begin{align*}
        F^{-1}(U^{1\mathrm{st}}_v) = \begin{cases}
            \mathrm{Ker}(F) & \text{if } vF \not \in W_1;\\
            \mathrm{Ker}(F) \oplus U^{1\mathrm{st}}_{vF} & \text{if } vF \in W_1,
        \end{cases}
        & &
        V(U^{1\mathrm{st}}_v) = \begin{cases}
            0 & \text{if } vV^{\#} \not \in W_1;\\
            U^{1\mathrm{st}}_{vV^{\#}} & \text{if } vV^{\#} \in W_1.
        \end{cases}
    \end{align*}
    Here, we use the surjectivity of $\rho^{1\mathrm{st}}_E$ whenever $E = (v,vV^{\#})$ is a $V$-arrow. Besides, we have 
    \begin{equation*}
        \mathrm{Ker}(F) = U^{1\mathrm{st}}_{\emptyset} \oplus \bigoplus_{\substack{v \in W_1 \text{ s.t.} \\ vV^{\#} \in W_1}} U^{1\mathrm{st}}_{vV^{\#}}.
    \end{equation*}
    Indeed, the injectivity of $\rho^{1\mathrm{st}}_E$ whenever $E = (vF,v)$ is an $F$-arrow forces the kernel of $F$ to be reduced to the sum of the $U^{1\mathrm{st}}_v$ for $v \in W_1$ which is not the tail of any $F$-arrow. Such $v$ correspond precisely to the empty word $\emptyset$ and to all words $v \not = \emptyset$ of the first kind whose last letter is a $V^{\#}$.\\
    Now, if $w = \emptyset$ is the empty word then $\mathrm{Dom}D^{1\mathrm{st}}(w) = \mathrm{Dom}(\mathbf 1) = \mathrm{gr}^{1\mathrm{st}}(M)$ is the sum of all the $U^{1\mathrm{st}}_v$'s, so the claim holds. Assume that it holds for all $w \in W_1$ with $l(w) = k$ for some $0 \leq k < m$, and let $w \in W_1$ with $l(w) = k+1$.\\
    Assume first that $w = w'F$ for some word $w'$ which is necessarily of the first kind. Thus, we have $\mathrm{Dom}D^{1\mathrm{st}}(w) = F^{-1}(\mathrm{Dom}D^{1\mathrm{st}}(w'))$. By induction, $\mathrm{Dom}D^{1\mathrm{st}}(w')$ is a direct sum of certain $U^{1\mathrm{st}}_v$'s. By our discussion above, it follows that $\mathrm{Dom}D^{1\mathrm{st}}(w)$ itself can be written as such a direct sum.\\
    \sloppy Assume now that $w = w'V^{\#}$ for some $w' \in W_1$. We have $\mathrm{Dom}D^{1\mathrm{st}}(w) = (V^{\#})^{-1}(\mathrm{Dom}D^{1\mathrm{st}}(w')) = V(\mathrm{Dom}D^{1\mathrm{st}}(w'))$. Again, by induction and the discussion above, the claim follows.\\

    \textbf{Step 2:} We prove \eqref{ElementaryIntervalsFirstKind} by induction on $l(w)$. First, observe that
    \begin{equation*}
        \mathrm{Dom}(V^{\#}) = \mathrm{Im}(V) = \bigoplus_{\substack{v \in W_1 \text{ s.t.} \\ vV^{\#} \in W_1}} U^{1\mathrm{st}}_{vV^{\#}}.
    \end{equation*}
     In particular, we see that 
    \begin{equation*}
        \mathrm{Ker}(F) =  U^{1\mathrm{st}}_{\emptyset} \oplus \mathrm{Dom}(V^{\#}),
    \end{equation*}
    which is precisely \eqref{ElementaryIntervalsFirstKind} for $w$ equal to the empty word. Now, let us assume that \eqref{ElementaryIntervalsFirstKind} is known for all $w \in W_1$ of length $k$ for some $0 \leq k < m$, and let $w \in W_1$ have length $k+1$. Assume first that $w = w'F$ for some word $w' \in W_1$. By induction, we have
    \begin{equation*}
        \mathrm{Ker}D^{1\mathrm{st}}(Fw) = F^{-1}(\mathrm{Ker}D^{1\mathrm{st}}(Fw')) = F^{-1}\left( U^{1\mathrm{st}}_{w'} \oplus \mathrm{Dom} D^{1\mathrm{st}}(V^{\#}w')\right).
    \end{equation*}
    By step 1, $F^{-1}$ preserves the sum on the right-hand side so that we have 
    \begin{equation*}
        \mathrm{Ker}D^{1\mathrm{st}}(Fw) = F^{-1}(U^{1\mathrm{st}}_{w'}) + F^{-1}(\mathrm{Dom} D^{1\mathrm{st}}(V^{\#}w')) = U_{w}^{1\mathrm{st}} + \mathrm{Dom}D^{1\mathrm{st}}(V^{\#}w).
    \end{equation*}
    \sloppy Here, we used the fact that $\mathrm{Dom}D^{1\mathrm{st}}(V^{\#}w)$ contains $\mathrm{Ker}(F)$. It remains to show that the sum is direct. So let $x \in U_{w}^{1\mathrm{st}} \cap \mathrm{Dom}D^{1\mathrm{st}}(V^{\#}w)$. Then $\rho_{w,w'}^{1\mathrm{st}}(x) \in U_{w'}^{1\mathrm{st}} \cap \mathrm{Dom} D^{1\mathrm{st}}(V^{\#}w') = \{0\}$ by induction. Since $\rho_{w,w'}^{1\mathrm{st}}$ is injective, we have $x = 0$ as desired. \\
    Assume now that $w = w'V^{\#}$ for some word $w' \in W_1$. By induction we have
    \begin{equation*}
        \mathrm{Ker}D^{1\mathrm{st}}(Fw) = V(\mathrm{Ker}D^{1\mathrm{st}}(Fw')) = V\left( U^{1\mathrm{st}}_{w'} \oplus \mathrm{Dom} D^{1\mathrm{st}}(V^{\#}w')\right).
    \end{equation*}
    By Step 1, $\mathrm{Dom} D^{1\mathrm{st}}(V^{\#}w')$ is a direct sum of some $U^{1\mathrm{st}}_v$'s. By induction, this sum does not involve $U^{1\mathrm{st}}_{w'}$. Thus, we have 
    \begin{equation*}
        \mathrm{Ker}D^{1\mathrm{st}}(Fw) = V(U^{1\mathrm{st}}_{w'}) \oplus V(\mathrm{Dom} D^{1\mathrm{st}}(V^{\#}w')) = U_{w}^{1\mathrm{st}} + \mathrm{Dom}D^{1\mathrm{st}}(V^{\#}w),
    \end{equation*}
    which concludes the proof of Step 2. \\

    Since $U^{1\mathrm{st}}_{w} \not = \{0\}$ for all $w \in W_1$, we deduce that all the monomials of the form $D^{1\mathrm{st}}(w)$ for $w \in W_1$ are monomials of the first kind in $\mathrm{gr}^{1\mathrm{st}}(M)$. Besides, if $w \in W_1$ then the corresponding elementary interval $(\beta_i^{1\mathrm{st}}, \beta_{i+1}^{1\mathrm{st}})$ of the first kind in $\mathrm{gr}^{1\mathrm{st}}(M)$ satisfies 
    \begin{equation*}
        \beta_{i+1}^{1\mathrm{st}}/\beta_i^{1\mathrm{st}} = U^{1\mathrm{st}}_{w}.
    \end{equation*}
    Since $\mathrm{gr}^{1\mathrm{st}}(M) = \bigoplus_{w \in W_1} U^{1\mathrm{st}}_{w}$, by dimensions we have already exhausted all the elementary intervals of $\mathrm{gr}^{1\mathrm{st}}(M)$. This concludes the proof.
\end{proof}

\begin{remark}
    It follows from the proof that $M$ and $\mathrm{gr}^{1\mathrm{st}}(M)$ determine the same set of words of the first kind. Moreover, the graph $\Gamma(\mathrm{gr}^{1\mathrm{st}}(M),1\mathrm{st})$ is naturally isomorphic to $\Gamma(M,1\mathrm{st})$. 
\end{remark}

Next, we explain how $\mathrm{gr}^{1\mathrm{st}}(M)$ can be decomposed as a sum of modules coming from connected linear Kraft quivers. This is also explained in \cite{gelfand-ponomarev:1968} p.48-50. Let us start with a corollary of the proof of Proposition \ref{Gr1stIsOfTheFirstKind}.

\begin{corollary}\label{CorolAllLinearKraftAreFirstKind}
    Let $\Gamma = (\mathcal V,\mathcal E,\mathrm{label})$ be a connected linear Kraft quiver together with a strict $(\sigma,\tau)$-linear representation $(U,\rho)$. The twisted Gelfand-Ponomarev module $M(\Gamma,U,\rho)$ is of the first kind.
\end{corollary}

\begin{proof}
    In this proof, we consider $M = M(\Gamma,U,\rho)$ as in the statement. If $n := \#\mathcal V = 1$, the statement is obvious. Indeed, we have $F=V=0$ on $M$ so that $\mathrm{Ker}(F) = M$ and $\mathrm{Dom}(V^{\#}) = \{0\}$. Thus there is only one elementary interval, it is of the first kind and associated to the empty word. Assume now that $n\geq 2$. Let $w$ denote the unique word such that $\Gamma \simeq \Gamma(w)$, as in Proposition \ref{WordOfKraftQuiver}. We have $l(w) = n-1 \geq 1$. Let also $v_1, \ldots , v_n$ be the ordering of the vertices of $\Gamma$ induced from Definition \ref{DefLinearKraftQuiverWithWord} and the isomorphism $\Gamma \simeq \Gamma(w)$. Eventually, let us write $w = w_{n-1}\cdots w_{2}w_{1}$, so that each $w_i$ is an element of $\{F,V^{\#}\}$. Consider the set 
    \begin{equation*}
        \mathcal V' := \{w_{n-1}\cdots w_{n-i+1}, 1 \leq i \leq n\}.
    \end{equation*}
    Thus, $\mathcal V'$ is the set of prefixes of $w$. Let us write $u_i := w_{n-1}\cdots w_{n-i+1}$. Moreover, for each $1 \leq i \leq n-1$, we define 
    \begin{equation*}
        E_i' := \begin{cases}
            (u_i,u_{i+1}) & \text{if } w_{n-i} = V^{\#};\\
            (u_{i+1},u_i) & \text{if } w_{n-i} = F,
        \end{cases}
    \end{equation*}
    and $\mathcal E' := \{E_i', 1 \leq i \leq n-1\}$. Eventually, put the label $F$ on all arrows $E_i'$ of the form $(u_{i+1},u_i)$, and the label $V$ on all the arrows $E_i'$ of the form $(u_i,u_{i+1})$. This defines a connected linear Kraft quiver $\Gamma'$, and it is easy to see that the map $v_i \mapsto u_{n-i+1}$ defines an isomorphism $\Gamma \simeq \Gamma'$. We may consider $(U,\rho)$ as a strict $(\sigma,\tau)$-linear representation over $\Gamma'$.\\
    The arguments in Step 1 and Step 2 of the proof of Proposition \ref{Gr1stIsOfTheFirstKind} still hold for $\Gamma'$ equipped with $(U,\rho)$, and show that for every $1 \leq i \leq n$,
    \begin{equation*}
        \mathrm{Ker} D(Fu_i) = U_{u_i} \oplus \mathrm{Dom} D(V^{\#}u_i).
    \end{equation*}
    This shows that $u_i$ is a word of the first kind for $M$, and that the corresponding elementary interval $(\beta_j,\beta_{j+1})$ satisfies $\beta_{j+1}/\beta_j \simeq U_{u_i}$. By dimension, these must exhaust all the elementary intervals of $M$, implying that $M$ is of the first kind. 
\end{proof}

\begin{remark}
    By construction, the Kraft quiver $\Gamma'$ defined in the proof is identic to $\Gamma(M,1\mathrm{st})$. In particular, $\Gamma(M,1\mathrm{st})$ recovers the original connected linear Kraft quiver $\Gamma$ when $M$ is associated to $\Gamma$.\\
    Conversely, a twisted Gelfand-Ponomarev module $M$ comes from a connected linear Kraft quiver if and only if $\Sigma_1$ is totally ordered under the prefix ordering relation $\preceq$ defined at the beginning of this section.
\end{remark}

\begin{proposition}[cf.~\cite{gelfand-ponomarev:1968} Theorem 4.5]\label{DecompositionSumLinearKraftQuiver}
    The twisted Gelfand-Ponomarev module $\mathrm{gr}^{1\mathrm{st}}(M)$ of the first kind can be decomposed as a finite direct sum of modules of the form $M(\Gamma_i,U^i,\rho^i)$ for certain linear connected Kraft quivers $\Gamma_i$ which are pairwise non-isomorphic, together with strict $(\sigma,\tau)$-linear representations $(U^i,\rho^i)$. 
\end{proposition}

\begin{proof}
    For every $w \in W_1$, we construct a compatible decomposition as follows
    \begin{equation}\label{EqnDecomposition}
        U_{w}^{1\mathrm{st}} = X_w^F \oplus X_w^V \oplus X_w.
    \end{equation}
    Recall that $m$ denotes the maximal length of the elements in $W_1$. If $l(w) = m$, set $X_w^F = X_w^V = \{0\}$ and $X_w := U_{w}^{1\mathrm{st}}$. Now, let us assume that the decomposition \eqref{EqnDecomposition} is already built for all words $w \in W_1$ with $l(w) \geq k+1$ for a certain $0 \leq k \leq m-1$. Let $w \in W_1$ with $l(w) = k$.\\
    If $wF \in W_1$, since $\rho_{wF,w}^{1\mathrm{st}}$ is injective, it defines an isomorphism onto its image. We set $X_{w}^F := \mathrm{Im}(\rho_{wF,w}^{1\mathrm{st}}) \subseteq  U_{w}^{1\mathrm{st}}$. If $wF \not \in W_1$, we set $X_w^F := \{0\}$. If $wV^{\#}\in W_1$, then $\rho_{wV^{\#},w}^{1\mathrm{st}}$ is surjective. We choose any supplementary subspace to its kernel in $U_{w}^{1\mathrm{st}}$, and call it $X_w^V$. In particular $\rho_{wV^{\#},w}^{1\mathrm{st}}$ defines an isomorphism from $X_w^V$ onto $U_{wV^{\#}}^{1\mathrm{st}}$. If $wV^{\#} \not \in W_1$, we set $X^V_w := \{0\}$. By construction, we have $X_w^F \cap X_w^V = \{0\}$. Eventually, we consider any supplementary subspace to $X_w^F \oplus X_w^V$ in $U_{w}^{1\mathrm{st}}$ and call it $X_w$.\\
    By induction, the decomposition \eqref{EqnDecomposition} is well-defined for all $w\in W_1$. Let us consider $w \in W_1$ such that $X_w \not = \{0\}$. Such $w$ always exist as soon as $W_1 \not = \emptyset$, since any $w \in W_1$ with $l(w) = m$ will do. Let $\Gamma_w$ denote the subgraph of $\Gamma(M,1\mathrm{st})$ which is cut out by all the prefixes of $w$. Then $\Gamma_w$ is a connected linear Kraft quiver which is isomorphic to the linear Kraft quiver $\Gamma(w)$ of Definition \ref{DefLinearKraftQuiverWithWord}; see also the proof of Corollary \ref{CorolAllLinearKraftAreFirstKind} where $\Gamma'$ there corresponds to $\Gamma_w$ here. We define a strict $(\sigma,\tau)$-linear representation $(U^w,\rho^w)$ on $\Gamma_w$ as follows. Set $U^w_w := X_w$, and assume that the space $U^w_{w'} \subseteq  U_{w'}^{1\mathrm{st}}$ is already defined for some non-empty prefix $w'$ of $w$. Let $w''$ denote the immediate prefix of $w'$. If $w' = w''F$, then define $U_{w''}^w$ as the image of $U^w_{w'}$ by $\rho_{w',w''}^{1\mathrm{st}}$, and $\rho^w_{w',w''}$ as the restriction of $\rho_{w',w''}^{1\mathrm{st}}$ to $U^w_{w'}$. If $w' = w''V^{\#}$, then define $U_{w''}^w$ as the preimage of $U^w_{w'}$ via the isomorphism $X_{w''}^V \xrightarrow{\sim} U_{w'}^{1\mathrm{st}}$, and $\rho^w_{w'',w'}$ as the restriction to $U^w_{w''}$ of this isomorphism. Clearly, $(U^w,\rho^w)$ defines a strict $(\sigma,\tau)$-linear representation of $\Gamma_w$.\\
    Since the isomorphism class of a connected linear Kraft quiver is determined by its associated finite word, the Kraft quivers $\Gamma_w$, for $w$ running through all the words of the first kind with $X_w \not = \{0\}$, are pairwise non-isomorphic. Moreover, by construction, the twisted Gelfand-Ponomarev modules $M_w := M(\Gamma_w,U^w,\rho^w)$ are naturally submodules of $\mathrm{gr}^{1\mathrm{st}}(M)$. We claim that $\mathrm{gr}^{1\mathrm{st}}(M)$ is equal to the direct sum of these modules.\\
    First, by induction on $l(v)$, let us prove that $U_{v}^{1\mathrm{st}}$ is included in the sum of the $M_w$'s. If $l(v) = m$ is maximal, then $U_{v}^{1\mathrm{st}} = X_v \subseteq  M_v$ by construction, so this case holds. Assume that the statement is known for all $v \in W_1$ with $l(v) = k+1$ for some $0 \leq k < m$, and let $v \in W_1$ with $l(v) = k$. Recall the decomposition \eqref{EqnDecomposition} of $U_{v}^{1\mathrm{st}}$. If $X_v \not = \{0\}$ then clearly $X_v \subseteq  M_v$. Besides, if $vF \in W_1$ then $\rho_{vF,v}^{1\mathrm{st}}$ induces an isomorphism from $U_{vF}^{1\mathrm{st}}$ to $X_v^F$. By induction, $U_{vF}^{1\mathrm{st}}$ is included in the sum of all the $M_w$'s, for which $w$ has $vF$ as prefix and $X_w \not = \{0\}$. By construction, for each such $w$, the Frobenius $F$ induces on $U_{vF}^{1\mathrm{st}}\cap M_w$ the restriction of $\rho_{vF,v}^{1\mathrm{st}}$. In particular, the image of $\rho_{vF,v}^{1\mathrm{st}}$ remains in the sum of these $M_w$'s, which proves the claim. Similarly, one may prove that $X_v^V$ is in the sum of all the $M_w$'s for which $w$ has $vV^{\#}$ as prefix and $X_w \not = \{0\}$. This proves the statement by induction. \\
    Thus, $\mathrm{gr}^{1\mathrm{st}}(M)$ is equal to the sum of all the $M_w$ for $w\in W_1$ with $X_w \not = \{0\}$. To prove that the sum is direct, it is enough to prove that $\dim\mathrm{gr}^{1\mathrm{st}}(M) = \sum \dim M_w$. Let us write 
    \begin{align*}
        d_w^{1\mathrm{st}} := \dim U_{w}^{1\mathrm{st}}, & & d^F_w := \dim X_w^F, & & d^V_w := \dim X_w^V, & & d_w := \dim X_w.
    \end{align*}
    Thus, we have $\dim \mathrm{gr}^{1\mathrm{st}}(M) = \sum_{w \in W_1} d_w^{1\mathrm{st}}$ and $d_w^{1\mathrm{st}} = d_w^F + d_w^F + d_w$ for all $w\in W_1$. Besides, we have 
    \begin{align*}
        d_w^F = \begin{cases}
            d_{wF}^{1\mathrm{st}} & \text{if } wF \in W_1;\\
            0 & \text{else},
        \end{cases}
        & & 
        d_w^V = \begin{cases}
            d_{wV^{\#}}^{1\mathrm{st}} & \text{if } wV^{\#} \in W_1;\\
            0 & \text{else}.
        \end{cases}
    \end{align*}
    Using these relations, the sum $\sum_{w \in W_1} d_w^{1\mathrm{st}}$ can be expressed as a sum involving only the integers $d_w$ with some positive multiplicities. Given a $w \in W_1$ with $d_w > 0$, each prefix of $w$ will contribute to one factor $d_w$ in the sum. In other words, we have 
    \begin{equation*}
        \sum_{w \in W_1} d_w^{1\mathrm{st}} = \sum_{\substack{w \in W_1 \text{ s.t.} \\ d_w > 0}} (l(w)+1)d_w = \sum_{\substack{w \in W_1 \text{ s.t.} \\ d_w > 0}} \dim M_w.
    \end{equation*}
    Thus the sum of the $M_w$'s is direct and the proof is over.
\end{proof}

Eventually, we prove that there exists a submodule $M_1 \subseteq  M$ which is isomorphic to $\mathrm{gr}^{1\mathrm{st}}(M)$. This is explained in \cite{gelfand-ponomarev:1968} Theorem 4.2. 

\begin{lemma}[cf.~\cite{gelfand-ponomarev:1968} Theorem 4.2 and Lemma 4.1]
    There exists a family of $K$-linear subspaces $\gamma(w) \subseteq  M$ for each $w \in W$, such that $\gamma(w) = 0$ if $w \not \in W_1$ and for every $w \in W_1$ we have 
    \begin{equation*}
        \mathrm{Ker}D(Fw) = \mathrm{Dom}D(V^{\#}w) \oplus \gamma(w),
    \end{equation*}
    and such that the action of $F$ induces a $\sigma$-linear injection $\gamma(wF) \hookrightarrow \gamma(w)$, and the action of $V$ induces a $\tau$-linear surjection $\gamma(w) \twoheadrightarrow \gamma(wV^{\#})$. 
\end{lemma}

In other words, each $\gamma(w)$ is a supplementary subspace to $\mathrm{Dom}D(V^{\#}w)$ in $\mathrm{Ker}D(Fw)$, chosen compatibly so that they fit a diagram as below.
\begin{center}
    \begin{tikzcd}
        \gamma(wV^{\#})  & & \gamma(wF) \\
        & \gamma(w) \arrow[ul, twoheadrightarrow, swap, "V"] \arrow[ur, hookleftarrow, "F"] &            
    \end{tikzcd}
\end{center}

\begin{proof}
We set $\gamma(w) = 0$ for all $w\in W \setminus W_1$. We define $\gamma(w)$ for $w \in W_1$ by induction on $l(w)$. If $l(w) = m$ is maximal, we can choose $\gamma(w)$ as any $K$-linear supplementary subspace to $\mathrm{Dom}D(V^{\#}w)$ in $\mathrm{Ker}D(Fw)$. Let us now assume that all the subspaces $\gamma(w)$ have been built for $w \in W_1$ with $l(w) = k+1$ for a certain $0 \leq k < m$. Let us consider $w \in W_1$ with $l(w) = k$. We will build $\gamma(w)$ as the direct sum of three suitably chosen subspaces 
\begin{equation*}
    \gamma(w) := \gamma^V(w) \oplus \gamma^{\circ}(w) \oplus \gamma^F(w).
\end{equation*}
Recall that $F^{-1}(\mathrm{Ker} D(Fw)) = \mathrm{Ker}D(FwF)$ and $F^{-1}(\mathrm{Dom} D(V^{\#}w)) = \mathrm{Dom} D(V^{\#}wF)$. This implies that the restriction of $F$ to $\gamma(wF)$ is injective. Indeed, if we have $F(x) = 0$ for some $x \in \gamma(wF)$, then in particular $F(x) \in \mathrm{Dom} D(V^{\#}w)$ so that $x \in \mathrm{Dom} D(V^{\#}wF)$. Since $\mathrm{Dom} D(V^{\#}wF) \cap \gamma(wF) = \{0\}$, we have $x = 0$. Thus, $F$ restricts to an isomorphism from $\gamma(wF)$ to its image which we denote $\gamma^F(w) \subseteq  \mathrm{Ker}D(Fw)$. On the other hand, recall that $V(\mathrm{Ker} D(Fw)) = \mathrm{Ker}D(FwV^{\#})$ and $V(\mathrm{Dom}D(V^{\#}w)) = \mathrm{Dom}D(V^{\#}wV^{\#})$. We choose a supplementary subspace to $\mathrm{Ker}(V) \cap \mathrm{Ker} D(Fw)$ in $V^{-1}(\gamma(wV^{\#})) \cap \mathrm{Ker} D(Fw)$ and call it $\gamma^V(w)$. By construction, $V$ induces an isomorphism from $\gamma^V(w)$ to $\gamma(wV^{\#})$.\\
Before going further, we introduce the subspace $\omega := \mathrm{Dom}D(V^{\#}w) + (\mathrm{Ker}(V) \cap \mathrm{Ker} D(Fw))$. Let us prove that
\begin{equation}\label{EqnDecompoOmega}
    \mathrm{Ker} D(Fw) = \gamma^V(w) \oplus \omega.
\end{equation} 
If $x \in \gamma^V(w) \cap \omega$, we can write $x = x_1 + x_2$ where $x_1 \in \mathrm{Dom}D(V^{\#}w)$ and $x_2 \in \mathrm{Ker}(V) \cap \mathrm{Ker} D(Fw)$. It follows that $V(x) = V(x_1) \in \mathrm{Dom}D(V^{\#}wV^{\#}) \cap \gamma(wV^{\#}) = \{0\}$. Thus $x \in \mathrm{Ker}(V) \cap \gamma^V(w) = \{0\}$ as desired. Next, let $x\in \mathrm{Ker} D(Fw)$. We can decompose $V(x) = y_1 + y_2$ with $y_1 \in \mathrm{Dom}D(V^{\#}wV^{\#})$ and $y_2 \in \gamma(wV^{\#})$. Let $x_1 \in \mathrm{Dom} D(V^{\#}w) \subseteq  \omega$ and $x_2 \in \gamma^V(w)$ be preimages respectively of $y_1$ and $y_2$ under $V$. Then $z := x - x_1 - x_2 \in \mathrm{Ker}(V) \cap \mathrm{Ker} D(Fw) \subseteq  \omega$. Thus, we have $x = x_2 + (z + x_1)$ with $x_2 \in \gamma^V(w)$ and $z+x_1 \in \omega$ as desired. This proves \eqref{EqnDecompoOmega}.\\
Now, observe that $\gamma^F(w) \cap \mathrm{Dom}D(V^{\#}w) = \{0\}$. Indeed, if $x$ belongs to this intersection, then there is a unique $x' \in \gamma(wF)$ such that $F(x') = x$. But then $x' \in \mathrm{Dom} D(V^{\#}wF)$ which implies that $x' = 0$, so $x=0$ as desired. \\
We define a subspace $\gamma^{\circ}(w) \subseteq  \mathrm{Ker}D(Fw)$ as follows. Observe that $\gamma^F(w) \oplus \mathrm{Dom}D(V^{\#}w) \subseteq  \omega$ (since $VF = 0$ on $M$). We can choose a supplementary $\gamma^{\circ}(w)$ to $\gamma^F(w) \oplus \mathrm{Dom}D(V^{\#}w)$ in $\omega$, such that $\gamma^{\circ}(w) \subseteq  \mathrm{Ker}(V) \cap \mathrm{Ker} D(Fw)$. Putting things together, we have 
\begin{align*}
    \mathrm{Ker} D(Fw) & = \gamma^V(w) \oplus \omega \\
    & = \gamma^V(w) \oplus \gamma^{\circ}(w) \oplus \gamma^F(w) \oplus \mathrm{Dom}D(V^{\#}w).
\end{align*}
We define 
\begin{equation*}
    \gamma(w) := \gamma^V(w) \oplus \gamma^{\circ}(w) \oplus \gamma^F(w).
\end{equation*}
By construction, it is clear that $\gamma(w)$ meets all the requirements of the Lemma.
\end{proof}

The spaces $\gamma(w)$ for $w \in W_1$, together with the transition maps induced by $F$ and $V$, form a $(\sigma,\tau)$-linear representation on $\Gamma(M,1\mathrm{st})$ which we denote $\gamma$. By construction, the associated module $M_1 := M(\Gamma(M,1\mathrm{st}),\gamma)$ is naturally a submodule of $M$, namely $M_1$ is the internal direct sum of the $\gamma(w)$ for $w \in W_1$. Moreover, for every $w \in W_1$ the quotient map gives a $K$-linear isomorphism 
\begin{equation*}
    \gamma(w) \xrightarrow{\sim} U_{D(w)}^{1\mathrm{st}}.
\end{equation*}
This induces an isomorphism $\gamma \xrightarrow{\sim} (U^{1\mathrm{st}},\rho^{1\mathrm{st}})$ of $(\sigma,\tau)$-linear representations over $\Gamma(M,1\mathrm{st})$. Thus, we have proved the following.

\begin{theorem}\label{TheoremExistenceM1}
    There is a twisted Gelfand-Ponomarev submodule $M_1$ of $M$ which is isomorphic to $\mathrm{gr}^{1\mathrm{st}}(M)$. 
\end{theorem}

Clearly, $M$ is of the first kind if and only if $M = M_1$. Besides, if $W_1 = \emptyset$, namely if $M$ has no elementary interval of the first kind, we set $M_1 = \{0\}$. We end this section with the following Proposition.

\begin{proposition}
    Let $M, M'$ be two twisted Gelfand-Ponomarev modules of the first kind. Let $(U^{1\mathrm{st}},\rho^{1\mathrm{st}})$ and $((U')^{1\mathrm{st}},(\rho')^{1\mathrm{st}})$ be the $(\sigma,\tau)$-linear representations of Definition \ref{DefinitionRepnFirstKind} for $M$ and for $M'$ respectively. Then $M$ and $M'$ are isomorphic if and only if they determine the same set $W_1$ of words of the first kind, and for every $w \in W_1$ we have 
    \begin{equation*}
        \dim U_{w}^{1\mathrm{st}} = \dim (U')_{w}^{1\mathrm{st}}.
    \end{equation*}
\end{proposition}

\begin{proof}
    The direct implication is clear. Let us prove the converse. Since $M$ and $M'$ have the same set of words of the first kind, the graphs $\Gamma(M,1\mathrm{st})$ and $\Gamma(M',1\mathrm{st})$ are naturally identified. For $w\in W_1$, let us define integers $(d')_w^{1\mathrm{st}}$, $(d')_w^F$, $(d')_w^V$ and $d'_w$ exactly as in the proof of Proposition \ref{DecompositionSumLinearKraftQuiver} but relative to $M'$. By assumption, we have $d_w^{1\mathrm{st}} = (d')_w^{1\mathrm{st}}$ for all $w \in W_1$. Starting with $l(w) = m$, it is easy to prove by induction on $l(w)$ that the integers $(d')_w^F$, $(d')_w^V,$ and $d'_w$ are equal to $d_w^F$, $d_w^V$ and $d_w$ respectively for all $w \in W_1$. In particular, the connected linear Kraft quivers involved in the decompositions of $\mathrm{gr}^{1\mathrm{st}}(M)$ and of $\mathrm{gr}^{1\mathrm{st}}(M')$ are the same, and the strict $(\sigma,\tau)$-linear representations on them have the same dimension. By Proposition \ref{ModuleAttachedToLinearKraftQuiver}, the modules attached to these representations of Kraft quivers are isomorphic. Therefore, $M = M_1 \simeq \mathrm{gr}^{1\mathrm{st}}(M)$ is isomorphic to $\mathrm{gr}^{1\mathrm{st}}(M') \simeq M_1' = M'$.
\end{proof}

In the next section, we are going to build a supplementary submodule $M_2$ to $M_1$ in $M$ such that $M_2$ is of the second kind. 

\subsection{Submodules of the second kind and circular Kraft quivers}

In this section, we assume that $M \not = M_1$, so that there is at least one elementary interval of the second kind in $M$. We start with explaining how one can associate a unique periodic infinite word $[w]$ in the sense of Section \ref{SubsectionWords} to any elementary interval of the second kind. 

\begin{lemma}\label{ShapeOfIntervalsOfSecondKind}
    Let $(\beta_i,\beta_{i+1})$ be an elementary interval of the second kind. There exists some $w \in W$ such that 
    \begin{equation*}
        (\beta_i,\beta_{i+1}) = (\mathrm{Ker}D(w), \mathrm{Dom}D(w)).
    \end{equation*}
    Moreover, if $w' \in W$ is another word such that $(\beta_i,\beta_{i+1}) = (\mathrm{Ker}D(w'), \mathrm{Dom}D(w'))$, then one of $w$ and $w'$ is a suffix of the other.
\end{lemma}

\begin{proof}
    First, we note that for any word $w \in W$, we have 
    \begin{align}\label{UsefulIdentityKerAndDom}
        \mathrm{Ker}D(V^{\#}w) = \mathrm{Ker}D(w), & & \mathrm{Dom} D(Fw) = \mathrm{Dom} D(w).
    \end{align}
    Indeed, this is due to $D(F) = \Gamma_F$ and $D(V) = \Gamma_V$ being the graphs of some semilinear maps on $M$. By definition of the stabilized sequence $\mathcal F_M$, we know that $\beta_i$ and $\beta_{i+1}$ can be expressed as the kernel or the domain of some monomial. \\
    Assume that $\beta_i = \mathrm{Dom}D(w)$ for some word $w$. Since $\beta_i \subsetneq M$, the word $w$ is not empty. By \eqref{UsefulIdentityKerAndDom}, we can assume that $w = V^{\#}w'$ for some $w' \in W$. By \eqref{CompletionInterval}, we have 
    \begin{equation*}
        \beta_i = \mathrm{Dom}D(V^{\#}w') \subseteq  \mathrm{Ker}D(Fw').
    \end{equation*}
    If this inclusion was proper, then Proposition \ref{MonomialOfFirstKind} tells us that the interval $(\beta_i, \mathrm{Ker}D(Fw'))$ is elementary of the first kind. Necessarily $\beta_{i+1} = \mathrm{Ker}D(Fw')$ and this contradicts our assumption. Thus, we have $\beta_i = \mathrm{Ker}D(Fw')$. Therefore if $\beta_i$ is the domain of some monomial, it can always be expressed as the kernel of another monomial of the same length. Likewise, one proves that if $\beta_{i+1}$ is the kernel of some monomial, then it can always be written as the domain of another monomial of the same length.\\
    Let $w$ (resp.~$v$) denote the word of shortest length such that $\beta_i = \mathrm{Ker}D(w)$ (resp.~$\beta_{i+1} = \mathrm{Dom}D(v)$). The minimality implies that the monomials $D(w)$ and $D(v)$ are non-null. Indeed, for instance if $D(w)$ were to be null, we would have $\beta_i = \mathrm{Dom}D(w)$. This would force $w$ to be non empty, and by \eqref{UsefulIdentityKerAndDom} the minimality of $w$ implies that $w = Fw'$ for some $w' \in W$. But then $\beta_i = \mathrm{Dom}D(w') = \mathrm{Ker}D(w'')$ for some other word $w'' \in W$ with $l(w'') = l(w')$, which is again a contradiction. A similar reasoning goes for $D(v)$. Thus, we have 
    \begin{equation*}
        \beta_i \subsetneq \mathrm{Dom}D(w).
    \end{equation*}
    Since $(\beta_i,\beta_{i+1})$ is elementary, we must have $\beta_i \subsetneq \beta_{i+1} \subseteq  \mathrm{Dom}D(w)$. Likewise, the proper inclusion $\mathrm{Ker}D(v) \subsetneq \beta_{i+1}$ implies that $\mathrm{Ker}D(v) \subseteq  \beta_i \subsetneq \beta_{i+1}$. In other words, we have
    \begin{equation*}
        \mathrm{Ker}D(v) \subseteq  \underbrace{\mathrm{Ker}D(w)}_{= \beta_i} \subsetneq \underbrace{\mathrm{Dom}D(v)}_{= \beta_{i+1}} \subseteq  \mathrm{Dom} D(w).
    \end{equation*}
    By Lemma \ref{LemmaRelativePositionIntervals}, this is only possible if $\beta_i = \mathrm{Ker}D(v)$, in which case $w$ is a suffix of $v$, or $\beta_{i+1} = \mathrm{Dom} D(w)$, in which case $v$ is a suffix of $w$. In both cases the first part of the Lemma is proved.\\
    Let us now assume that 
    \begin{equation*}
        (\beta_i,\beta_{i+1}) = (\mathrm{Ker}D(w), \mathrm{Dom}D(w)) = (\mathrm{Ker}D(w'), \mathrm{Dom}D(w')).
    \end{equation*}
    Since both $D(w)$ and $D(w')$ are non-null, Lemma \ref{LemmaRelativePositionIntervals} implies that $w$ or $w'$ is a suffix of the other as desired.
\end{proof}

\begin{remark}
    We point out that the converse of Lemma \ref{ShapeOfIntervalsOfSecondKind} does not hold. Namely, the fact that an elementary interval can be written in the form $(\beta_i,\beta_{i+1}) = (\mathrm{Ker}D(w), \mathrm{Dom}D(w))$ for some $w\in W$, does not necessarily imply that it is of the second kind. For instance, if $M = M(\Gamma,\mathbf 1_{\Gamma})$ where $\Gamma = \Gamma(V^{\#}F)$, one may check that
    \begin{align*}
        \mathrm{Ker}D(FV^{\#}) = \mathrm{Dom}D(V^{\#}), & & \mathrm{Dom}D(V^{\#}V^{\#}) = \mathrm{Ker}D(V^{\#}) = \{0\}.
    \end{align*}
    Therefore, the elementary interval of the first kind determined by $V^{\#} \in W_1$ can also be written as the kernel and domain of $V^{\#}$.
\end{remark}

Recall our definitions and notations regarding infinite words, cf. Definition \ref{DefinitionInfiniteWords}. If $\tilde w$ is an infinite word and $j\geq 0$, we write $\tilde w\{j\} := w_jw_{j-1}\cdots w_1 \in W$ for the $j$-th suffix of $\tilde w$, which is a finite word of length $j$. In particular we have 
    \begin{equation*}
        \tilde w = \tilde w[j] \tilde w\{j\},
    \end{equation*}
    for all $j\geq 0$.

\begin{proposition}\label{PropositionWordsOfSecondKind}
    Let $(\beta_i,\beta_{i+1})$ be an elementary interval of the second kind. There exists a unique infinite word $\tilde w$ and an integer $j_0 \geq 0$ such that for all $j\geq j_0$, 
    \begin{equation*}
        (\beta_i,\beta_{i+1}) = (\mathrm{Ker} D(\tilde w\{j\}), \mathrm{Dom} D(\tilde w\{j\})).
    \end{equation*}
\end{proposition}

\begin{proof}
    Let $w \in W$ a word such that $(\beta_i,\beta_{i+1}) = (\mathrm{Ker}D(w), \mathrm{Dom}D(w))$. By Lemma \ref{LemmaRelativePositionIntervals}, we have a diagram as follows.
    \begin{center}
        \begin{tikzcd}[sep = small,cramped]
            & \mathrm{Ker}D(Fw) \arrow[r,phantom,"\subseteq "] & \mathrm{Dom}D(Fw) \arrow[rd,phantom,sloped,"="] & \\
            \beta_i \arrow[ru,phantom,sloped,"\subseteq "] \arrow[rd,phantom,sloped,"="] & & & \beta_{i+1} \\
            & \mathrm{Ker}D(V^{\#}w) \arrow[r,phantom,"\subseteq "] & \mathrm{Dom}D(V^{\#}w) \arrow[ru,phantom,sloped,"\subseteq "] &
        \end{tikzcd}
    \end{center}
    The equalities follows from \eqref{UsefulIdentityKerAndDom}. We claim that exactly one of $D(Fw)$ and $D(V^{\#}w)$ is non-null. Indeed, assume first that both are null. The diagram implies that $\beta_i = \mathrm{Dom}D(V^{\#}w)$ and $\beta_{i+1} = \mathrm{Ker}D(Fw)$. This would make $(\beta_i,\beta_{i+1})$ an elementary interval of the first kind, which is absurd. Assume now that both are non-null. Since $(\beta_i,\beta_{i+1})$ is elementary, the diagram implies that $\beta_i = \mathrm{Ker}D(Fw)$ and $\beta_{i+1} = \mathrm{Dom}D(V^{\#}w)$. But this would imply that $\beta_{i+1} \subseteq  \beta_i$, which is absurd. This proves the claim.\\
    If $D(Fw)$ is non null, then $D(V^{\#}w)$ is null and the diagram implies that $\beta_i = \mathrm{Ker}D(Fw)$ and $\beta_{i+1} = \mathrm{Dom}D(Fw)$. If $D(V^{\#}w)$ is non null, then $D(Fw)$ is null and the diagram implies that $\beta_i = \mathrm{Ker}D(V^{\#}w)$ and $\beta_{i+1} = \mathrm{Dom}D(V^{\#}w)$. Thus, there is a single word $w' \in  W$ of length $l(w') = l(w) + 1$ and with $w$ as a suffix, such that $(\beta_i,\beta_{i+1}) = (\mathrm{Ker}D(w'), \mathrm{Dom}D(w'))$. By repeating this process, we obtain a well-defined infinite word $\tilde w$ such that, for every $j \geq l(w) =: j_0$, we have $(\beta_i,\beta_{i+1}) = (\mathrm{Ker} D(\tilde w\{j\}), \mathrm{Dom} D(\tilde w\{j\}))$ as desired.\\
    We claim that the infinite word $\tilde w$ does not depend on the initial choice of $w$. Let $w'$ be another starting word, that is, such that $(\beta_i,\beta_{i+1}) = (\mathrm{Ker}D(w'), \mathrm{Dom}D(w'))$. Without loss of generality, let us say that $l(w) \leq l(w')$. With $j = l(w')$, we know that $(\beta_i,\beta_{i+1}) = (\mathrm{Ker} D(\tilde w\{j\}), \mathrm{Dom} D(\tilde w\{j\}))$. By Lemma \ref{ShapeOfIntervalsOfSecondKind}, we know that one of $w'$ or $w\{j\}$ must be a suffix of the other. Since both words have the same length, they must actually be equal. Then it is clear that $\tilde{w}' = \tilde w$.
\end{proof}

\begin{definition}
    An infinite word $\tilde w$ arising from an elementary interval of the second kind is called \textit{a word of the second kind}. We denote by $W_2$ the set of words of the second kind. 
\end{definition}

It follows that $W_2$ is a finite set, in bijection with the set of elementary intervals of the second kind. Beware that $W_2 \not \subseteq W$, since $W_2$ consists of infinite words by definition. 

\begin{lemma}[cf.~\cite{gelfand-ponomarev:1968} Theorem 5.1]\label{AllShiftsAreWordsOfSecondKind}
    Let $\tilde w \in W_2$ be a word of the second kind. Then for every $j \geq 0$, $\tilde w[j]$ is again a word of the second kind.
\end{lemma}

\begin{proof}
    Let us write $(\beta_i,\beta_{i+1})$ for the elementary interval of the second kind associated to $\tilde w$. Let us fix $j_0 \geq 0$ such that $w := \tilde w\{j_0\}$ satisfies $(\beta_i,\beta_{i+1}) = (\mathrm{Ker}D(w), \mathrm{Dom}D(w))$. Eventually, let us fix $j \geq 0$ and consider 
    \begin{equation*}
        w' := \tilde w[j]\{j'\} = w_{j+j'} w_{j+j'-1} \cdots w_{j+1},
    \end{equation*}
    where $j' \geq j_0$ is an integer that we will fix later. Note that we have $w'\tilde w\{j\} = \tilde w\{j+j'\}$. In particular, it implies that
    \begin{align}
        \label{Eqnw'Product1}\beta_i & = \mathrm{Ker}D(w\{j+j'\}) = D(\tilde w\{j\})^{-1}(\mathrm{Ker}D(w')),\\
        \label{Eqnw'Product2}\beta_{i+1} &= \mathrm{Dom}D(w\{j+j'\}) = D(\tilde w\{j\})^{-1}(\mathrm{Dom}D(w')).
    \end{align}
     If the interval $(\mathrm{Ker}D(w'), \mathrm{Dom}D(w'))$ is elementary, then it must be of the second kind. Indeed, if it were elementary of the fist kind, we could write $(\mathrm{Ker}D(w'), \mathrm{Dom}D(w')) = (\mathrm{Dom}D(V^{\#}v), \mathrm{Ker}D(Fv))$ for some word $v \in W_1$. But then \eqref{Eqnw'Product1} and \eqref{Eqnw'Product2} would imply that $\beta_i = \mathrm{Dom}D(V^{\#}vw\{j\})$ and $\beta_{i+1} = \mathrm{Ker}D(Fv\tilde w\{j\})$, which contradicts the fact that $(\beta_i,\beta_{i+1})$ is of the second kind.\\ 
     Thus, if $(\mathrm{Ker}D(w'), \mathrm{Dom}D(w'))$ is elementary, it is necessarily of the second kind and therefore it determines an infinite word $\tilde w' \in W_2$. Let us rewrite this interval $(\beta_k,\beta_{k+1})$. We claim that $\tilde w' = \tilde w[j]$. Following the proof of Proposition \ref{PropositionWordsOfSecondKind}, it is enough to check that for every $j'' \geq j'$, the monomial $D(w'')$ is non null where $w'' := \tilde w[j]\{j''\}$. But the equations \eqref{Eqnw'Product1} and \eqref{Eqnw'Product2} also hold with $w''$ instead of $w'$, so that $D(w'')$ being null would imply that $\beta_i = \beta_{i+1}$, which is absurd. Thus we conclude that $\tilde w' = \tilde w[j] \in W_2$ as desired.\\
     It remains to show that we can choose $j'$ so that the interval $(\mathrm{Ker}D(w'), \mathrm{Dom}D(w'))$ is elementary. Towards a contradiction, assume that we can not. For $s \geq 0$, let us define $w'_s := \tilde w [j]\{j_0+s\}$. We have $w_{s+1}' = w_{j+j_0+s+1}w_s'$ for all $s\geq 0$. It follows that 
     \begin{equation*}
         \mathrm{Ker}D(w'_s) \subseteq  \mathrm{Ker}D(w'_{s+1}) \subseteq  \mathrm{Dom} D(w'_{s+1}) \subseteq  \mathrm{Dom} D(w'_s).
     \end{equation*}
     Since $M$ is finite dimensional, there must be some $s_0 \geq 0$ such that for all $s \geq s_0$, we have $\mathrm{Ker}D(w'_s) = \mathrm{Ker}D(w'_{s_0})$ and $\mathrm{Dom} D(w'_s) = \mathrm{Dom} D(w'_{s_0})$. By assumption, this interval is not elementary. Therefore there is some $\beta_k \in \mathcal F_M$ such that, for every $s \geq s_0$, we have 
     \begin{equation*}
         \mathrm{Ker}D(w'_s) \subsetneq \beta_k \subsetneq \mathrm{Dom} D(w'_s).
     \end{equation*}
     Let $v \in W$ such that $\beta_k$ is equal to the kernel or the domain of $D(v)$. By Lemma \ref{LemmaRelativePositionIntervals}, the only possibility for such strict inclusions to hold is that $w'_s$ is a suffix of $v$. In particular $l(v) \geq l(w'_s) = j_0+s$. Since this holds for all $s\geq s_0$, we have reached a contradiction as desired.
\end{proof}

\begin{proposition}[cf.~\cite{gelfand-ponomarev:1968} Theorem 5.2]\label{WordsArePeriodic}
    Any word $\tilde w \in W_2$ of the second kind is periodic.
\end{proposition}

\begin{proof}
    Let $\tilde w \in W_2$ be a word of the second kind. By Lemma \ref{AllShiftsAreWordsOfSecondKind}, we know that $\tilde w[j] \in W_2$ for all $j \geq 0$. Since $W_2$ is a finite set, we must have $\tilde w[j] = \tilde w[j']$ for some $j \not = j'$. It follows that $\tilde w$ can be written in the form
    \begin{equation*}
        \tilde w = [w]v,
    \end{equation*}
    where $w,v \in W$ with $w$ non empty, and $[w]$ denotes the infinite periodic word $[w] := \cdots www$ as in Definition \ref{DefinitionInfiniteWords}. Assume that $v$ has minimal length among all the words $v$ for which such a decomposition of $\tilde w$ exists. We show that $v$ is empty, from which it would follow that $\tilde w$ is periodic.\\
    Towards a contradiction, assume that $v$ is not empty. Let us write $w = w_tw_{t-1}\cdots w_1$ and $v = v_sv_{s-1}\cdots v_1$ with $s\geq 1$. By minimality of $s$, we have $v_s \not = w_t$. Observe that
    \begin{align*}
        [w] = \tilde w[s], & & [w]v_s = \tilde w[s-1], & & [w]w_t = \tilde w[s+t-1].
    \end{align*}
    Thus, by Lemma \ref{AllShiftsAreWordsOfSecondKind}, these three infinite words are of the second kind. Let $j_0 \geq 0$ be such that for all $j \geq j_0$, the kernels and domains of $[w]\{j\}$, $([w]v_s)\{j\}$ and $([w]w_t)\{j\}$ define the associated intervals of the second kind, which we denote by $(\beta_i,\beta_{i+1})$, $(\beta_{k},\beta_{k+1})$ and $(\beta_{l},\beta_{l+1})$ respectively. Since $([w]w_t)\{j+1\} = [w]\{j\}w_t$ and $([w]v_s)\{j+1\} = [w]\{j\}v_s$, we have 
    \begin{align*}
        \beta_k & = D(v_s)^{-1}\beta_i, & \beta_{k+1} & = D(v_s)^{-1}\beta_{i+1},\\
        \beta_l & = D(w_t)^{-1}\beta_i, & \beta_{l+1} & = D(w_t)^{-1}\beta_{i+1}.
    \end{align*}
    Since $v_s \not = w_t$, one of them is $F$ while the other if $V^{\#}$. Without loss of generality, we may assume that $v_s = F$. Similar to Lemma \ref{Morphisms1stKind}, we have a diagram 
    \begin{center}
        \begin{tikzcd}
            \beta_{l+1}/\beta_l  & & \beta_{k+1}/\beta_k \\
            & \beta_{i+1}/\beta_i \arrow[ul, twoheadrightarrow, swap, "V"] \arrow[ur, hookleftarrow, "F"] &            
        \end{tikzcd}
    \end{center}
    and the composition of both arrows is zero. It follows that 
    \begin{equation}\label{InequalityDimensions}
        \dim \beta_{i+1}/\beta_i \geq \dim \beta_{l+1}/\beta_l + \dim \beta_{k+1}/\beta_k.
    \end{equation}
    If $t=1$ then $[w] = [w]w_t$, so that $(\beta_i,\beta_{i+1}) = (\beta_{l+1},\beta_l)$ and a contradiction readily follows. Thus, we may assume that $t\geq 2$. Recall that $[w]w_t = \tilde w[s+t-1]$ and $[w] = \tilde w[s]$. If $u := w_{s+t-1}\cdots w_{s+2}w_{s+1} \in W$, we can write 
    $\tilde w[s] = \tilde w[s+t-1]u$, so that by a similar reasoning as above we have 
    \begin{align*}
        \beta_i = D(u)^{-1}\beta_l, & & \beta_{i+1} = D(u)^{-1}\beta_{l+1}.
    \end{align*}
    Now, if $X\subseteq  Y \subseteq  M$ are two $K$-subspaces, then
    \begin{align*}
        \dim Y/X \geq \max(\dim F^{-1}Y/F^{-1}X, \dim VY/VX).
    \end{align*}
    In particular, we have 
    \begin{align*}
        \dim \beta_{l+1}/\dim \beta_l \geq \dim D(w_{s+t-1})^{-1}\beta_{l+1}/D(w_{s+t-1})^{-1}\beta_{l} \geq \cdots & \geq \dim D(u)^{-1}\beta_{l+1}/D(u)^{-1}\beta_l \\
        & = \dim \beta_{i+1}/\beta_i.
    \end{align*}
    Combined with \eqref{InequalityDimensions}, we obtain a contradiction and this concludes the proof.
    \end{proof}

\begin{corollary}\label{CorollaryAllInervalsSecondKindSameLength}
    Let $\tilde w \in W_2$ be a word of the second kind, and let $j \geq 0$. Let $(\beta_i,\beta_{i+1})$ and $(\beta_k,\beta_{k+1})$ denote the elementary intervals of the second kind determined by $\tilde w$ and by $\tilde w[j]$ respectively. Then we have
    \begin{equation*}
        \dim \beta_{i+1}/\beta_i = \dim \beta_{k+1}/\beta_k.
    \end{equation*}
\end{corollary}

\begin{proof}
    By Proposition \ref{WordsArePeriodic}, we know that $\tilde w = \tilde w[t]$ for some $t > 0$. For $0 \leq i \leq t$, let us write $(\beta_0(i),\beta_1(i))$ for the elementary interval of the second kind associated to $\tilde w[i]$. Just as in the end of the proof of Proposition \ref{WordsArePeriodic}, we have 
    \begin{equation*}
        \dim \beta_1(t)/\beta_0(t) \geq \dim \beta_1(t-1)/\beta_0(t-1) \geq \cdots \geq \dim \beta_1(0) / \dim \beta_0(0).
    \end{equation*}
    Since $\tilde w = \tilde w[0] = \tilde w[t]$, we have $(\beta_0(t),\beta_1(t)) = (\beta_0(0),\beta_1(0))$ so that all the inequalities above are in fact equalities.
\end{proof}

\begin{notation}
    If $\tilde w \in W_2$, we can write $\tilde w = [w]$ for some finite word $w\in W$. In accordance with Remark \ref{RemarkInfinitePeriodicWord}, when we write $[w] \in W_2$, the period of $[w]$ is always assumed to be $l(w) \geq 1$. For $0 \leq j < l(w)$, recall the notation 
    \begin{equation*}
        w(j) := w_j\cdots w_2w_1w_{l(w)}\cdots w_{j+2}w_{j+1}.
    \end{equation*}
\end{notation}

\begin{proposition}\label{EveryElementaryIntervalSecondKindIsStable}
    Let $[w]$ be an infinite periodic word. Then 
    \begin{equation*}
        [w] \in W_2 \iff \mathrm{Ker}D(w)^{\infty} \subsetneq \mathrm{Dom}D(w)^{\infty},      
    \end{equation*}
    where $\mathrm{Ker}D(w)^{\infty}$ and $\mathrm{Dom}D(w)^{\infty}$ denote respectively the stable kernel and stable domain of $D(w)$, as introduced in Definition \ref{DefStableSubspaces}. In this case, the elementary interval of the second kind $(\beta_i,\beta_{i+1})$ determined by $[w]$ is equal to $(\mathrm{Ker}D(w)^{\infty}, \mathrm{Dom}D(w)^{\infty})$.
\end{proposition}

\begin{proof}
    Let $(\beta_i,\beta_{i+1})$ be an interval of the second kind, determined by a word $[w] \in W_2$. From Proposition \ref{PropositionWordsOfSecondKind}, it is not difficult to see that $(\beta_i,\beta_{i+1})=(\mathrm{Ker}D(w)^{\infty}, \mathrm{Dom}D(w)^{\infty})$. Thus the interval on the RHS is non trivial. \\
    Conversely, assume that $\mathrm{Ker}D(w)^{\infty} \subsetneq \mathrm{Dom}D(w)^{\infty}$. We claim that this interval must be elementary, with a reasoning similar to the proof of Lemma \ref{AllShiftsAreWordsOfSecondKind}. For otherwise, there would be some $\beta_k \in \mathcal F_M$ such that 
    \begin{equation*}
        \mathrm{Ker}D(w)^{\infty} \subsetneq \beta_k \subsetneq \mathrm{Dom}D(w)^{\infty}.
    \end{equation*}
    Since $M$ is finite dimensional, there is some $n_0 \geq 0$ such that for all $n\geq n_0$, we have $\mathrm{Ker}D(w)^{\infty} = \mathrm{Ker}D(w^n)$ and $\mathrm{Dom}D(w)^{\infty} = \mathrm{Dom}D(w^n)$, where $w^n := w \cdots w$ with $w$ appearing $n$ times. Let us consider such an $n$. If $v$ is a word such that $\beta_k$ is equal to the kernel or domain of $v$, then $v$ must have $w^n$ as a suffix. As this holds for all $n \geq n_0$, this is absurd. \\
    It remains to prove that the elementary interval $(\mathrm{Ker}D(w)^{\infty}, \mathrm{Dom}D(w)^{\infty})$ can not be of the first kind. Towards a contradiction, assume that there is some $v \in W_1$ such that 
    \begin{align*}
        \mathrm{Ker}D(w)^{\infty} = \mathrm{Dom}D(V^{\#}v), & & \mathrm{Dom}D(w)^{\infty} = \mathrm{Ker}D(Fv).
    \end{align*}
    Let $n_0$ be as above. We have 
    \begin{equation*}
        \mathrm{Ker}D(w)^{\infty} = \mathrm{Ker}D(w^{n_0+1}) = D(w)^{-1} \mathrm{Ker}D(w^{n_0}) = D(w)^{-1}\mathrm{Ker}D(w)^{\infty},
    \end{equation*}
    from which it follows that 
    \begin{equation*}
        \mathrm{Dom}D(V^{\#}v) = D(w)^{-1}\mathrm{Dom}D(V^{\#}v) = \mathrm{Dom}D(V^{\#}vw).
    \end{equation*}
    Likewise, we have $\mathrm{Ker}D(Fv) = \mathrm{Ker}D(Fvw)$. By Proposition \ref{MonomialOfFirstKind}, we must have $v = vw$ which is absurd.
\end{proof}

\begin{proposition}\label{AllCircularKraftAreSecondKind}
    Let $\Gamma= (\mathcal V, \mathcal E, \mathrm{label})$ be a connected circular Kraft quiver with no repetition, together with a strict $(\sigma,\tau)$-linear representation $(U,\rho)$. The twisted Gelfand-Ponomarev module $M(\Gamma,U,\rho)$ is of the second kind. Moreover, if $\Gamma = \Gamma([w],t)$ where $t = l(w)$ as in Definition \ref{DefCircularKraftQuiverWithWord}, then 
    \begin{equation*}
        \Sigma_2 = \{[w(j)], 0 \leq j \leq t-1\}.
    \end{equation*}
\end{proposition}

\begin{proof}
    In this proof we consider $M = M(\Gamma,U,\rho)$ as in the statement. Observe that $\Sigma_1 = \emptyset$ if and only if the empty word does not belong to $\Sigma_1$. This is equivalent to the condition that $\mathrm{Ker}D(F) = \mathrm{Dom}D(V^{\#})$ on $M(\Gamma,U,\rho)$. Let $\mathcal V'$ denote the set of vertices $v\in \mathcal V$ which are not the tail of any $F$-arrow. Since $(U,\rho)$ is strict, we have
    \begin{equation*}
        \mathrm{Ker}D(F) = \bigoplus_{v \in \mathcal V'} U_v.
    \end{equation*}
    But since $\Gamma$ is circular, it is easy to see that $\mathcal V'$ is also the set of vertices which are the head of some $V$-arrow. Thus, since $(U,\rho)$ is strict again, we have 
    \begin{equation*}
        \mathrm{Dom}D(V^{\#}) = \bigoplus_{v \in \mathcal V'} U_v.
    \end{equation*}
    Thus $\mathrm{Ker}D(F) = \mathrm{Dom}D(V^{\#})$ so that $M$ is of the second kind. Now let $v_1,\ldots ,v_t$ denote the vertices of $\Gamma = \Gamma([w],t)$ as in Definition \ref{DefCircularKraftQuiverWithWord}. If $0 \leq j < t$, by construction we have 
    \begin{equation*}
        D(w(j))_{|U_{v_{j+1}}} = \Gamma_{\Phi_{v_{j+1}}},
    \end{equation*}
    where $\Phi_{v_{j+1}} : U_{v_{j+1}} \xrightarrow{\sim} U_{v_{j+1}}$ is the monodromy operator of $(U,\rho)$ at $v_{j+1}$ as defined in Definition \ref{DefinitionMonodromy}, and $\Gamma_{\Phi_{v_{j+1}}}$ is its graph. Since the monodromy operator is non-singular, we have $U_{v_{j+1}} \subseteq  \mathrm{Dom}D(w(j))^{\infty}$ and $U_{v_{j+1}} \cap \mathrm{Ker}(D(w(j))^{\infty} = \{0\}$. It follows from Proposition \ref{EveryElementaryIntervalSecondKindIsStable} that that $[w(j)] \in W_2$ and that the interval $(\mathrm{Ker}(D(w(j))^{\infty}, \mathrm{Dom}D(w(j))^{\infty})$ is elementary. Moreover, the quotient $        \mathrm{Dom}D(w(j))^{\infty} / \mathrm{Ker}(D(w(j))^{\infty}$ has dimension at least equal to $\dim U_{v_{j+1}}$. Since $M$ is equal to the direct sum of all the $\dim U_{v_i}$'s, these intervals for $0 \leq j < t$ already exhaust the whole stabilized sequence of $M$. This concludes the proof.    
\end{proof} 

Define an equivalence relation on $W_2$ as follows:
\begin{equation*}
    [w] \sim [v] \iff [v] = [w(j)] \text{ for some } j \geq 0.
\end{equation*}
This defines a partition of $W_2$, where each equivalence class has the form 
\begin{equation*}
    W_2(w) := \{[w],[w(1)],\ldots , [w(t-1)]\}, t:= l(w).
\end{equation*}
To such an equivalence class $W_2(w)$, we associate the connected circular Kraft quiver $\Gamma([w],t)$ as in Definition \ref{DefCircularKraftQuiverWithWord}. In particular, the vertices of $\Gamma([w],t)$ are denoted by $v_1,\ldots ,v_t$. Of course, up to isomorphism this Kraft quiver does not depend on the choice of the representative of the equivalence class. We define a strict $(\sigma,\tau)$-linear representation $(U^{[w]}, \rho^{[w]})$ on $\Gamma([w],t)$ as follows:
\begin{itemize}
    \item $U^{[w]}_{v_{i+1}} := \mathrm{Dom}(w(i))^{\infty} / \mathrm{Ker}(D(w(i))^{\infty}$ for all $0 \leq i < t$,\\
    \item $\rho_E^{[w]}$ is the natural isomorphism induced by $F$ or by $V$ depending on whether $E = (v_i,v_{i+1})$ or $E = (v_{i+1},v_i)$ respectively.
\end{itemize}
Here, $\rho_{E}^{[w]}$ is defined just as in the proof of Proposition \ref{WordsArePeriodic}, using the identity $[w(i)] = [w(i+1)]w_{i+1}$ for all $0 \leq i < t$.\\
By construction, the Kraft quivers $\Gamma([w],t)$ are pairwise non-isomorphic when $[w]$ runs along $W_2/\sim$. Let $\Gamma^{2\mathrm{nd}}$ denote the disjoint union of these Kraft quivers, and let $(U^{2\mathrm{nd}},\rho^{2\mathrm{nd}})$ denote the induced $(\sigma,\tau)$-linear representation on $\Gamma^{2\mathrm{nd}}$. We define 
\begin{equation*}
    \mathrm{gr}^{2\mathrm{nd}}(M) := M(\Gamma^{2\mathrm{nd}},U^{2\mathrm{nd}},\rho^{2\mathrm{nd}}).
\end{equation*}
By Proposition \ref{AllCircularKraftAreSecondKind}, $\mathrm{gr}^{2\mathrm{nd}}(M)$ is a twisted Gelfand-Ponomarev module of the second kind. We prove that there exists a submodule $M_2 \subseteq  M$ which is isomorphic to $\mathrm{gr}^{2\mathrm{nd}}(M)$. This is explained in \cite{gelfand-ponomarev:1968} Theorem 5.3. For $[w] \in W_2$, recall from Theorem \ref{WeakDecomposition1} that there exists a subspace $S_w \subseteq  \mathrm{Dom}D(w)^{\infty} \cap \mathrm{Im}D(w)^{\infty}$ which is a supplement to $\mathrm{Ker}D(w)^{\infty}$ in $\mathrm{Dom}D(w)^{\infty}$, such that $D(w)_{|S_w}$ is the graph of a $\xi(w)$-linear automorphism $S_w \xrightarrow{\sim} S_w$, and such that $D(w)$ decomposes as 
\begin{equation*}
    D(w)_{|\mathrm{Dom}D(w)^{\infty}} = D(w)_{|S_w} \oplus D(w)_{|\mathrm{Ker}D(w)^{\infty}}.
\end{equation*}
We claim that these subspaces can be chosen compatibly within every equivalence class in $W_2$.

\begin{lemma}[cf.~\cite{gelfand-ponomarev:1968} Theorem 5.3]\label{LemmaTechnicalSecondKind}
    Let $[w] \in W_2$ and let $t := l(w)$. Let us write $w = w_t \cdots w_1$. The subspaces $S_{w(j)}$ for $0 \leq j < t$ can be chosen so that 
    \begin{equation*}
        \begin{cases}
            F(S_{w(j)}) = S_{w(j+1)} & \text{if } w_{j+1} = F;\\
            V(S_{w(j+1)}) = S_{w(j)} & \text{if } w_{j+1} = V^{\#}.
        \end{cases}
    \end{equation*}
\end{lemma}

\begin{proof}
    We can choose $S_w$ arbitrarily, then determine suitable $S_{w(j)}$'s for $0 < j < t$ from our initial choice. Let $(e_i^w)_{1\leq i \leq d}$ be a basis of $S_w$, where $d := \dim S_w$. By construction, there exists unique vectors $f_i^w \in S_w$ such that $e_i^w \xrightarrow{D(w)} f_i^w$ for all $1\leq i \leq d$. Since $D(w) = D(w_t)\cdots D(w_1)$, by definition of the composition of relations, we can find vectors $e_i^{w(j)}$ for $1 \leq j \leq t-1$ such that for all $i$,
    \begin{equation*}
        e_i^w \xrightarrow{D(w_1)} e_1^{w(1)} \xrightarrow{D(w_2)} \cdots \xrightarrow{D(w_{t-1})} e_1^{w(t-1)} \xrightarrow{D(w_t)} f_i^{w}.
    \end{equation*}
    We define $S_{w(j)} := \mathrm{Span}_K\{e_i^{w(j)}, 1\leq i \leq d\}$. Let us show that these spaces $S_{w(j)}$ meet all the requirements. By construction, it is obvious that we have 
        \begin{equation*}
        \begin{cases}
            F(S_{w(j)}) = S_{w(j+1)} & \text{if } w_{j+1} = F;\\
            V(S_{w(j+1)}) = S_{w(j)} & \text{if } w_{j+1} = V^{\#}.
        \end{cases}
    \end{equation*}
    For $j = t-1$, this follows from the fact that $(f_i^w)_{1\leq i \leq d}$ is also a basis of $S_w$. Thus all we have to prove is that $S_{w(j)}$ satisfies the conditions (1), (2) and (3) in Theorem \ref{WeakDecomposition1}. We have 
    \begin{align*}
        [w(j)] = [w]w_t\cdots w_{j+1}, & & [w] = [w(j)]w_j\cdots w_1.
    \end{align*}
    In particular, we deduce that 
    \begin{align}
        \label{Useful4Identities1} \mathrm{Dom}D(w(j))^{\infty} & = D(w_t\cdots w_{j+1})^{-1}\mathrm{Dom}D(w)^{\infty}, \\ 
        \label{Useful4Identities2} \mathrm{Ker}D(w(j))^{\infty} & = D(w_t\cdots w_{j+1})^{-1}\mathrm{Ker}D(w)^{\infty},\\
        \label{Useful4Identities3} \mathrm{Dom}D(w)^{\infty} & = D(w_j\cdots w_1)^{-1}\mathrm{Dom}D(w(j))^{\infty}, \\ 
        \label{Useful4Identities4} \mathrm{Ker}D(w)^{\infty} & = D(w_j\cdots w_1)^{-1}\mathrm{Ker}D(w(j))^{\infty}.
    \end{align}
    Since we have $e_i^{w(j)} \xrightarrow{D(w_t\cdots w_{j+1})} f_i^w \in S_w \subseteq  \mathrm{Dom}D(w)^{\infty}$ for every $1\leq i \leq d$, it follows that $S_{w(j)} \subseteq  \mathrm{Dom}D(w(j))^{\infty}$. Let us now consider $x \in S_{w(j)} \cap \mathrm{Ker}D(w(j))^{\infty}$. There is some $y \in \mathrm{Ker}D(w)^{\infty}$ such that $x \xrightarrow{D(w_t\cdots w_{j+1})} y$. Thus, if we decompose $x = \sum \lambda_ie_i^{w(j)}$ for some scalars $\lambda_i$'s, we get 
    \begin{equation*}
        \sum_{i=1}^d (\xi(w)_j\cdots \xi(w)_1)^{-1}(\lambda_i)e_i^w \xrightarrow{D(w_j\cdots w_1)} x \xrightarrow{D(w_t\cdots w_{j+1})} y.
    \end{equation*}
    Since $S_w$ satisfies (2) of Theorem \ref{WeakDecomposition1}, we deduce that $\sum_{i=1}^d (\xi(w)_j\cdots \xi(w)_1)^{-1}(\lambda_i)e_i^w = 0$ so that all the $\lambda_i$'s are $0$. Thus $S_{w(j)} \cap \mathrm{Ker}D(w(j))^{\infty} = \{0\}$. By a similar reasoning, it is easy to prove that the vectors $e_{i}^{w(j)}$ are linearly independent. Thus $\dim S_{w(j)} = d$, which is also the the codimension of $\mathrm{Ker}D(w(j))^{\infty}$ in $\mathrm{Dom}D(w(j))^{\infty}$ according to Corollary \ref{CorollaryAllInervalsSecondKindSameLength}. Thus, we have 
    \begin{equation*}
        \mathrm{Dom}D(w(j))^{\infty} = S_{w(j)} \oplus \mathrm{Ker}D(w(j))^{\infty},
    \end{equation*}
    as desired. Before continuing, let us fix some notations. For $1 \leq i \leq d$, we can decompose $f_i^w = \sum_k \mu_{ik}e_k^w$ for some scalars $\mu_{ik}$. We define $f_i^{w(j)} := \sum_k \xi(w)_j\cdots \xi(w)_1(\mu_{ik})e_k^{w(j)} \in S_{w(j)}$. The vectors $(f_i^{w(j)})_{1\leq i\leq d}$ form another basis of $S_{w(j)}$. Moreover, we have $f_i^{w} \xrightarrow{D(w_j\cdots w_1)} f_i^{w(j)}$ for all $1 \leq j \leq t-1$.\\
    Let us now consider $x,y \in \mathrm{Dom}D(w(j))^{\infty}$ such that $x \xrightarrow{D(j)} y$. We can decompose $x = x_1 + x_2$ and $y = y_1 + y_2$ with $x_1,x_2 \in S_{w(j)}$ and $y_1,y_2 \in \mathrm{Ker}D(w(j))^{\infty}$. Let us write $y_1 = \sum \lambda_if_i^{w(j)}$ for some scalars $\lambda_i$'s. To simplify notations, we also write 
    \begin{equation*}
    z := \sum_{i=1}^d (\xi(w)_j\cdots \xi(w)_1\xi(w)_t\cdots \xi(w)_{j+1})^{-1}(\lambda_i) e_i^{w(j)}.
    \end{equation*}
    We have
    \begin{equation*}
         z \xrightarrow{D(w_t\cdots w_{j+1})} \sum_{i=1}^d (\xi(w)_j\cdots \xi(w)_1)^{-1}(\lambda_i) f_i^w \xrightarrow{D(w_j\cdots w_1)} y_1.
    \end{equation*}
    By taking the difference, we deduce that 
    \begin{equation*}
        x - z \xrightarrow{D(w(j))} y_2.
    \end{equation*}
    Thus $x - z \in D(w(j))^{-1} \mathrm{Ker}D(w(j))^{\infty} = \mathrm{Ker}D(w(j))^{\infty}$. It follows that $x_1 = z$, and therefore $x_1 \xrightarrow{D(w(j))} y_1$. This proves that
\begin{equation*}
    D(w(j))_{|\mathrm{Dom}D(w(j))^{\infty}} = D(w(j))_{|S_{w(j)}} \oplus D(w(j))_{|\mathrm{Ker}D(w(j))^{\infty}}.
\end{equation*}
    It remains to prove that $D(w(j))_{|S_{w(j)}}$ is the graph of a $\xi(w(j))$-linear automorphism of $S_{w(j)}$. By what precedes, we already saw that $D(w(j))_{|S_{w(j)}}$ has domain and image equal to $S_{w(j)}$. It only remains to prove that the kernel and the indeterminacy of $D(w(j))_{|S_{w(j)}}$ are zero. If $x \in S_{w(j)}$ is in the kernel, then in particular $x \in \mathrm{Ker}D(w(j))^{\infty}$, so that $x=0$. Let us now consider some $y \in S_{w(j)} \cap \mathrm{Indet}D(w(j))$. Let us write $y = \sum \lambda_ie_i^{w(j)}$ for some scalars $\lambda_i$'s. Since $0 \xrightarrow{D(w(j))} y$, we can find some $z \in \mathrm{Dom}D(w)^{\infty}$ such that 
    \begin{equation*}
        0 \xrightarrow{D(w_t\cdots w_{j+1})} z \xrightarrow{D(w_j\cdots w_1)} y.
    \end{equation*}
    In particular, $z \in \mathrm{Indet}D(w)$. Now, by taking the difference we also have
    \begin{equation*}
        \sum_{i=1}^d (\xi(w)_j\cdots \xi(w)_1)^{-1}(\lambda_i)e_i^w - z\xrightarrow{D(w_j\cdots w_1)} 0.
    \end{equation*}
    In particular, $(\xi(w)_j\cdots \xi(w)_1)^{-1} - z \in \mathrm{Ker}D(w)$. Let $N_w :=\mathrm{Ker}D(w)^{\infty} + \mathrm{Inder}D(w)^{\infty}$. We have just proved that $\sum_{i=1}^d (\xi(w)_j\cdots \xi(w)_1)^{-1}(\lambda_i)e_i^w \in S_w \cap N_w$. But this intersection is zero according to the stronger condition (1) of Theorem \ref{WeakDecomposition2}, which is also satisfied by $S_w$. Therefore, all the $\lambda_i$'s are zero, so that $\mathrm{Indet}D(w(j))_{|S_{w(j)}} = \{0\}$. This concludes the proof.    
\end{proof}

For any $[w] \in W_2/\sim$, the spaces $S_{w(j)}$ built in Lemma \ref{LemmaTechnicalSecondKind} form a $(\sigma,\tau)$-linear representation of $\Gamma([w],t)$ which we denote by $S^{[w]}$. Moreover, for each $0 \leq j < t$, the quotient map gives a $K$-linear isomorphism 
\begin{equation*}
    S_{w(j)} \xrightarrow{\sim} U^{[w]}_{w_{j+1}}.
\end{equation*}
This induces an isomorphism of $(\sigma,\tau)$-linear representation $S^{[w]} \xrightarrow{\sim} (U^{[w]},\rho^{[w]})$. Taking the union over all the $[w] \in W_2/\sim$, we obtain a $(\sigma,\tau)$-linear representation $S^{2\mathrm{nd}}$ of $\Gamma^{2\mathrm{nd}}$ which is isomorphic to $(U^{2\mathrm{nd}},\rho^{2\mathrm{nd}})$. The twisted Gelfand-Ponomarev module $M_2$ associated to $S^{2\mathrm{nd}}$ is naturally a submodule of $M$. Therefore, we have just proved the following. 

\begin{theorem}\label{TheoremExistenceM2}
    There is a twisted Gelfand-Ponomarev submodule $M_2$ of $M$ which is isomorphic to $\mathrm{gr}^{2\mathrm{st}}(M)$. 
\end{theorem}

\section{The classification of twisted Gelfand-Ponomarev modules} \label{Section4}

In this last section, we give a summary of the results we have been building up so far.

\begin{theorem}\label{TheoremDecompositionIntoM1M2}
    Let $M$ be a twisted Gelfand-Ponomarev module. There is a decomposition 
    \begin{equation*}
        M = M_1 \oplus M_2,
    \end{equation*}
    where $M_1$ and $M_2$ are submodules of the first and of the second kind respectively. Moreover, $M_1$ (resp.~$M_2$) decomposes further as a direct sum of submodules which are isomorphic to modules associated to strict $(\sigma,\tau)$-linear representations of linear (resp.~circular) connected Kraft quivers. 
\end{theorem}

\begin{proof}
    Let $M_1$ and $M_2$ be given by Theorems \ref{TheoremExistenceM1} and \ref{TheoremExistenceM2}. All we have to prove is the decomposition of $M$ into the direct sum of $M_1$ and $M_2$. Let us write $D^{1}$, $D^2$ and $D^{1\cap 2}$ for the homomorphisms of monoids, sending a word $w \in W$ into the associated monomial acting respectively on $M_1$, $M_2$ and on $M_1\cap M_2$. For every $w \in W$, we have 
    \begin{equation*}
        D^{1\cap 2}(w) = D^1(w)_{|M_1\cap M_2} = D^2(w)_{|M_1\cap M_2}.
    \end{equation*} 
    Since $M_2$ is of the second kind, for every $w \in W$ we have $\mathrm{Dom}D^2(V^{\#}w) = \mathrm{Ker}D^2(Fw)$, so that the same equality holds for $D^{1\cap 2}$. Likewise, since $M_1$ is of the first kind, for every $w \in W$ we have $\mathrm{Ker}D^1(w)^{\infty} = \mathrm{Dom}D^1(w)^{\infty}$, so that the same equality holds for $D^{1\cap 2}$. In other words, $M_1\cap M_2$ has no elementary interval, thus must be zero. Eventually, since we have 
    \begin{align*}
        \dim M & = \sum_{i=1}^s \dim \beta_{i+1}/\beta_i = \\
        &= \sum_{w \in W_1} \dim \mathrm{Ker}D(Fw)/\mathrm{Dom}D(V^{\#}w) + \sum_{[w]\in W_2} \dim \mathrm{Dom}D(w)^{\infty}/\mathrm{Ker}D(w)^{\infty}\\
        & = \dim M_1 + \dim M_2,
    \end{align*}
    we conclude that $M = M_1 \oplus M_2$.
\end{proof}

\begin{theorem}\label{AllTwistedGFModulesComeFromKraftQuiver}
    Let $M$ be a twisted Gelfand-Ponomarev module. There exists 
    \begin{itemize}
        \item a Kraft quiver $\Gamma$ whose connected components are pairwise non-isomorphic, and all of whose circular connected components have no repetitions, and
        \item a strict $(\sigma,\tau)$-linear representation $(U,\rho)$ on $\Gamma$,
    \end{itemize}
    such that $M \simeq M(\Gamma,U,\rho)$ as $K[F,V]_{\sigma,\tau}$-modules. If $\Gamma'$ and $(U',\rho')$ are another Kraft quiver and strict $(\sigma,\tau)$-linear representation meeting the same conditions, then there exists an isomorphism $\Gamma \simeq \Gamma'$ making $(U,\rho)$ isomorphic to $(U',\rho')$.
\end{theorem}

\begin{proof}
    Let $M = M_1 \oplus M_2$ as in Theorem \ref{TheoremDecompositionIntoM1M2}. Both $M_1$ and $M_2$ decompose as the direct sum of modules associated to connected Kraft quivers which are linear or circular respectively. We build $\Gamma$ by taking the disjoint union of all these Kraft quivers, together with the induced $(\sigma,\tau)$-linear representation $(U,\rho)$. \\
    Assume that $M \simeq M'$ are two isomorphic Gelfand-Ponomarev modules. The choice of an isomorphism $\varphi$ transforms the stabilized sequence $\mathcal F_M$ into $\mathcal F_{M'}$, thus preserves the sets $W_1$ and $W_2$ of words of the first and second kinds. By construction, for $i = 1,2$ the modules $M_i$ and $M_i'$ for $M$ and $M'$ can be constructed from the same Kraft quivers, and the isomorphism $\varphi$ defines an isomorphism between the associated $(\sigma,\tau)$-linear representations. This implies that the data $(\Gamma,U,\rho)$ is well-determined by $M$ up to isomorphism.
\end{proof}

\begin{corollary}
    Assume that $K$ is an algebraically closed field of positive characteristic $p$, that $\sigma$ is given by $\sigma(x) = x^p$ and that $\tau = \sigma^{-1}$. Let $M$ be a twisted Gelfand-Ponomarev module. There exists 
    \begin{itemize}
        \item a Kraft quiver $\Gamma$ whose connected components are pairwise non-isomorphic, and all of whose circular connected components have no repetitions, and
        \item a family of positive integers $(d_i)_{1\leq i \leq r}$,
    \end{itemize}
    such that $\Gamma = \Gamma_1 \sqcup \cdots \sqcup \Gamma_r$ has $r$ connected components, and 
    \begin{equation*}
        M \simeq \bigoplus_{i=1}^r M(\Gamma_i,\mathbf 1_{\Gamma_i})^{d_i},
    \end{equation*}
     as $K[F,V]_{\sigma,\tau}$-modules. The Kraft quiver $\Gamma$ is determined by $M$ up to isomorphism, and the family $(d_i)_{1\leq i\leq r}$ is unique.
\end{corollary}

\begin{proof}
    This follows from Theorem \ref{AllTwistedGFModulesComeFromKraftQuiver}, combined with Proposition \ref{ModuleAttachedToLinearKraftQuiver} and Corollary \ref{ModulesAttachedToCircularKraftQuiverAlgCl}.
\end{proof}

We can classify indecomposable twisted Gelfand-Ponomarev modules as well. 

\begin{lemma}\label{LemmaIndecomposableModules}
    Let $\Gamma$ be a connected Kraft quiver, which has no repetition if it is circular. Let $(U,\rho)$ be a strict $(\sigma,\tau)$-linear representation of $\Gamma$. Then $M(\Gamma,U,\rho)$ is indecomposable if and only if $(U,\rho)$ is indecomposable.
\end{lemma}

\begin{proof}
    If $(U,\rho)$ decomposes in a non-trivial sum of two $(\sigma,\tau)$-linear representations, then it is clear that $M(\Gamma,U,\rho)$ decomposes as a non-trivial sum as well. Thus, we may now assume that $(U,\rho)$ is indecomposable. Towards a contradiction, assume that $M := M(\Gamma,U,\rho)$ decomposes into a non-trivial direct sum $M = M' \oplus M''$. By Theorem \ref{AllTwistedGFModulesComeFromKraftQuiver}, there exists Kraft quivers $\Gamma'$ and $\Gamma''$, and $(\sigma,\tau)$-linear representations $(U',\rho')$ and $(U'',\rho'')$ such that $M' \simeq M(\Gamma',U',\rho')$ and $M'' \simeq M(\Gamma'',U'',\rho'')$. Define a Kraft quiver $\Gamma^{\mathrm{new}}$ with a $(\sigma,\tau)$-linear representation $(U^{\mathrm{new}},\rho^{\mathrm{new}})$ as follows. For each pair $\Gamma'_i \simeq \Gamma''_j$ of isomorphic connected components of $\Gamma'$ and of $\Gamma''$, $\Gamma^{\mathrm{new}}$ has a single connected component $\Gamma^{\mathrm{new}}_k$ isomorphic to $\Gamma'_i \simeq \Gamma''_j$, bearing a $(\sigma,\tau)$-linear representation which is the direct sum of $( U',\rho')$ and $(U'',\rho'')$ restricted to $\Gamma'_i$ and $\Gamma''_j$ respectively. Else, for each connected component of $\Gamma'$ (resp.~$\Gamma''$) which is not isomorphic to any connected component of $\Gamma''$ (resp.~$\Gamma'$), $\Gamma^{\mathrm{new}}$ has a single connected component isomorphic to it, bearing a $(\sigma,\tau)$-linear representation which is isomorphic to the restriction of $(U',\rho')$ (resp.~of $(U'',\rho'')$).\\
    The Kraft quiver $\Gamma^{\mathrm{new}}$ together with its $(\sigma,\tau)$-linear representation specified above meet the conditions of Theorem \ref{AllTwistedGFModulesComeFromKraftQuiver}. By unicity, $\Gamma$ must be isomorphic to $\Gamma^{\mathrm{new}}$. Since $\Gamma$ is connected, both $\Gamma'$ and $\Gamma''$ must be connected and isomorphic to each other. But then, $(U,\rho)$ must be isomorphic to $(U^{\mathrm{new}},\rho^{\mathrm{new}})$ which is the direct sum of $(U',\rho')$ and $(U'',\rho'')$. This is a contradiction.
\end{proof}

\begin{theorem}\label{TheoremClassificationIndecomposableGFModules}
    Let $M$ be an indecomposable twisted Gelfand-Ponomarev module. Then $M$ is either of the first kind or of the second kind. 
    \begin{itemize}
        \item If $M$ is of the first kind, there exists a connected linear Kraft quiver $\Gamma$, unique up to isomorphism, such that $M \simeq M(\Gamma,\mathbf 1_{\Gamma})$.
        \item If $M$ is of the second kind, there exists a connected circular Kraft quiver with no repetitions and a strict indecomposable $(\sigma,\tau)$-linear representation $(U,\rho)$ such that $M \simeq M(\Gamma,U,\rho)$. The data $(\Gamma,U,\rho)$ is unique up to isomorphism.
    \end{itemize}
\end{theorem}

\begin{proof}
    This follows easily from Theorem \ref{AllTwistedGFModulesComeFromKraftQuiver} and Lemma \ref{LemmaIndecomposableModules}.
\end{proof}

\section{Application} \label{Section5}

As an application of Theorem \ref{AllTwistedGFModulesComeFromKraftQuiver} on the classification of twisted Gelfand-Ponomarev modules, we give a new proof of Theorem 6.1 of \cite{kottwitzrapoport:2003}. First, let us recall a few notions from \cite{aryapoor:2024}. Given a field $K$ and an automorphism $\sigma \in \mathrm{Aut}(K)$, we denote by $K[T;\sigma]$ the ring of \textit{skew polynomials} over the pair $(K,\sigma)$. In other words, $K[T;\sigma]$ is the $K$-algebra generated by an indeterminate $T$ which satisfies 
\begin{equation*}
    \forall x \in K, \quad Tx = \sigma(x)T.
\end{equation*}
Given $x \in K$, there is a unique $K$-linear form $\mathrm{eval}_x:K[T;\sigma] \to K$ which satisfies 
\begin{equation*}
    \mathrm{eval}_x(T^n) := \begin{cases}
        1 & \text{if } n =0;\\
        x\sigma(x)\cdots \sigma^{n-1}(x) & \text{if } n>0.
    \end{cases}
\end{equation*}
If $P(T) \in K[T;\sigma]$, we write $P(x) := \mathrm{eval}_x(P(T))$. We say that $x \in K$ is a \textit{root} of $P(T)$ if $P(x) = 0$. 

\begin{definition}[\cite{aryapoor:2024} Section 4.1]
    The pair $(K,\sigma)$ is said to be \textit{$1$-algebraically closed} if every non-constant skew polynomial $P(T) \in K[T;\sigma]$ has a root in $K$.
\end{definition}

For instance, if $K$ is algebraically closed then $(K,\mathrm{id})$ is $1$-algebraically closed. If $K$ is algebraically closed of characteristic $p>0$, and $\sigma$ is a power (positive or negative) of the Frobenius morphism $x \mapsto x^p$, then $(K,\sigma)$ is $1$-algebraically closed by \cite{aryapoor:2024} Section 4.3. Besides, Theorem 4.1 of \cite{aryapoor:2024} states that for any pair $(K,\sigma)$, there exists a $1$-algebraically closed pair $(L,\tau)$ such that $K$ can be embedded in $L$ in such a way that $\tau$ induces $\sigma$ upon restriction. This provides many more abstract examples of such pairs. 

\begin{remark}
    As pointed out in \cite{aryapoor:2024} Proposition 4.3, if $(K,\sigma)$ is $1$-algebraically closed and $\sigma \not = \mathrm{id}$, then $\sigma$ has infinite order in $\mathrm{Aut}(K)$.
\end{remark}

\begin{proposition}\label{Equivalence1-algebraicallyclosed}
    Let $K$ be a field and let $\sigma \in \mathrm{Aut}(K)$. The following statements are equivalent:
    \begin{enumerate}
        \item the pair $(K,\sigma)$ is $1$-algebraically closed,\\
        \item any $\sigma$-linear automorphism $f$ of any non-zero $K$-vector space $V$ fixes a line, that is we have $f(\ell) = \ell$ for some $1$-dimensional subspace $\ell \subseteq V$.
    \end{enumerate}
\end{proposition}

\begin{proof}
    We prove $(1) \implies (2)$. Let $v \in V$ be any non-zero vector, and let $W \subseteq V$ be the subspace generated by the vectors $f^n(v)$ for $n \geq 0$, where $f^n := f \circ \cdots \circ f$. Clearly $f(W) \subseteq W$, and since $f$ is injective it determines an automorphism on $W$ upon restriction. Upon replacing $V$ by $W$, we may assume that $V$ has a basis $e_1,\ldots , e_n$ such that $f(e_i) = e_{i+1}$ for all $1 \leq i \leq n-1$, and 
    \begin{equation*}
        f(e_n) = \sum_{i=1}^n a_ie_i,
    \end{equation*}
    for some $a_1, \ldots , a_n \in K$. The surjectivity of $f$ forces $a_1 \not = 0$. We consider the equation 
    \begin{equation} \label{eq:eigenvalue}
        f(v) = xv,
    \end{equation}
    for some $v = \sum_{i=1}^n v_ie_i \in V$ with $v_1 = 1$, and $x \in K^{\times}$. It is equivalent to the system 
    \begin{equation} \label{eq:systemeigenvalue}
        \begin{cases}
            \sigma(v_n)a_1 = x, & \\
            \sigma(v_n)a_i + \sigma(v_{i-1}) = xv_i & \text{for all } 2 \leq i \leq n.
        \end{cases}
    \end{equation}
    By substitution, we get $xv_i = \frac{a_i}{a_1}x + \sigma(v_{i-1})$ for all $2 \leq i \leq n$. By induction, we deduce that 
    \begin{equation} \label{eq:conditionv_n}
        v_n = \frac{a_n}{a_1} + \sigma\left(\frac{a_{n-1}}{a_1}\right)\frac{1}{x} + \cdots + \sigma^{n-2}\left(\frac{a_2}{a_1}\right)\frac{1}{x\sigma(x)\cdots \sigma^{n-3}(x)}+\frac{1}{x\sigma(x)\cdots \sigma^{n-2}(x)}.
    \end{equation}
    Applying $\sigma$ to \eqref{eq:conditionv_n} and substituing $\sigma(v_n)$ with the first line of \eqref{eq:systemeigenvalue}, we deduce that 
    \begin{equation*}
        - \frac{1}{a_1} + \sigma\left(\frac{a_n}{a_1}\right)\frac{1}{x} + \sigma^2\left(\frac{a_{n-1}}{a_1}\right)\frac{1}{x\sigma(x)} + \cdots + \sigma^{n-1}\left(\frac{a_2}{a_1}\right)\frac{1}{x\sigma(x)\cdots \sigma^{n-2}(x)}+\frac{1}{x\sigma(x)\cdots \sigma^{n-1}(x)} = 0.
    \end{equation*}
    In other words, equation \eqref{eq:eigenvalue} is equivalent to $x^{-1}$ being a root of the skew polynomial $P_f(T) \in K[T;\sigma]$ given by 
    \begin{equation*}
        P_f(T) := -\frac{1}{a_1} + \sigma\left(\frac{a_n}{a_1}\right)T + \sigma^2\left(\frac{a_{n-1}}{a_1}\right)T^2 + \cdots + \sigma^{n-1}\left(\frac{a_2}{a_1}\right)T^{n-1}+T^n,
    \end{equation*}
    and the $v_i$'s are uniquely determined by $x$ and the system \eqref{eq:systemeigenvalue}. Since $(K,\sigma)$ is $1$-algebraically closed, $P_f(T)$ has a root in $K$, which necessarily is non-zero, and any $v$ solution of \eqref{eq:eigenvalue} determines a line which is fixed by $f$. \\
    
    We prove $(2) \implies (1)$. Fix a non-constant skew polynomial $P(T) \in K[T;\sigma]$. Up to dividing by the leading coefficient, we may assume that
    \begin{equation*}
        P(T) = b_1 + b_2T + \cdots + b_{n}T^{n-1} + T^n,
    \end{equation*}
    for some $b_1,\ldots ,b_n \in K$. If $b_1 = 0$ then $0$ is clearly a root of $P(T)$, thus we may assume that $b_1 \not = 0$. The system of equations 
    \begin{equation*}
        \begin{cases}
            b_1 = -\frac{1}{a_1}, & \\
            b_i = \sigma^{i-1}\left(\frac{a_{n+2-i}}{a_1}\right) & \text{for all } 2 \leq i \leq n,
        \end{cases}
    \end{equation*}
    has a unique solution $a_1, \ldots , a_n \in K$ with $a_1 \not = 0$. Consider the space $V := K^n$ with canonical basis $e_1,\ldots ,e_n$, and let $f: V \to V$ be the $\sigma$-linear automorphism determined by $f(e_i) = e_{i+1}$ for $1 \leq i \leq n-1$, and $f(e_n) = \sum_{i=1}^n a_ie_i$. By assumption, there exists some non-zero vector $v = \sum_{i=1}^n v_ie_i$ and some $x \in K^{\times}$ such that $f(v) = xv$. Up to rescaling, we may assume that $v_1=1$. Then the arguments of the previous implication show that $x$ is a root of $P(T)$.
\end{proof}

We consider the following situation. Let $K$ be a field and let $f \geq 1$. For all $i \in \mathbb Z/f\mathbb Z$, let $M_i$ be a non-zero finite dimensional $K$-vector space. Assume that the $M_i$'s are equidimensional, that is, we have $\dim M_i = m$ for all $i$ and some $m > 0$. Let $(\sigma_i,\tau_i)_{i \in \mathbb Z/f\mathbb Z}$ be a family of automorphisms of $K$, and let $(\varphi_i,\psi_i)_{i \in \mathbb Z/f\mathbb Z}$ be a collection of semilinear morphisms such that 
\begin{align*}
    \varphi_i:M_{i-1} \to M_{i}, & & \psi_i:M_i \to M_{i-1},
\end{align*}
with $\varphi_i$ being $\sigma_i$-linear and $\psi_i$ being $\tau_i$-linear. This situation is summed up in Figure \ref{Figure6}. Eventually, we impose the so-called ``Orpheus condition''\footnote{This name is coined in \cite{kottwitzrapoport:2003} with the following explanation: ``Whenever you turn back while traveling through this diagram you are killed'', p.~176.}, that is, 
\begin{equation*}
    \forall i \in \mathbb Z/f\mathbb Z, \quad \varphi_i \circ \psi_i = 0, \quad \psi_i\circ\varphi_i = 0.
\end{equation*}

\begin{figure} 
\centering
\begin{tikzcd}
    & M_0 \arrow[r,shift left,"\varphi_1"] \arrow[dl,shift left,"\psi_0"] & M_1 \arrow[l,shift left,"\psi_1"] \arrow[dr,shift left,"\varphi_2"] & \\
    M_{f-1} \arrow[ur,shift left,"\varphi_0"] \arrow[d,shift left,"\psi_{f-1}"] & & & M_2 \arrow[ul,shift left, "\psi_2"] \arrow[d,shift left,"\varphi_3"] \\
    \vdots \arrow[u,shift left,"\varphi_{f-1}"] & & & \vdots \arrow[u,shift left,"\psi_3"]
\end{tikzcd}
\caption{The configuration of Theorem \ref{MaintheoremC}.}
\label{Figure6}
\end{figure}

% The statement is the following.

\begin{theorem}\label{MaintheoremC}
%[\cite{kottwitzrapoport:2003} Theorem 6.1] \label{TheoremKottwitzRapoport}
Let $\mathcal M := (M_i,\sigma_i,\tau_i,\varphi_i,\psi_i)_{i\in \mathbb Z/f\mathbb Z}$ be an algebraic datum underlying Figure \ref{Figure6} as above.  Assume that the family $(\sigma_i,\tau_i)_{i \in \mathbb Z/f\mathbb Z}$ of automorphisms of $K$  satisfies both of the following conditions:
\begin{enumerate}[label=\Alph*]
    \item \label{Condition1-algebraicallyclosedA} the pair $(K,\Xi)$ is $1$-algebraically closed for all automorphisms $\Xi \in \mathrm{Aut}(K)$ of the form $\Xi = \sigma_{a}\xi_{a-1}\cdots\xi_{a+1}$ for some $a \in \mathbb Z/f\mathbb Z$, and where $\xi_{a-i} = \sigma_{a-i}$ or $\tau_{a-i}^{-1}$ for all $1 \leq i \leq f-1$;
    \item \label{Condition1-algebraicallyclosedB} the pair $(K,\Xi)$ is $1$-algebraically closed for all automorphisms $\Xi \in \mathrm{Aut}(K)$ of the form $\Xi = \tau_{a}\xi_{a+1} \cdots \xi_{a-1}$ for some $a \in \mathbb Z/f\mathbb Z$, and where $\xi_{a+i} = \tau_{a+i}$ or $\sigma_{a+i}^{-1}$ for all $1 \leq j \leq f-1$,
\end{enumerate}
     Then there exists a collection of lines $\ell_i \subseteq M_i$ such that 
    \begin{equation*}
        \forall i \in \mathbb Z/f\mathbb Z, \quad \varphi_i(\ell_{i-1}) \subseteq \ell_i, \quad \psi_{i}(\ell_i) \subseteq \ell_{i-1}.
    \end{equation*}
\end{theorem}

%Let us write $\mathcal M := (M_i,\sigma_i,\tau_i,\varphi_i,\psi_i)_{i\in \mathbb Z/f\mathbb Z}$ for the algebraic datum underlying Figure \ref{Figure6}, satisfying both conditions \ref{Condition1-algebraicallyclosedA} and \ref{Condition1-algebraicallyclosedB}, and the Orpheus condition. 

We say that an algebraic datum $\mathcal M$ \textit{has a solution} if there exists a collection of lines $(\ell_i)_{i \in \mathbb Z/f\mathbb Z}$ satisfying Theorem~\ref{MaintheoremC}. 

\begin{corollary}[\cite{kottwitzrapoport:2003} Theorem 6.1] \label{TheoremKottwitzRapoport}
Assume that $K$ is an algebraically closed field of characteristic $p$, and all the automorphisms $\sigma_i$ and $\tau_i$ are some power (positive, zero or negative) of the Frobenius morphism $x \mapsto x^p$. Then $\calM$ has a solution.   
\end{corollary}

%Our goal is to show that any such $\mathcal M$ has a solution. 

%\begin{remark}
%    One situation where both conditions \ref{Condition1-algebraicallyclosedA} and \ref{Condition1-algebraicallyclosedB} are naturally satisfied is when $K$ is algebraically closed of positive characteristic $p$, and all the automorphisms $\sigma_i$ and $\tau_i$ are some power (positive, zero or negative) of the Frobenius morphism $x \mapsto x^p$. This situation is precisely the one considered in \cite{kottwitzrapoport:2003} for applications towards the theory $F$-isocrystals. The notion of $1$-algebraically closedness defined in \cite{aryapoor:2024} gives a natural generalization of this particular setting.
% \end{remark}

\begin{remark}\label{RemarkTheorem6.1}
    The original proof of Corollary~\ref{TheoremKottwitzRapoport} in \cite{kottwitzrapoport:2003} is indirect, as it uses algebraic geometry and a ``density argument''. About the possibility of a direct proof, the authors state that already with the case $f=2$, ``a large number of case distinctions has to be made and this approach quickly gets out of hand for a larger number of vector spaces''. In 2005, Ringel managed to give a direct proof during a lecture at Bielefeld University; see \cite{ringellecture:2005}. His proof was based on an explicit classification of the algebraic datum $\mathcal M$, which can be understood as a module over a certain semilinear string algebra, generalizing the algebra $K[F,V]_{\sigma,\tau}$ studied in the previous sections of this paper. However, Ringel's proof has apparently never been published, and no public lecture notes remain. The classification result used by Ringel can be recovered as a special case of the main theorem of \cite{clannishalgebras:2024}, applied to the quiver described by the diagram of Figure \ref{Figure6}, together with the Orpheus relation. \\
    In what follows, we give a direct and elementary proof, using only the classification of twisted Gelfand-Ponomarev modules, and which should be somewhat close in flavor to Ringel's proof.
\end{remark}

\begin{lemma}\label{LemmaCasef=1}
    Assume that $f=1$. A solution exists for any datum $\mathcal M$.
\end{lemma}

\begin{proof}
    When $f=1$, a datum $\mathcal M$ is equivalent to the datum of a $K$-vector space $M$ of dimension $m>0$, together with a $\sigma$-linear endomorphism $\varphi$ and a $\tau$-linear endomorphism $\psi$, such that $\varphi\psi = \psi\varphi = 0$, and the pairs $(K,\sigma)$ and $(K,\tau)$ are both $1$-algebraically closed. In particular, $M$ is a twisted Gelfand-Ponomarev module over the algebra $K[F,V]_{\sigma,\tau}$, where $F$ acts via $\varphi$ and $V$ via $\psi$. A solution for $\mathcal M$ consists in a line $\ell \subseteq M$ stable under both $\varphi$ and $\psi$. Let $\Gamma$ and $(U,\rho)$ be the Kraft quiver and the $(\sigma,\tau)$-linear representation of $\Gamma$ such that $M \simeq M(\Gamma,U,\rho)$ as in Theorem \ref{AllTwistedGFModulesComeFromKraftQuiver}.\\

    \textbf{Case 1:} We have $\mathrm{Ker}(\varphi) \cap \mathrm{Ker}(\psi) \not = \{0\}$.\\ 
    In terms of the Kraft quiver $\Gamma$, this condition means that there is some vertex $v$ which is not the tail of any arrow. This happens precisely when $\Gamma$ contains a linear connected component, as illustrated in Figure \ref{Figure7}, or when $\Gamma$ contains a circular connected component with no repetition and containing at least $2$ vertices, in which case we can find an $F$-arrow followed by a $V$-arrow.\\
    In this situation, any line $\ell$ within the intersection of both kernels is trivially a solution to $\mathcal M$.\\

    \textbf{Case 2:} We have 
    \begin{equation}\label{KernelsNotIntersect}
        \mathrm{Ker}(\varphi) \cap \mathrm{Ker}(\psi) = \{0\}.
    \end{equation}
    This condition implies that $\Gamma$ can not contain any linear connected component, nor any circular connected component with no repetition and containing $2$ vertices or more. By elimination, $\Gamma$ must consist of a single vertex with a loop, so that in fact one of $\varphi$ and $\psi$ vanishes while the other is an isomorphism. We will not need this fact though.\\
    At least one of $\mathrm{Ker}(\varphi)$ or $\mathrm{Ker}(\psi)$ must be non-zero, for otherwise both $\varphi$ and $\psi$ would be isomorphisms, which is incompatible with the fact that $\varphi\psi = \psi\varphi = 0$. Without loss of generality, let us assume that $\mathrm{Ker}(\psi) \not = \{0\}$, and pick any non-zero vector $x$ in this kernel. Let $N \subseteq M$ be the $K$-linear subspace generated by the vectors $\varphi^k(x)$ for $k \geq 0$ (where $\varphi^k$ is $\varphi$ composed $k$-times with itself). Clearly, we have $N \subseteq \mathrm{Ker}(\psi)$ and $\varphi(N) \subseteq N$. By \eqref{KernelsNotIntersect}, the restriction of $\varphi$ to $N$ is injective, so that $\varphi$ induces a $\sigma$-linear automorphism of $N$. Since $(K,\sigma)$ is $1$-algebraically closed, by Proposition \ref{Equivalence1-algebraicallyclosed} there is a line $\ell \subseteq N$ such that $\varphi(\ell) = \ell$. Then clearly $\ell$ is a solution to $\mathcal M$.
\end{proof}

\begin{figure} 
\centering
\begin{tikzpicture}[
       decoration = {markings,
                     mark=at position .5 with {\arrow{Stealth[length=2mm]}}},
       dot/.style = {circle, fill, inner sep=2.4pt, node contents={},
                     label=#1},
every edge/.style = {draw, postaction=decorate}
                        ]

\node (e1) at (0,2) [dot,label=above:$v_1$];
\node (e2) at (2,2) [dot,label=above:$v_2$,label=right:$\cdots$];
\node (e3) at (6,2) [dot,label=above:$v_m$,label=left:$\cdots$];
\node (e4) at (8,2) [dot,label=above:$v_{m+1}$];
\node (e5) at (1,0) [dot,label=above:$v_1$];
\node (e6) at (5,0) [dot,label=above:$v_{i}$,label=left:$\cdots$];
\node (e7) at (7,0) [dot,label=above:$v_{i+1}$];
\node (e8) at (9,0) [dot,label=above:$v_{i+2}$,label=right:$\cdots$];

\path (e2) edge node[above] {$V$} (e1);
\path (e3) edge node[above] {$F$} (e4);
\path (e6) edge node[above] {$F$} (e7);
\path (e8) edge node[above] {$V$} (e7);

\end{tikzpicture}
\caption{In any connected linear Kraft quiver $\Gamma(w)$ where $w = w_m\cdots w_1$, there is always some vertex which is not the tail of any arrow (top left: if $w_1 = V^{\#}$; top right: if $w_m = F$; bottom left: if $w = \emptyset$; bottom right: if $w_i = F$ and $w_{i+1} = V^{\#}$ for some $i$).}
\label{Figure7}
\end{figure}

\begin{proof}[Proof of Theorem~\ref{MaintheoremC}.]
    We use induction on $f \geq 1$, the case $f = 1$ being already checked in Lemma \ref{LemmaCasef=1}. Thus, we may now assume that $f \geq 2$, and that any datum with $f-1$ vector spaces has a solution. Let us fix a datum $\mathcal M$ with $f$ vector spaces.\\
    First, assume that $\psi_{i_0}$ is an isomorphism for some $0 \leq i_0 \leq f-1$. In this case, $\varphi_{i_0} = 0$ by the Orpheus condition. By identifying $M_{i_0-1} \simeq M_{i_0}$ via $\psi_{i_0}$, we are reduced to a datum $\mathcal M'$ consisting of $f-1$ vector spaces, and such that any solution for $\mathcal M$ corresponds to a solution for $\mathcal M'$. Explicitly, up to renumbering we may assume that $i_0 = f-1$. We define $M_i' := M_i$ for $0 \leq i \leq f-2$. We also define 
    \begin{align*}
        \sigma_i' & := \begin{cases}
            \sigma_i & \text{if } 1 \leq i \leq f-2;\\
            \sigma_{0}\tau^{-1}_{f-1} & \text{if } i = 0;
        \end{cases}
        & \tau_i' & := \begin{cases}
            \tau_i & \text{if } 1 \leq i \leq f-2;\\
            \tau_{f-1}\tau_0 & \text{if } i=0;\\
        \end{cases}
        \\
        \varphi_i' & := \begin{cases}
            \varphi_i & \text{if } 1 \leq i \leq f-2;\\
            \varphi_0\circ \psi_{f-1}^{-1} & \text{if } i=0;
        \end{cases}
        & \psi_i' & := \begin{cases}
            \psi_i & \text{if } 1 \leq i \leq f-2;\\
            \psi_{f-1}\circ\psi_0 & \text{if } i=0.
        \end{cases}
    \end{align*}
    Clearly, $\varphi_i'$ (resp. $\psi_i'$) is a $\sigma_i'$-linear (resp. $\tau'_i$-linear) morphism, they satisfy the Orpheus condition, and the collection $(\sigma_i',\tau_i')_{i \in \mathbb Z/(f-1)\mathbb Z}$ satisfies the conditions \ref{Condition1-algebraicallyclosedA} and \ref{Condition1-algebraicallyclosedB}. It determines a datum $\mathcal M'$ for which a solution $(\ell_j)_{j \in \mathbb Z/(f-1)\mathbb Z}$ exists by the induction hypothesis. Setting $\ell_{f-1} := \psi_{f-1}^{-1}(\ell_{f-2})$, we obtain a solution for the original datum $\mathcal M$.\\
    Thus, from now on we can assume that 
    \begin{equation}\label{ConditionKernelNonTrivial}
        \forall i \in \mathbb Z/f\mathbb Z, \quad \mathrm{Ker}(\psi_i) \not = \{0\}.
    \end{equation}
   Under the assumption \eqref{ConditionKernelNonTrivial}, we are going to build a solution $(\ell_j)_{j \in \mathbb Z/f\mathbb Z}$ to $\mathcal M$ for which $\psi_{j}(\ell_j) = \{0\}$ for all $j$.\\
    
    For every $i \in \mathbb Z/f\mathbb Z$, we define 
    \begin{align*}
        \Phi_i := \varphi_i \circ \varphi_{i-1} \circ \cdots \circ \varphi_{i+2} \circ\varphi_{i+1}, & & \Psi_i := \psi_{i+1} \circ \psi_{i+2} \circ \cdots \circ \psi_{i-1} \circ \psi_i.
    \end{align*}
    Similarly, we define $\sigma_{[i]} := \sigma_i \sigma_{i-1} \cdots \sigma_{i+2}\sigma_{i+1}$ and $\tau_{[i]} := \tau_{i+1} \tau_{i+2} \cdots \tau_{i-1}\tau_i$, so that $\Phi_i$ and $\Psi_i$ are respectively $\sigma_{[i]}$-linear and $\tau_{[i]}$-linear endomorphisms of $M_i$. Clearly, we have $\Phi_i\Psi_i = \Psi_i\Phi_i = 0$, so that $M_i$ has the structure of a twisted Gelfand-Ponomarev module over $K[F,V]_{\sigma_{[i]},\tau_{[i]}}$. Let $\Gamma_i$ and $(U^i,\rho^i)$ be the Kraft quiver and the $(\sigma_{[i]},\tau_{[i]})$-linear representation of $\Gamma_i$ such that $M_i \simeq M(\Gamma_i,U^i,\rho^i)$ as in Theorem \ref{AllTwistedGFModulesComeFromKraftQuiver}.\\

    \textbf{Case 1:} There exists some $i \in \mathbb Z/f\mathbb Z$ and some $x \in \mathrm{Ker}(\psi_i)$ such that $\Phi_i^k(x) \not = 0$ for all $k \geq 0$ (where $\Phi_i^k$ is $\Phi_i$ composed $k$-times with itself). \\
    This condition implies that $\Gamma_i$ has some connected component of the form $\Gamma([F],1)$, that is,  a single vertex $v$ together with an $F$-loop. Let $N := U^i_v \subseteq M_i$. We have $\Psi_i(N) = \{0\}$ (since $v$ is not the tail of any $V$-arrow), and $\Phi_i$ induces a $\sigma_{[i]}$-linear automorphism of $N$. By condition \ref{Condition1-algebraicallyclosedA}, the pair $(K,\sigma_{[i]})$ is $1$-algebraically closed, so that by Proposition \ref{Equivalence1-algebraicallyclosed} we can find some line $\ell_i \subseteq N$ such that $\Phi_i(\ell_i) = \ell_i$. For all $0 \leq r \leq f-2$, define 
    \begin{equation*}
        \ell_{i+1+r} := \varphi_{i+1+r} \circ \cdots \circ \varphi_{i+1}(\ell_i).
    \end{equation*}
    Then clearly $\dim(\ell_j) = 1$ for all $j \in \mathbb Z/f\mathbb Z$, we have $\varphi_{j}(\ell_{j-1}) = \ell_{j}$ and $\psi_j(\ell_j) =\{0\}$. It follows that $(\ell_j)_j$ is a solution for $\mathcal M$.\\

    \textbf{Case 2:} There exists some $i \in \mathbb Z/f\mathbb Z$ and some $x \in \mathrm{Ker}(\psi_i)$ such that $\Phi_i(x) \not = 0$ and $\Phi_i^k(x) = 0$ for some $k>1$. \\
    This condition implies that $\Gamma$ contains some vertex $v$ which is not the tail of any $V$-arrow, but is the tail of an $F$-arrow whose head is another vertex $v' \not = v$. Let $x$ be as above and assume that $k \geq 2$ is the smallest integer such that $\Phi_i^k(x) = 0$. Let us write $y := \Phi^{k-1}_i(x)$ and $z := \Phi_i^{k-2}(x)$, so that $y = \Phi_i(z)$ and $\Phi_i(y) = 0$. Let $0 \leq r \leq f-1$ denote the smallest integer such that 
    \begin{equation*}
        \varphi_{i+1+r} \circ \cdots \circ \varphi_{i+1}(y) = 0.
    \end{equation*}
    The fact that $r \leq f-1$ follows from $\Phi_i(y) = 0$. We define $\ell_i := K \cdot y \subseteq M_i$, and for every $0 \leq r' \leq f-2$ we define 
    \begin{equation*}
        \ell_{i+1+r'} := \begin{cases}
            K \cdot \varphi_{i+1+r'} \circ \cdots \circ \varphi_{i+1}(y) & \text{if } 0 \leq r' < r;\\
            K \cdot \varphi_{i+1+r'} \circ \cdots \circ \varphi_{i+1}(z) & \text{if } r \leq r' \leq f-2.
        \end{cases}
    \end{equation*}
    Clearly, we have $\dim(\ell_j) = 1$ for all $j \in \mathbb Z/f\mathbb Z$. We claim that $(\ell_j)_j$ is a solution for $\mathcal M$. Indeed, we have $\varphi_{i+1+r'}(\ell_{i+r'}) = \ell_{i+1+r'}$ for all $0 \leq r' < r$ and all $r < r' \leq f-1$ (for $r' = f-1$, this is due to $y = \Phi_i(z)$). Besides $\varphi_{i+1+r}(\ell_{i+r}) = \{0\}$. On the other hand, we have $\psi_j(\ell_j) = \{0\}$ since $\ell_j \subseteq \mathrm{Im}(\varphi_j)$ for all $j \in \mathbb Z/f\mathbb Z$.\\

    \textbf{Case 3:} For every $i \in \mathbb Z/f\mathbb Z$, we have 
    \begin{equation}\label{ConditionPhiVanishesOnKernel}
        (\Phi_i)_{|\mathrm{Ker}(\psi_i)} = 0.
    \end{equation}
    Let us pick any $i_0 \in \mathbb Z/f\mathbb Z$ and any non-zero vector $x_0 \in \mathrm{Ker}(\psi_{i_0}) \subseteq M_{i_0}$, which we are allowed to do by \eqref{ConditionKernelNonTrivial}. Let $0 \leq r_0 \leq f-1$ be the smallest integer such that 
    \begin{equation*}
        \varphi_{i_0+1+r_0} \circ \cdots \circ \varphi_{i_0+1}(x_0) = 0.
    \end{equation*}
    The fact that $r_0 \leq f-1$ follows from \eqref{ConditionPhiVanishesOnKernel}. If $r_0 = f-1$, we can define $\ell_{i_0} = K\cdot x_0$, and $\ell_{i_0+1+r} = K\cdot \varphi_{i_0+1+r} \circ \cdots \circ \varphi_{i_0+1}(x_0)$ for all $0 \leq r \leq f-2$. Clearly we have $\dim(\ell_j) = 1$ for all $j$, and $\varphi_{j}(\ell_{j-1}) = \ell_j$ except for $j = i_0$ in which case we have $\varphi_{i_0}(\ell_{i_0-1}) = \{0\}$. Besides, we have $\psi_j(\ell_j) = \{0\}$ for all $j$, so that $(\ell_j)_j$ is a solution for $\mathcal M$.\\
    Thus, we may assume that $r_0 < f-1$. We pick some non-zero vector $x_1 \in \mathrm{Ker}(\psi_{i_0+(1+r_0)}) \subseteq M_{i_0+(1+r_0)}$. Let $0 \leq r_1 \leq f-1$ be the smallest integer such that 
    \begin{equation*}
        \varphi_{i_0+(1+r_0)+(1+r_1)} \circ \cdots \circ \varphi_{i_0+(1+r_0)+1}(x_1) = 0.
    \end{equation*}
    If $(1+r_0)+(1+r_1) \geq f$, we stop here. Otherwise, we continue by picking some non-zero vector $x_2 \in \mathrm{Ker}(\psi_{i_0+(1+r_0)+(1+r_1)})$, and repeat the construction. \\
    In doing so, we have built two finite sequences $(x_i)_{0\leq i \leq k}$ and $(r_i)_{0 \leq i \leq k}$ for some $k\geq 1$ such that the $r_i$'s are integers between $0$ and $f-1$, and if we write $s_0 = 0$ and $s_i := \sum_{j=0}^{i-1}(1+r_j)$ for $1 \leq i \leq k+1$, then $x_i$ is a non-zero vector in $\mathrm{Ker}(\psi_{i_0+s_i})$ and 
    \begin{equation*}
        \varphi_{i_0+s_{i+1}}\circ \cdots \circ \varphi_{i_0+s_{i}+1}(x_{i}) = 0
    \end{equation*}
    for all $0 \leq i \leq k$. Moreover, we assume that $s_{k+1} \geq f$ but $s_{k} < f$. Since $r_{k} \leq f-1$, we have $0 \leq s_{k+1}-f \leq s_k-1$. Thus, there is a unique $0 \leq t \leq k-1$ such that 
    \begin{equation*}
        s_t \leq s_{k+1}-f < s_{t+1}.
    \end{equation*}
    Let us write $a := s_{k+1}-f$. We define 
    \begin{equation*}
        \ell_{i_0+s_t+r} := K\cdot \varphi_{i_0+s_t+r} \circ \cdots \circ \varphi_{i_0+s_t+1}(x_t),
    \end{equation*}
    for all $a - s_t \leq r \leq s_{t+1}-s_t-1$. Then, for $t+1 \leq t' \leq k$, we define 
    \begin{equation*}
        \ell_{i_0+s_{t'}+r} := K \cdot \varphi_{i_0+s_{t'}+r} \circ \cdots \circ \varphi_{i_0+s_{t'}+1}(x_{t'}),
    \end{equation*}
    for all $0 \leq r \leq s_{t'+1}-s_{t'}-1$. Since $i_0 + s_{k+1}-1 = i_0 + a - 1$ modulo $f$, we have successfully defined a sequence of lines 
    \begin{equation*}
        \ell_{i_0+a}, \ell_{i_0+a+1}, \cdots , \ell_{i_0+a-1},
    \end{equation*}
    such that $\ell_j \subseteq \mathrm{Ker}(\psi_j)$ for all $j$. Moreover, we have 
    \begin{equation*}
        \varphi_{j+1}(\ell_j) = \begin{cases}
            \{0\} & \text{if } j = i_0 + s_{t'+1}-1\mod f \text{ for some } t \leq t' \leq k;\\
            \ell_{j+1} & \text{else}.
        \end{cases}
    \end{equation*}
    It follows that $(\ell_j)_j$ is a solution for $\mathcal M$. \\

    Since cases 1, 2 and 3 are mutually exclusive and cover all the possibilities, the proof is over.\end{proof}

    \begin{remark}
        The principle of the proof in Case 3 is simple, even though the notations might hinder the intuition. Let us rephrase the idea in an informal tone. Starting with a non-zero vector $x_0$ in some $M_{i_0}$ which is mapped to $0$ in the $\psi$-direction, we take its successive images in the $\varphi$-direction until we reach $0$. By assumption, we are guaranteed to reach $0$ before completing a full turn in the diagram of Figure \ref{Figure6}. When we reach $0$, we pick another non-zero vector $x_1$ which is mapped to $0$ in the $\psi$-direction, and we take its successive images in the $\varphi$-direction until we reach $0$ again. We keep going until we complete a full turn, namely the vector $x_k$ is the first to reach $0$ after passing the starting point $i_0$. We let $i_0+a$ be the position where $x_k$ reaches $0$, so that $0 \leq a \leq f-1$. By construction, the vector space $M_{i_0+a}$ contains some non-zero vector $x_t$ ($t \leq k-1$), or a non-zero translate of $x_t$ in the $\varphi$-direction. This defines a line $\ell_{i_0+a}$, and the subsequent lines are generated by the corresponding non-zero translates of $x_{t'}$ (for some $t \leq t' \leq k$) in the $\varphi$-direction. 
    \end{remark}

%\bibliographystyle{amsplain}
%\bibliography{BibTeX}
\providecommand{\bysame}{\leavevmode\hbox to3em{\hrulefill}\thinspace}
\providecommand{\MR}{\relax\ifhmode\unskip\space\fi MR }
\providecommand{\MRhref}[2]{%
  \href{http://www.ams.org/mathscinet-getitem?mr=#1}{#2}
}
\providecommand{\href}[2]{#2}

\end{document}